\documentclass[11pt]{article}

\usepackage{amsmath}
\usepackage{amssymb}
\usepackage{amsfonts}
\usepackage{amsthm}
\usepackage{array}
\usepackage{bm}
\usepackage[shortlabels]{enumitem}
\usepackage{stmaryrd}

\usepackage{amsxtra}
\usepackage{array}
\usepackage{mathrsfs}
\usepackage{slashed}
\usepackage{xcolor}
\usepackage{tikz-cd}
\usepackage{tkz-euclide} 
\usepackage{pgfplots}

\newcolumntype{C}[1]{>{\centering\hspace{0pt}}p{#1}}
\usepackage[margin=1in]{geometry}

\newcommand{\GL}{\mathrm{GL}}

\newcommand{\SO}{\mathrm{SO}}
\newcommand{\Spin}{\mathrm{Spin}}
\newcommand{\U}{\mathrm{U}}
\newcommand{\SU}{\mathrm{SU}}
\newcommand{\Sp}{\mathrm{Sp}}

\newcommand{\G}{\mathrm{G}}

\newcommand{\Ker}{\mathrm{Ker}}

\newcommand{\End}{\mathrm{End}}

\newcommand{\Sym}{\mathrm{Sym}}

\newcommand{\Cone}{\mathrm{C}}

\newcommand{\Z}{\mathbb{Z}}

\newcommand{\R}{\mathbb{R}}
\newcommand{\C}{\mathbb{C}}
\newcommand{\HH}{\mathbb{H}}

\newcommand{\CP}{\mathbb{CP}}

\newcommand{\vol}{\mathrm{vol}}
\newcommand{\Hol}{\mathrm{Hol}}

\newcommand{\LB}{[\![}
\newcommand{\RB}{]\!]}

\newtheorem{thm}{Theorem}[section]
\newtheorem{prop}[thm]{Proposition}
\newtheorem{lem}[thm]{Lemma}
\newtheorem{cor}[thm]{Corollary}

\theoremstyle{definition}
\newtheorem{defn}[thm]{Definition}
\newtheorem{example}[thm]{Example}
\newtheorem{rmk}[thm]{Remark}
\newtheorem{notat}[thm]{Notation}

\usepackage[nottoc]{tocbibind}
\numberwithin{equation}{section}

\usepackage{hyperref}
\hypersetup{colorlinks, linkcolor=blue, citecolor=blue}

\title{On $\Sp(n)$-Instantons and the \\ Fourier-Mukai Transform of Complex Lagrangians}
\author{Jesse Madnick, Emily Autumn Windes}
\date{July 2024}

\newcommand{\Addresses}
{{  \bigskip
\noindent	\textsc{Seton Hall University} \par\nopagebreak
\noindent	\textsc{South Orange, NJ, United States}\par\nopagebreak
\noindent	\texttt{jesse.ochs.madnick@gmail.com} \\

\medskip
\noindent	\textsc{University of Oregon} \par\nopagebreak
\noindent	\textsc{Eugene, OR, United States} \par\nopagebreak
\noindent	\texttt{ewindes@uoregon.edu} \\

}}

\begin{document}

\maketitle

\begin{abstract}
The real Fourier-Mukai (RFM) transform relates calibrated graphs to so-called ``deformed instantons" on Hermitian line bundles.  We show that under the RFM transform, complex Lagrangian graphs in $\R^{2n} \times T^{2n}$ correspond to $\Sp(n)$-instantons over $\R^{2n} \times (T^{2n})^*$.  In other words, the deformed $\Sp(n)$-instanton equation coincides with the usual $\Sp(n)$-instanton equation. \\
\indent Motivated by this observation, we study $\Sp(n)$-instantons on hyperk\"{a}hler manifolds $X^{4n}$, with an emphasis on conical singularities.  First, when $X = C(M)$ is a hyperk\"{a}hler cone, we relate $\Sp(n)$-instantons on $X$ to tri-contact instantons on the $3$-Sasakian link $M$ and consider various dimensional reductions.  Second, when $X$ is an asymptotically conical (AC) hyperk\"{a}hler manifold of rate $\nu \leq -\frac{2}{3}(2n+1)$, we prove a Lewis-type theorem to the following effect: If the set of AC $\Sp(n)$-instantons is non-empty, then every AC Hermitian Yang-Mills connection over $X$ with sufficiently fast decay at infinity is an $\Sp(n)$-instanton.
\end{abstract}

 \tableofcontents
 
\section{Introduction}

\subsection{Background and Motivation}

\indent \indent Gauge-theoretic equations are a potent tool in geometry and topology.  As such, mathematicians and physicists have begun considering certain gauge-theoretic PDE in the context of Ricci-flat manifolds with special holonomy, i.e., those Riemannian manifolds with holonomy group
$$H = \textstyle \SU(n), \ \ \Sp(n), \ \ \G_2, \ \ \Spin(7).$$
Perhaps the simplest such PDE are the \emph{$H$-instanton equations}, which generalize the anti-self-dual (ASD) equations familiar from the theory of $4$-manifolds. \\
\indent Historically, the study of these four instanton equations has proceeded along rather different lines.  For example, the $\SU(n)$-instanton equation is precisely the \emph{primitive Hermitian Yang-Mills} (pHYM) condition, a highly well-studied PDE that enjoys deep ties to algebraic geometry via the Donaldson--Uhlenbeck--Yau Theorem.  Where $\G_2$- and $\Spin(7)$-instantons are concerned, influential papers of Donaldson--Thomas \cite{donaldson1998gauge} and Donaldson--Segal \cite{donaldson2011gauge} suggest that such connections might play a role in defining an enumerative invariant of $\G_2$- and $\Spin(7)$-manifolds.   \\
\indent The $\Sp(n)$-instanton equation, on the other hand, has received comparably less attention.  In the literature, vector bundles equipped with an $\Sp(n)$-instanton are called \emph{hyperholomorphic}, as they carry a triple of holomorphic structures that satisfy the quaternionic relations.  Such bundles have been studied mathematically by, for example, Verbitsky \cite{verbitsky1993hyperholomorphic, verbitsky2003hyperkahler}, Kaledin--Verbitsky \cite{kaledin1998non}, Feix \cite{feix2002hypercomplex}, Haydys \cite{haydys2008hyperkahler}, Hitchin \cite{hitchin2014hyperholomorphic}, Iona{\c{s}} \cite{ionacs2019twisted}, Meazzini--Onorati \cite{meazzini2023hyper}, among others, and arise in physics via $4$-dimensional supersymmetric field theories \cite{neitzke2011hyperholomorphic}.  Given the importance of ASD connections (which are nothing other than $\Sp(1)$-instantons) in the study of $4$-manifolds, we expect that the moduli space of $\Sp(n)$-instantons will prove useful in the study of hyperk\"{a}hler $4n$-manifolds. \\

\indent A second motivation for this work comes from mirror symmetry.  Roughly speaking, mirror symmetry predicts the existence of deep relationships between the calibrated geometry of a special holonomy manifold $X$, and gauge theory on its ``mirror" $X^*$.  In the Calabi-Yau setting, Leung--Yau--Zaslow \cite{leung2000special} showed that under certain simplifying assumptions, the \emph{real Fourier-Mukai (RFM) transform} provides such a relationship.  Later, Lee--Leung \cite{lee2009geometric} studied this transform in the $\G_2$ and $\Spin(7)$ settings, the latter of which was explored further by Kawai--Yamamoto \cite{kawai2021real}.  In all three contexts, the RFM transform relates calibrated graphs to \emph{deformed $H$-instantons}, a nonlinear analogue of the $H$-instanton equation.
\begin{align*}
\text{Calabi-Yau geometry:} &&   \text{special Lagrangian graph} & \ \longleftrightarrow \text{ deformed HYM connection} \\
\text{$\G_2$ geometry:} & & \text{co/associative graph} & \ \longleftrightarrow \text{ deformed $\G_2$-instanton} \\
\text{$\Spin(7)$ geometry:} & & \text{Cayley graph} &\  \longleftrightarrow \text{ deformed $\Spin(7)$-instanton}
\end{align*}
In hyperk\"{a}hler $4n$-manifolds, perhaps the most natural class of calibrated submanifolds consists of the \emph{complex Lagrangians}, those $2n$-dimensional submanifolds that are both complex with respect to $I$ and Lagrangian with respect to $J$ and $K$.  Thus, by analogy with the Calabi-Yau, $\G_2$, and $\Spin(7)$ settings, we would like to say that a connection arising from the RFM transform of a graphical submanifold $S$ satisfies the \emph{deformed $\Sp(n)$-instanton equation} if and only if $S$ is complex Lagrangian. \\
\indent Our first result is a derivation of this putative deformed $\Sp(n)$-instanton equation.  Surprisingly, we find that it coincides with the ordinary $\Sp(n)$-instanton equation.  To be precise:

\begin{thm} \label{thm:RFM-Result} Let $B \subset \R^{2n}$ be an open set, and equip $X = B \times T^{2n}$ with a product hyperk\"{a}hler structure $(g, (I,J,K))$ such that $B \times \{\mathrm{pt}\} \subset X$ is a complex Lagrangian with respect to $I$.  Let $h \colon B \to T^{2n}$ be a smooth function, let $S_h \subset X$ be its graph, and let $\nabla^h$ be its real Fourier-Mukai transform on $\underline{\C} \to B \times (T^n)^*$.  Then:
$$S_h \text{ is complex Lagrangian with respect to }I \ \iff \ \nabla^h \text{ is an }\Sp(n)\text{-instanton.}$$
\end{thm}

\indent Perhaps the most striking feature of Theorem \ref{thm:RFM-Result} is the linearity of the deformed $\Sp(n)$-instanton equation.  Indeed, since the complex Lagrangian condition is defined by the vanishing of a pair of $2$-forms, one expects that the deformed $\Sp(n)$-instanton equation will be quadratic in the curvature rather than linear.  (The reason for this linearity is explained in the proof.)  From the point of view of mirror symmetry, it is tempting to interpret the above result as an instance of the philosophy that hyperk\"{a}hler manifolds are in some sense ``self-mirror" \cite{verbitsky1999mirror}.

\subsection{$\Sp(n)$-Instantons}

\indent \indent We are thus motivated to study $\Sp(n)$-instantons on hyperk\"{a}hler manifolds $(X, g, (I,J,K))$.  Essentially by definition,
$$A \text{ is an }\Sp(n)\text{-instanton} \ \iff \ A \text{ is primitive HYM with respect to } I, J,\text{ and }K.$$
In this work, we will study the situation in which $X$ is a cone or asymptotically conical. \\
\indent Suppose first that $X^{4n+4} = \Cone(M)$ is a hyperk\"{a}hler cone with $n \geq 1$, and let $M^{4n+3}$ be its $3$-Sasakian link.  In Proposition \ref{prop:Cone-Sp(n)}, we show that conical $\Sp(n+1)$-instantons on $X$ are modeled by \emph{tri-contact instantons} on $M$, connections that are contact instantons in an $S^2$-family of ways.  This raises the question of how to construct tri-contact instantons.  To this end, we establish two dimensional reductions.  Specifically, fixing one of the projections $p \colon M \to Z^{4n+2}$ to the $(4n+2)$-dimensional (K\"{a}hler-Einstein) twistor space $Z$, and letting $h \colon M \to Q^{4n}$ denote the projection to the $4n$-dimensional quaternionic-K\"{a}hler quotient, in $\S$\ref{sub:DimRed} we show that:
\begin{align}
p^*A \text{ is a tri-contact instanton on }M^{4n+3} & \iff A \text{ is an }\Sp(n)\text{-instanton on }Z^{4n+2} \label{eq:DimRed1} \\
h^*A \text{ is a tri-contact instanton on }M^{4n+3} & \iff A \text{ is an }\Sp(n)\text{-instanton on }Q^{4n}. \label{eq:DimRed2}
\end{align}
That the notion of ``$\Sp(n)$-instanton" makes sense on the $(4n+2)$-dimensional twistor space $Z$ is explained in $\S$\ref{sub:DimRed}.  These generalize results that are well-known when $n = 1$. \\
\indent Returning to the setting of smooth hyperk\"{a}hler manifolds, we consider the obvious implication
$$A \text{ is an }\Sp(n)\text{-instanton } \implies A \text{ is primitive HYM with respect to }I,$$
and ask to what extent the converse might hold.  If $X$ is compact and at least one $\Sp(n)$-instanton exists, then the converse can be deduced from results of Verbitsky \cite{verbitsky1993hyperholomorphic}.  More precisely:

\begin{thm} \label{thm:Lewis-Compact} Let $(X^{4n}, g, (I,J,K))$ be a compact hyperk\"{a}hler $4n$-manifold.  Let $P \to X$ be a principal $G$-bundle, where $G$ is a compact Lie group.  Suppose that $P$ admits an $\Sp(n)$-instanton.  For any connection $A$ on $P$, we have
$$A \text{ is an }\Sp(n)\text{-instanton } \iff A \text{ is primitive HYM with respect to }I.$$
\end{thm}

\indent In $\S$\ref{sec:CompactCase}, we will give a self-contained proof of Theorem \ref{thm:Lewis-Compact} with an eye towards the non-compact setting.  That is, if $X$ is asymptotically conical (AC), and assuming that at least one AC $\Sp(n)$-instanton exists, we expect by analogy that AC primitive HYM connections with sufficiently fast decay ought to be $\Sp(n)$-instantons.  Indeed, in $\S$\ref{sec:AC-Case}, we prove the following theorem:

\begin{thm} \label{thm:Lewis-AC} Let $P \to (X^{4n}, g, (I,J,K))$ be a principal $G$-bundle over an AC hyperk\"{a}hler manifold $X$ of rate $\nu < 0$, where $G$ is a compact Lie group.  Let $\Sigma^{4n-1}$ be the asymptotic link of $X$, and fix a tri-contact instanton $A_\infty$ on $Q \to \Sigma$, where $Q$ is an asymptotic framing of $P$. \\
\indent Suppose that $\nu \leq -\frac{2}{3}(2n+1)$, and that $P$ admits an $\Sp(n)$-instanton asymptotic to $A_\infty$ with decay rate less than $\nu + 1$.  Then every $I$-primitive HYM connection on $P$ that is asymptotic to $A_\infty$ with sufficiently fast decay is an $\Sp(n)$-instanton.
\end{thm}

\noindent The term ``sufficiently fast decay" is explained in the more precise statement of Theorem \ref{thm:ACLewis}.

\subsection{Organization and Notation}

\indent \indent This work is organized as follows.  In $\S$\ref{sec:Prelims}, we recall the definitions of ``$\Omega$-ASD connection" and ``$H$-instanton," and discuss basic aspects of the RFM transform.  In $\S$\ref{sec:Spn-Instantons}, we systematically discuss $\Sp(n)$-instantons in both the quaternionic-K\"{a}hler and hyperk\"{a}hler settings, and speculate on bubbling loci in the latter context (see $\S$\ref{subsub:Bubbling}).  Subsection \ref{sub:RFM-ComplexLag} contains a proof of Theorem \ref{thm:RFM-Result}, and $\S$\ref{sub:Complex-Lag-dHYM} points out that in hyperk\"{a}hler manifolds with complex Lagrangian distributions, there exists a simple PDE whose solutions yield dHYM connections. \\
\indent Section \ref{sec:HKCones-3Sas} discusses conical $\Sp(n)$-instantons on hyperk\"{a}hler cones, tri-contact instantons on $3$-Sasakian links, and establishes the dimensional reductions (\ref{eq:DimRed1}) and (\ref{eq:DimRed2}).  As preparation, we include the preliminary $\S$\ref{sec:CYCones-SasEin} concerning the simpler situation of Calabi-Yau cones and their Sasaki-Einstein links.  Finally, in $\S$\ref{sec:LewisThms}, we prove the Lewis-type Theorems \ref{thm:Lewis-Compact} and \ref{thm:Lewis-AC}.

\begin{notat} \label{notation} ${}$
\begin{itemize}
\item Let $E \to (M,g_M)$ be a Hermitian vector bundle over a Riemannian manifold.  For an $E$-valued $k$-form $\alpha \in \Lambda^k(T^*M) \otimes E$, we let $|\alpha| = |\alpha|_{g_M}$ denote the natural tensor product norm induced by the metrics on $T^*M$ and $E$.  When $\alpha \in \Omega^k(M;E)$ is a section, we let $|\alpha|$ denote the pointwise norm, and (when $M$ is compact) let $\Vert \alpha \Vert := \Vert \alpha \Vert_{L^2}$ denote the $L^2$ norm.
\item We employ Salamon bracket notation \cite{salamon1989riemannian}.  That is, if $V \simeq \C^k$ is a $\C$-vector space, then $\LB V \RB \simeq \R^{2k}$ denotes its underlying $\R$-vector space.  If, in addition, $V$ is equipped with a real structure $r \in \End(V)$ (i.e, an endomorphism with $r^2 = \mathrm{Id}_V$), we let $[V] \simeq \R^k$ denote the $(+1)$-eigenspace of $r$.  Note that $[V] \otimes_{\R} \C \cong V$, whereas $\LB V \RB \otimes_{\R} \C \cong V \oplus \overline{V}$.
\end{itemize}
\end{notat}

\noindent \textbf{Convention:} On a Hermitian vector bundle $E \to M$, the term ``connection on $E$" always refers to a connection on $E$ that is compatible with the Hermitian bundle metric. \\

\noindent \textbf{Acknowledgements:} The representation theory in $\S$\ref{subsub:RepTheory} grew out of discussions with Benjamin Aslan.  We also acknowledge helpful conversations with Nick Addington, Gavin Ball, Daniel Fadel, Udhav Fowdar, Jason Lotay, Gon\c{c}alo Oliveira, Henrique S\`{a} Earp, and Alex Waldron.  Finally, we thank Boris Botvinnik and Aleksander Doan for their support and encouragement.

\section{Preliminaries} \label{sec:Prelims}

\indent \indent We begin by setting foundations, recalling basic ideas from gauge theory.  Subsections \ref{sub:Omega-ASD} and \ref{sub:H-instantons} concern two notions of higher-dimensional instantons, namely $\Omega$-ASD connections and $H$-instantons, our discussion drawing from \cite{chen2022compactness}, \cite{fadel2019gauge},  \cite{oliveira2021yang}.  Subsection \ref{sub:RFM} recalls the real Fourier-Mukai transform following \cite{kawai2021real}.

\subsection{$\Omega$-ASD Connections} \label{sub:Omega-ASD}

\indent \indent Let $(M^n, g)$ be an oriented Riemannian $n$-manifold with $n \geq 4$, and fix a smooth $(n-4)$-form $\Omega \in \Omega^{n-4}(M)$ with comass at most $1$.  A connection $A$ on a vector bundle $E \to M$ is called \emph{$\Omega$-anti-self-dual} (or \emph{$\Omega$-ASD}) if it satisfies
\begin{equation} \label{eq:Omega-ASD}
\ast\! F_A = -\Omega \wedge F_A.
\end{equation}
By applying $d_A$ and the Bianchi identity, one sees that any such connection $A$ satisfies the \emph{$\Omega$-Yang-Mills equation} (or the \emph{Yang-Mills with torsion equation}):
\begin{equation} \label{eq:Omega-YM}
d_A(\ast F_A) = -d\Omega \wedge F_A.
\end{equation}
An equivalent condition is $d_A^*(F_A + \ast (F_A \wedge \Omega)) = 0$.  It is well-known that the $\Omega$-Yang-Mills equations are the Euler-Lagrange equations of the functional
\begin{equation} \label{eq:Omega-YM-Functional}
\mathrm{YM}_\Omega(A) = \int_M |F_A|^2\,\mathrm{vol} - \int_M \mathrm{tr}(F_A \wedge F_A) \wedge \Omega.
\end{equation}
Moreover, $A$ is an absolute minimizer of $\mathrm{YM}_\Omega$ if and only if it is $\Omega$-ASD. \\

\indent Now, if $\Omega$ is closed, then the $\Omega$-Yang-Mills equation reduces to the usual \emph{Yang-Mills equation}
\begin{equation} \label{eq:YM}
d_A(\ast F_A) = 0.
\end{equation}
If, in addition, $M$ is compact, then the second term in (\ref{eq:Omega-YM-Functional}) is topological, and so the critical points of $\mathrm{YM}_\Omega$ are equivalent to those of the usual Yang-Mills functional $\mathrm{YM}(A) = \int_M |F_A|^2\,\mathrm{vol}$.

\begin{example} The classical situation is that of $n = 4$ and $\Omega = 1$.  In that case, the $\Omega$-ASD equation (\ref{eq:Omega-ASD}) reduces to the familiar ASD equation $\ast F_A = -F_A$.
\end{example}

\indent We reiterate that if $\Omega$ is not closed, then $\Omega$-ASD connections are $\Omega$-Yang Mills, but not necessarily Yang-Mills.  Nevertheless, for certain special choices of (non-closed) $\Omega$, it can happen that $\Omega$-ASD connections are automatically Yang-Mills.  For example, this phenomenon occurs when $(M,g)$ is the (Einstein) link of a Ricci-flat cone with special holonomy $H$.  More precisely:
\begin{itemize}
\item ($H = \SU(n)$) If $M^{2n-1}$ is a Sasaki-Einstein manifold with $n \geq 3$, then the \emph{contact instanton equation} (see $\S$\ref{subsec:Contact-Sas-Ein}) is equivalent to an $\Omega$-ASD equation for a particular non-closed $\Omega \in \Omega^{2n-5}(M)$.  In this context, contact instantons are Yang-Mills \cite{baraglia2016moduli}, \cite{portilla2023instantons}.
\item ($H = \G_2$) If $M^6$ is a strict nearly-K\"{a}hler $6$-manifold, then the \emph{Hermitian-Yang-Mills equation} (see $\S$\ref{subsec:HYM}) is equivalent to an $\Omega$-ASD equation for a particular non-closed $\Omega \in \Omega^2(M)$.  In this context, HYM connections are Yang-Mills \cite{xu2009instantons}.
\item ($H = \Spin(7)$) If $M^7$ is a nearly-parallel $7$-manifold, then the \emph{$\G_2$-instanton equation} (see $\S$\ref{subsec:Examples-H-Instantons}) is equivalent to an $\Omega$-ASD equation for a particular non-closed $\Omega \in \Omega^{3}(M)$.  In this context, $\G_2$-instantons are Yang-Mills \cite{harland2012instantons}.
\end{itemize}

\subsection{$H$-Structures} \label{sec:H-struc}

\indent \indent Let $M$ be a smooth $n$-manifold, and let $H \leq \GL_n(\R)$ be a Lie subgroup.  An \emph{$H$-structure on $M$} is an $H$-subbundle of the general frame bundle of $M$.  In most cases of interest, the data of an $H$-structure is equivalent to a collection of tensors on $M$ satisfying suitable algebraic conditions.  For instance:

\begin{example} \label{ex:U(n)-str} Let $M^{2n}$ be a $2n$-manifold.  A \emph{$\U(n)$-structure} on $M$ is equivalent to a triple $(g,J,\omega)$ consisting of a Riemannian metric $g$, a $g$-orthogonal almost-complex structure $J$, and the non-degenerate $2$-form $\omega := g(J\cdot, \cdot)$.  A \emph{K\"{a}hler structure} is a $\U(n)$-structure with $\nabla \omega = 0$ (or, equivalently, $\nabla J = 0$), where $\nabla$ is the Levi-Civita connection of $M$.
\end{example}

\indent We will be particularly interested in those $H$-structures where $H$ is one of the Ricci-flat holonomy groups on Berger's list, i.e.:
\begin{equation} \label{eq:Ricci-Flat-Holonomy}
\textstyle \SU(n), \ \ \ \Sp(n), \ \ \ \G_2, \ \ \ \Spin(7).
\end{equation}
In each of the following, $\nabla$ denotes the Levi-Civita connection of the given Riemannian metric.
\begin{enumerate}
\item Let $M^{2n}$ be a $2n$-manifold.  An \emph{$\SU(n)$-structure} $(g,J,\omega,\Upsilon)$ on $M$ is equivalent to a $\U(n)$-structure $(g,J,\omega)$ together with a a complex volume form $\Upsilon \in \Omega^{n,0}(M)$, normalized so that
$$c_n\Upsilon \wedge \overline{\Upsilon} = \frac{1}{n!}\omega^n, \ \ \ \ c_n = (-1)^{n(n-1)/2} (i/2)^n.$$
A \emph{Calabi-Yau structure} is an $\SU(n)$-structure satisfying $\nabla \omega = 0$ and $\nabla \Upsilon = 0$ (or, equivalently, $\nabla J = 0$ and $\nabla \Upsilon = 0$).  It is well-known that if $(g,J,\omega,\Upsilon)$ is Calabi-Yau, then $\Hol^0(g) \leq \SU(n)$, and hence $g$ is Ricci-flat.  Conversely, every Riemannian metric with holonomy contained in $\SU(n)$ arises from a Calabi-Yau structure.
\item Let $M^{4n}$ be a $4n$-manifold.  An \emph{$\Sp(n)$-structure} on $M$ is equivalent to a pair $(g, (I,J,K))$ consisting of a Riemannian metric $g$ and a triple $(I,J,K)$ of $g$-orthogonal almost-complex structures satisfying the quaternionic relations $IJ = K$.  In this case, we often write $\omega_I := g(I\cdot, \cdot)$, and similarly for $\omega_J$ and $\omega_K$.  Note that an $\Sp(n)$-structure on $M$ induces an $\SU(2n)$-structure $(g, I, \omega_I, \Upsilon_I)$ on $M$, where $\Upsilon_I := \frac{1}{n!}(\omega_J + i\omega_K)^n$. \\

A \emph{hyperk\"{a}hler structure} is an $\Sp(n)$-structure satisfying $\nabla I = \nabla J = \nabla K = 0$.  It is well-known that if $(g,(I,J,K))$ is hyperk\"{a}hler, then $\Hol^0(g) \leq \Sp(n)$, and hence $g$ is Ricci-flat.  Conversely, every Riemannian metric with holonomy contained in $\Sp(n)$ arises from a hyperk\"{a}hler structure.
\item Let $M^7$ be a $7$-manifold.  A \emph{$\G_2$-structure} on $M$ is equivalent to a $3$-form $\varphi \in \Omega^3(M)$ such that at each $x \in M$, there exists a coframe $(e^1, \ldots, e^7) \colon T_xM \to \R^7$ for which
$$\left.\varphi\right|_x = e^{123} + e^{145} + e^{167} + e^{246} - e^{257} - e^{347} - e^{356}$$
where $e^{ijk} := e^i \wedge e^j \wedge e^k$.  Every $\G_2$-structure $\varphi$ induces a canonical Riemannian metric $g_{\varphi}$ on $M$.  It is well-known that if $\nabla \varphi = 0$, then $\Hol^0(g_\varphi) \leq \G_2$, and hence $g_\varphi$ is Ricci-flat.  Conversely, every Riemannian metric with holonomy contained in $\G_2$ arises in this way.
\item Let $M^8$ be an $8$-manifold.  A \emph{$\Spin(7)$-structure} on $M$ is equivalent to a $4$-form $\Phi \in \Omega^4(M)$ such that at each $x \in M$, there exists a coframe $(e^1, \ldots, e^8) \colon T_xM \to \R^8$ for which
$$\left.\Phi\right|_x = e^{1234} + (e^{12} + e^{34}) \wedge (e^{56} + e^{78}) + (e^{13} - e^{24}) \wedge (e^{57} - e^{68}) - (e^{14} + e^{23}) \wedge (e^{58} + e^{67}) + e^{5678}$$
where $e^{ij} = e^i \wedge e^j$ and similarly for $e^{ijk\ell}$.  Every $\Spin(7)$-structure $\Phi$ induces a canonical Riemannian metric $g_{\Phi}$ on $M$.  It is well-known that if $\nabla \Phi = 0$, then $\Hol^0(g_\Phi) \leq \Spin(7)$, and hence $g_\Phi$ is Ricci-flat.  Conversely, every Riemannian metric with holonomy contained in $\Spin(7)$ arises in this way.
\end{enumerate}
Later in this work, we will require the notions of an $\Sp(n)\U(1)$-structure on $M^{4n+2}$, and an $\Sp(n)\Sp(1)$-structure on $M^{4n}$; these will be introduced as needed.

\subsection{$H$-Instantons} \label{sub:H-instantons}

\indent \indent Let $(M^n, g)$ be an oriented Riemannian $n$-manifold equipped with an $H$-structure, where $H \leq \SO(n)$ is a closed subgroup.  Letting $\mathfrak{h} \subset \mathfrak{so}(n)$ denote the Lie algebra of $H$, we may orthogonally split $\Lambda^2(\R^n)^* \cong \mathfrak{so}(n) = \mathfrak{h} \oplus \mathfrak{h}^\perp$ with respect to the Killing form.  Accordingly, the bundle of $2$-forms on $M$ splits as
$$\Lambda^2(T^*M) \cong \Lambda^2_{\mathfrak{h}} \oplus \Lambda^2_{\mathfrak{h}^\perp}$$
\indent Let $E \to M$ be a vector bundle with structure group $G$, where $G$ is a compact Lie group, and let $\mathrm{ad}_E \to M$ denote the adjoint bundle.  Then the bundle of $\mathrm{ad}_E$-valued $2$-forms decomposes as
\begin{align*}
\Lambda(T^*M) \otimes \mathrm{ad}_E \cong \left( \Lambda^2_{\mathfrak{h}} \otimes \mathrm{ad}_E \right) \oplus ( \Lambda^2_{\mathfrak{h}^\perp} \otimes \mathrm{ad}_E ).
\end{align*}
We let $\pi_{\mathrm{h}}$ and $\pi_{\mathfrak{h}^\perp}$ denote the corresponding projection maps from $\Lambda^2(T^*M) \otimes \mathrm{ad}_E$ onto its respective components.  The following definition is due to Reyes Carri\'{o}n \cite{carrion1998generalization}.

\begin{defn} A connection $A$ on the vector bundle $E \to M$ is called an \emph{$H$-instanton} if its curvature $2$-form $F_A \in \Omega^2(M; \mathrm{ad}_E)$ takes values in $\mathfrak{h}$, in the sense that:
$$\pi_{\mathfrak{h}}(F_A) = F_A.$$
An equivalent condition is $\pi_{\mathfrak{h}^\perp}(F_A) = 0$.
\end{defn}
We emphasize that this condition is linear.  More precisely, if $A$ is a connection on the trivial complex line bundle $\underline{\C} \to M$, then the $H$-instanton equation is a linear condition on $F_A$, and is therefore a first-order linear PDE for $A$.

\subsubsection{Hermitian Yang-Mills Connections} \label{subsec:HYM}

\indent \indent Let $M^{2n}$ be a $2n$-manifold with a $\U(n)$-structure $(g,J,\omega)$, and let $E \to M$ be a Hermitian vector bundle of complex rank $r$.  By the above discussion, a connection $A$ on $E \to M$ is a \emph{$\U(n)$-instanton} provided that any of the following equivalent conditions holds:
$$F_A \in \Lambda^2_{\mathfrak{u}(n)} \otimes \mathrm{ad}_E \ \iff \ F_A \in [\Lambda^{1,1}] \otimes \mathrm{ad}_E.$$
Extending $A$ by $\C$-linearity (and denoting this extension by $A$), the $\U(n)$-instanton condition becomes:
$$F_A \in \Lambda^{1,1} \otimes \mathrm{ad}_E \ \iff \ F_A^{0,2} = 0.$$
Note that if $A$ is a $\U(n)$-instanton, then
$$F_A \wedge \omega^{n-1} = \langle F_A,\omega \rangle\,\omega^n,$$
where $\langle F_A,\omega \rangle \in \Omega^0(M; \mathfrak{u}(r))$ is a matrix-valued function.  A \emph{Hermitian Yang-Mills (HYM) connection} is a $\U(n)$-instanton $A$ satisfying $\langle F_A,\omega \rangle = \lambda_E\,\mathrm{Id}$ for some constant $\lambda_E \in \C$.  In this case, if $M$ is a compact K\"{a}hler manifold, then it is well-known that
$$\lambda_E = -\frac{2\pi i}{n!\,\mathrm{vol}(M)} \frac{\deg(E)}{r}.$$
The HYM connections with $\lambda_E = 0$ are called \emph{primitive HYM connections}.

\begin{rmk} \label{rmk:U(n)-Holo} Let $E \to M^{2n}$ be a Hermitian vector bundle, and suppose $M$ is K\"{a}hler.  In this setting, there is a well-known relationship between $\U(n)$-instantons and holomorphic structures.  Indeed, if $E$ has a holomorphic structure, then there exists a unique $\U(n)$-instanton $A$ (the \emph{Chern connection}) compatible with the holomorphic structure.  Conversely, if $A$ is a $\U(n)$-instanton on $E$, then by the Koszul-Malgrange Theorem, $E$ admits a holomorphic structure compatible with $A$.  
\end{rmk}

\subsubsection{Examples of $H$-instantons} \label{subsec:Examples-H-Instantons}

\indent \indent We will primarily be concerned with $H$-instantons where $H$ is a Ricci-flat holonomy group on Berger's list (\ref{eq:Ricci-Flat-Holonomy}).  As we now recall, in each of these cases, the $H$-instanton equation is equivalent to the $\Omega$-ASD equation for a suitable choice of $\Omega \in \Omega^{n-4}(M)$.

\begin{enumerate}
\item ($H = \SU(n)$) Let $M^{2n}$ be an oriented $2n$-manifold with an $\SU(n)$-structure $(g,J,\omega, \Upsilon)$, and let $E \to M$ be a Hermitian vector bundle.  It is well-known that the $\SU(n)$-instanton equation can be rephrased as follows
\begin{equation*}
F_A \in \Lambda^2_{\mathfrak{su}(n)} \otimes \mathrm{ad}_E \ \iff \  F_A \in [\Lambda^{1,1}_0] \otimes \mathrm{ad}_E
\end{equation*}
This is equivalent to requiring that the $\C$-linear extension of $A$ (which we continue to call $A$) satisfies:
\begin{equation*}
F_A \in \Lambda^{1,1}_0 \otimes \mathrm{ad}_E \ \iff \  \begin{cases} F_A^{0,2} = 0 \\ F_A \wedge \omega^{n-1} = 0. \end{cases} \ \iff \  \ast F_A = -\frac{1}{(n-2)!}\omega^{n-2} \wedge F_A.
\end{equation*}
In particular, ($\C$-linear extensions of) $\SU(n)$-instantons are precisely the \emph{primitive} Hermitian Yang-Mills connections, which are in turn precisely the $\Omega$-ASD connections for $\Omega = \frac{1}{(n-2)!}\omega^{n-2} \in \Omega^{2n-4}(M)$. \\  

\item ($H = \Sp(n)$) Let $M^{4n}$ be an oriented $4n$-manifold with an $\Sp(n)$-structure $(g, (I,J,K))$, and let $E \to M$ be a Hermitian vector bundle.  Then the $\Sp(n)$-instanton equation can be rephrased as follows:
$$F_A \in \Lambda^2_{\mathfrak{sp}(n)} \otimes \mathrm{ad}_E \ \iff \ F_A \in \left( [\Lambda^{1,1}_{0,I}] \cap [\Lambda^{1,1}_{0,J}] \cap [\Lambda^{1,1}_{0,K}] \right) \otimes \mathrm{ad}_E,$$
where $\Lambda^{1,1}_{0,L} \to M$ denotes the bundle of $2$-forms that are type $(1,1)$ and primitive with respect to $L \in \{I,J,K\}$.

Thus, a connection $A$ is an $\Sp(n)$-instanton if and only if (its $\C$-linear extension) is \emph{hyper-holomorphic}, meaning that $A$ is primitive HYM with respect to $I,J,K$ simultaneously.  In particular, hyper-holomorphic connections are those that are simultaneously $\Omega_I$-, $\Omega_J$-, and $\Omega_K$-ASD, where $\Omega_L = \frac{1}{(2n-2)!}\omega_L^{2n-2} \in \Omega^{4n-4}(M)$ for $L \in \{I,J,K\}$.  We will explore these further in $\S$\ref{sec:Spn-Instantons}. \\

\item ($H = \G_2$) Let $M^7$ be an oriented $7$-manifold with a $\G_2$-structure $\varphi \in \Omega^3(M)$, and let $E \to M$ be a Hermitian vector bundle.  It is well-known that the $\G_2$-instanton equation can be rephrased as follows:
$$F_A \in \Lambda^2_{\mathfrak{g}_2} \otimes \mathrm{ad}_E \ \ \iff \ \ F_A \wedge \ast \varphi = 0 \ \ \iff \ \ \ast F_A = -\varphi \wedge F_A.$$
In particular, $\G_2$-instantons are $\Omega$-ASD for the choice $\Omega = \varphi \in \Omega^3(M)$. \\

\item ($H = \Spin(7)$)  Let $M^8$ be an oriented $8$-manifold with a $\Spin(7)$-structure $\Phi \in \Omega^4(M)$, and let $E \to M$ be a Hermitian vector bundle.  It is well-known that the $\Spin(7)$-instanton equation can be rephrased as follows:
$$F_A \in \Lambda^2_{\mathfrak{spin}(7)} \otimes \mathrm{ad}_E \ \ \iff \ \ \ast F_A = -\Phi \wedge F_A.$$
In particular, $\Spin(7)$-instantons are $\Omega$-ASD for the choice $\Omega = \Phi \in \Omega^4(M)$.
\end{enumerate}

\begin{rmk}[Representation dependence] \label{rmk:RepDep} Despite the terminology, the notions of $H$-structure and $H$-instanton depend not only on the abstract Lie group $H$, but \emph{also} on the representation $H \leq \SO(n)$.
\end{rmk}

\begin{rmk}[Generalization] \label{rmk:NH-Gen} The notion of $H$-instanton actually makes sense for a more general class of geometric structures.  Indeed, let $H \leq \SO(n)$ be a closed subgroup, let $N(H) \leq \SO(n)$ denote its normalizer in $\SO(n)$, and assume that $N(H)$ is connected.  Note that $N(H)$ acts on $\mathfrak{so}(n)$ via the adjoint action, and it can be shown \cite{carrion1998generalization} that this $N(H)$-action preserves the subspace $\mathfrak{h} \subset \mathfrak{so}(n)$.  Consequently, ``$H$-instanton" is a well-defined concept on manifolds equipped with simply an $N(H)$-structure.  This generalization is relevant to several contexts:
\begin{enumerate}[(a)]
\item If $H = \SU(2) \leq \SO(4)$, then $N(H) = \SO(4)$, so the notion of an $\SU(2)$-instanton makes sense on any oriented Riemannian $4$-manifold.  In fact, it coincides with the classical SD or ASD equation, depending on the conjugacy class of $\SU(2)$ in $\SO(4)$ (see Remark \ref{rmk:RepDep}).
\item If $H = \SU(n) \leq \SO(2n)$ with $n \geq 3$, then $N(H) = \U(n)$, so the notion of an $\SU(n)$-instanton makes sense on any $2n$-manifold (for $2n \geq 6$) with a $\U(n)$-structure.
\item If $H = \Sp(n) \leq \SO(4n)$, then $N(H) = \Sp(n)\Sp(1)$, so the notion of an $\Sp(n)$-instanton makes sense on any $4n$-manifold with a $\Sp(n)\Sp(1)$-structure.  We will return to this point in $\S$\ref{sub:Spn-QK}.
\end{enumerate}
\end{rmk}

\subsection{The Real Fourier-Mukai Transform} \label{sub:RFM}

\indent \indent Let $X$ denote a Calabi-Yau, $\G_2$, or $\Spin(7)$-manifold.  Roughly speaking, mirror symmetry predicts the existence of a geometric functor relating the calibrated geometry and gauge theory of $X$ to that of its ``mirror" $X^*$.  When $(X, X^*)$ is a mirror pair of Calabi-Yau manifolds that satisfies certain conditions, Leung--Yau--Zaslow \cite{leung2000special} proposed such a functor, known as the \emph{real Fourier-Mukai (RFM) transform}.  Later, Lee--Leung \cite{lee2009geometric} considered the $\G_2$ and $\Spin(7)$ analogues of the RFM transform, the latter of which was revisited by Kawai--Yamamoto \cite{kawai2021real}.  Here, we will review this transform in the simplest idealized setting.  Our discussion closely follows \cite[$\S$2]{kawai2021real}. The reader desiring more details might consult \cite{kawai2021real} as well as \cite[$\S$3.2.1]{donaldson1997geometry}, \cite{leung2000special}, \cite{lee2009geometric}.

\indent Let $T^m = \R^m/(2\pi \Z^m)$ denote the $m$-torus, and let $(T^m)^* = (\R^m)^*/(2\pi \Z^m)^*$ denote its dual torus.  Every point $p \in T^m$ yields a flat connection $\nabla^p$ on the trivial Hermitian line bundle $\underline{\C} \to (T^m)^*$ in the following way. First, each point $p = (p^1, \ldots, p^m) \in T^m$ can be identified with the homomorphism
$$\mu_p = e^{-i\langle \cdot, p \rangle} \colon ( (2\pi\Z)^m)^* \to \U(1),$$
where $\langle \cdot, \cdot \rangle \colon (\R^m)^* \times \R^m \to \R$ is the natural dual pairing.  Identifying $((2\pi\Z)^m)^* \cong \pi_1( (T^m)^*)$, we can view $\mu_p \colon \pi_1( (T^m)^*) \to \U(1)$ as a unitary representation of the fundamental group of $(T^m)^*$.  From this representation and the principal $\pi_1( (T^m)^*)$-bundle $(\R^m)^* \to (T^m)^*$, we obtain an associated Hermitian line bundle over $(T^m)^*$, namely
$$E_p := (\R^m)^* \times_{\mu_p} \C \to (T^m)^*,$$
along with a flat Hermitian connection $\widetilde{\nabla}^p$ on $E_p$.  It turns out that $E_p$ admits a non-vanishing global section, so that $E_p$ is trivializable, and hence there is a bundle isomorphism $E_p \cong \underline{\C}$.  Via this isomorphism, we obtain a flat connection called $\nabla^p$ on $\underline{\C} \to (T^m)^*$, as desired.  In fact, letting $(y^1, \ldots, y^m)$ denote the standard coordinates on $(T^m)^*$, we have
$$\nabla^p = d + i \sum_{a=1}^m p^a dy^a.$$

\indent Next, we shift our perspective and view $T^m$ as a fiber of the trivial bundle $X := B \times T^m \to B$, where here $B \subset \R^k$ is an open set with coordinates $(x^1, \ldots, x^k)$.  Notice that every smooth function $h \colon B \to T^m$ provides a section $(\mathrm{Id}, h) \colon B \to B \times T^m$ whose image is the graph of $h$.  On the other hand, by the above discussion, the family of points $\{h(x) \in T^m \colon x \in B\}$ yields a family of connections $\nabla^{h(x)}$, which altogether produces a connection $\nabla^h$ on the trivial line bundle $\underline{\C} \to X^* := B \times (T^m)^*$. 

\indent In summary, to every smooth function $h \colon B^k \to T^m$, we may associate two objects:
\begin{enumerate}
\item A submanifold $S_h \subset X$ given by the graph of $h$:
$$S_h := \text{Graph}(h) = \{(x, h(x)) \colon x \in B\} \subset B^k \times T^m$$
\item A connection $\nabla^h$ on the trivial complex line bundle $\underline{\C} \to X^* := B \times (T^m)^*$ given by:
$$\nabla^h := d + i \sum_{a=1}^m h^a dy^a.$$
\end{enumerate}
We refer to $\nabla^h$ as the \emph{real Fourier-Mukai (RFM) transform} of $S_h$. 

\subsubsection{Examples} \label{subsub:Examples-Deformed}

\indent \indent The real Fourier-Mukai transform provides a correspondence between graphs $S_h \subset X$ and connections $\nabla^h$ on $\underline{\C} \to X^*$.  In particular, requiring that $S_h$ satisfy a geometric PDE is equivalent to requiring that $\nabla^h$ satisfy a ``mirror" geometric PDE. \\
\indent For example, suppose that $X^{k+m} = B^k \times T^m$ is equipped with a product $H$-structure in such a way that $B \times \{\text{pt}\} \subset X$ is a $\phi$-calibrated submanifold for a suitable calibration $\phi \in \Omega^k(X)$.  Then the condition that $S_h \subset X$ be $\phi$-calibrated is equivalent to requiring that $\nabla^h$ satisfy a so-called ``deformed $H$-instanton equation."  In general, the latter is a nonlinear first-order PDE system given by a polynomial condition (of degree at most $k$) on the curvature of $\nabla^h$. \\
\indent For illustration, we now recall some examples of this correspondence in which $H \leq \SO(k+m)$ is a Ricci-flat holonomy group on Berger's list (\ref{eq:Ricci-Flat-Holonomy}).

\begin{enumerate}
\item ($H = \SU(n)$) Let $(k,m) = (n,n)$, and equip $X^{2n} = B^n \times T^n$ with a product Calabi-Yau structure $(g, J, \omega, \Upsilon)$ such that $B \times \{\text{pt}\} \subset X$ is a special Lagrangian.  Let $h \colon B \to T^n$ be a smooth function.  Then it is well-known \cite{leung2000special}, \cite{kawai2021real} that
$$S_h \text{ is special Lagrangian } \iff \nabla^h \text{ is deformed Hermitian Yang-Mills,}$$
where the \emph{deformed Hermitian Yang-Mills (dHYM) equation} (or \emph{deformed $\SU(n)$-instanton equation}) is the system
\begin{align*} 
(F_A)^{0,2} & = 0 \\
\mathrm{Im}( (\omega + F_A)^{n}) & = 0.
\end{align*}

\item ($H = \Sp(n)$) Let $(k,m) = (2n, 2n)$, and equip $X^{4n} = B^{2n} \times T^{2n}$ with a product hyperk\"{a}hler structure $(g, (I,J,K))$ such that $B \times \{\text{pt}\} \subset X$ is complex Lagrangian with respect to $I$.  Let $h \colon B \to T^{2n}$ be a smooth function.  Then we may consider a \emph{deformed $\Sp(n)$-instanton equation} by requiring that
$$S_h \text{ is complex Lagrangian w.r.t. } I \iff \nabla^h \text{ is a deformed }\Sp(n)\text{-instanton.}$$
In Theorem \ref{thm:RFM}, we will prove that this putative ``deformed $\Sp(n)$-instanton equation" coincides with the ordinary $\Sp(n)$-instanton equation discussed in $\S$\ref{subsec:Examples-H-Instantons} and $\S$\ref{sub:Spn-HK}. \\

\item ($H = \G_2$) Let $(k,m) = (3,4)$, and equip $X^{7} = B^3 \times T^4$ with a product $\G_2$ structure $\varphi \in \Omega^3(X)$ such that $B \times \{\text{pt}\} \subset X$ is an associative submanifold.  Let $h \colon B \to T^4$ be a smooth function.  Then according to \cite{lee2009geometric}, \cite{kawai2021real}, we have that
$$S_h \text{ is associative } \iff \nabla^h \text{ is a deformed }\G_2\text{-instanton},$$
where the \emph{deformed $\G_2$-instanton equation} is
\begin{align*}
\textstyle \frac{1}{6}F_A^3 + F_A \wedge \ast \varphi = 0.
\end{align*}
If one instead takes $(k,m) = (4,3)$ and requires that $B \times \{\text{pt}\} \subset \R^4$ be a coassociative submanifold, one obtains the same deformed $\G_2$-instanton equation. \\

\item ($H = \Spin(7)$) Let $(k,m) = (4,4)$, and equip $X^{8} = B^4 \times T^4$ with a product $\Spin(7)$ structure $\Phi \in \Omega^4(X)$ such that $B \times \{\text{pt}\} \subset X$ is a Cayley submanifold.  Let $h \colon B \to T^4$ be a smooth function.  Then according to \cite{kawai2021real}, we have that
$$S_h \text{ is Cayley } \iff \nabla^h \text{ is a deformed }\Spin(7)\text{-instanton},$$
where the \emph{deformed $\Spin(7)$-instanton equation} is the system
\begin{align*}
\textstyle \pi^2_7\left( F_A + \frac{1}{6}F_A^3 \right) & = 0  \\
\textstyle \pi^4_7\left( F_A \wedge F_A \right) & = 0.
\end{align*}
\end{enumerate}

The following table summarizes the various geometric PDE discussed above.
$$\begin{tabular}{| l | l | l | l |} \hline
$(k,m)$ & Structure on $X$ & Condition on $S_h$ & Condition on $\nabla^h$ \\ \hline \hline
$(n,n)$ & Calabi-Yau & Special Lagrangian & Deformed HYM (dHYM) connection 
 \\ \hline
 $(2n, 2n)$ & Hyperk\"{a}hler & Complex Lagrangian & Deformed $\Sp(n)$-instanton 
 \\ \hline
$(3,4)$ & $\G_2$ & Associative & Deformed $\G_2$-instanton \\ \hline
$(4,3)$ & $\G_2$ & Coassociative & Deformed $\G_2$-instanton \\ \hline
$(4,4)$ & $\Spin(7)$ & Cayley & Deformed $\Spin(7)$-instanton \\ \hline
\end{tabular}$$

\section{$\Sp(n)$-Instantons} \label{sec:Spn-Instantons}

\indent \indent In this section, we set the foundations for our study of $\Sp(n)$-instantons.  In $\S$\ref{sub:Spn-QK}, we study $\Sp(n)$-instantons on quaternionic-K\"{a}hler $4n$-manifolds.  Then, in $\S$\ref{sub:Spn-HK}, we specialize to the hyperk\"{a}hler setting.  In that context, there are several equivalent characterizations of $\Sp(n)$-instantons: see Proposition \ref{prop:Sp(n)-equiv-def}.  At the end of the section, we briefly speculate on the geometry of bubbling loci of sequences of $\Sp(n)$-instantons.

\subsection{$\Sp(n)$-Instantons on Quaternionic-K\"{a}hler Manifolds} \label{sub:Spn-QK}

\subsubsection{Linear Algebra} \label{subsub:Lin-Alg-QuatHerm}

\indent \indent A \emph{quaternionic-Hermitian vector space} is a $4n$-dimensional $\R$-vector space $V$ equipped with a pair $(g, \mathcal{F})$, where $g(\cdot, \cdot) = \langle \cdot, \cdot \rangle$ is a positive-definite inner product, and $\mathcal{F} \subset \End(V)$ is a $3$-dimensional subspace that admits a basis $\{I,J,K\}$ of $g$-orthogonal complex structures satisfying the quaternionic relations $IJ = K$, etc.  Such a basis is called \emph{admissible}.  For each $k \geq 1$, the \emph{fundamental $4k$-forms} are
$$\textstyle \Pi_k := \frac{1}{(2k+1)!}(\omega_I^2 + \omega_J^2 + \omega_K^2)^k \in \Lambda^{4k}(V^*).$$
Each $\Pi_k$ is well-defined (independent of the admissible basis) and has co-mass one. \\
\indent It is well-known that the space of real $2$-forms $\Lambda^2(V^*)$ decomposes into three $\Sp(n)\Sp(1)$-irreducible submodules, say
\begin{align*}
\Lambda^2(V^*) = U \oplus W \oplus Y,
\end{align*}
where $W \cong \mathfrak{sp}(n)$ and $Y \cong \mathfrak{sp}(1)$.  In particular,
\begin{align*}
\dim(U) & = 6n^2 - 3n - 3 & \dim(W) & = 2n^2 + n & \dim(Y) & = 3.
\end{align*}
To be more explicit, let $\{I,J,K\}$ be an admissible basis of $V$, let $\{I^*, J^*, V^*\}$ be the dual basis of $V^*$, and let $\omega_L = g(L \cdot, \cdot) \in \Lambda^2(V^*)$ for $L \in \{I,J,K\}$.   Note that the $\Sp(n)\Sp(1)$-invariant endomorphism $T \in \End( \Lambda^2(V^*) )$ given by $T = I^* \otimes I^* + J^* \otimes J^* + K^* \otimes K^*$ is independent of the choice of basis.  Then according to \cite{capria1988yang}, the map $T$ has eigenvalues $3$ and $-1$, and moreover
\begin{align*}
W & = [\Lambda^{1,1}_I] \cap [\Lambda^{1,1}_J] \cap [\Lambda^{1,1}_K] = \{ \beta \in \Lambda^2(V^*) \colon T \beta = 3\beta\}, \\
U \oplus Y & = \LB \Lambda^{2,0}_I \RB + \LB \Lambda^{2,0}_J \RB + \LB \Lambda^{2,0}_K \RB = \{ \beta \in \Lambda^2(V^*) \colon T\beta = -\beta\}\!, \\
Y & = \mathrm{span}_{\R}(\omega_I, \omega_J, \omega_K),
\end{align*}
where $\Lambda^{p,q}_L \subset \Lambda^{p+q}(V^*; \C)$ denotes the $(p,q)$-forms on $V$ with respect to $L \in \{I,J,K\}$. Alternatively, according to \cite[Theorem 2.2]{galicki1991duality}, one has
\begin{align}
U & = \left\{ \beta \in \Lambda^2(V^*) \colon \ast\! \beta = 3\,\beta \wedge \Pi_{n-1} \right\} \notag \\
W & = \left\{ \beta \in \Lambda^2(V^*) \colon \ast\! \beta = -\beta \wedge \Pi_{n-1} \right\} \label{eq:W-char} \\
Y & = \textstyle \left\{ \beta \in \Lambda^2(V^*) \colon \ast\! \beta = \frac{3}{2n+1}\,\beta \wedge \Pi_{n-1} \right\}\!. \notag
\end{align}

\subsubsection{Quaternionic-K\"{a}hler Manifolds}

\indent \indent Let $Q^{4n}$ be an oriented $4n$-manifold.  An \emph{$\Sp(n)\Sp(1)$-structure} on $Q$ is equivalent to a pair $(g_Q, \mathcal{F})$ consisting of a Riemannian metric $g_Q$ and a rank $3$ subbundle $\mathcal{F} \subset \End(TQ)$ such that:
\begin{enumerate}
\item At each $q \in Q$, there exists a local frame $(j_1, j_2, j_3)$ of $\mathcal{F}$ satisfying the quaternionic relations $j_1 j_2 = j_3$ and $j_1^2 = j_2^2 = j_3^2 = -\mathrm{Id}$.
\item Every $j \in \mathcal{F}$ acts by isometries: $g_Q(jX, jY) = g_Q(X,Y)$ for all $X,Y \in TQ$.
\end{enumerate}
Equivalently, an $\Sp(n)\Sp(1)$-structure may be defined as a $4$-form $\Pi \in \Omega^4(Q)$ such that at each $q \in Q$, there exists a coframe $u \colon T_qQ \to \R^{4n}$ in which $\Pi|_q = \frac{1}{6}u^*(\omega_I^2 + \omega_2^2 + \omega_3^2)$, where here $\{\omega_I, \omega_2, \omega_3\}$ is the standard hyperk\"{a}hler triple on $\R^{4n}$. \\
\indent For $n \geq 2$, a \emph{quaternionic-K\"{a}hler structure} is an $\Sp(n)\Sp(1)$-structure $(g_Q, \mathcal{F})$ such that $\mathcal{F} \subset \End(TQ)$ is a $\nabla$-parallel subbundle, where $\nabla$ is the connection induced by $g_Q$.  An equivalent condition is that the $4$-form $\Pi \in \Omega^4(Q)$ is $g_Q$-parallel.  For $n = 1$, we say that $(Q^4, g_Q)$ is \emph{quaternionic-K\"{a}hler} provided that $g_Q$ is Einstein and anti-self-dual. \\ 

\indent Let $Q^{4n}$ be a quaternionic-K\"{a}hler manifold.  Since each tangent space $T_qQ$ has a natural quaternionic-Hermitian structure, the bundle of $2$-forms on $Q$ admits an $\Sp(n)\Sp(1)$-invariant decomposition
$$\Lambda^2(T^*Q) = U \oplus W \oplus Y$$
where $U, W, Y$ now refer to vector bundles (rather than vector spaces).  By Remark \ref{rmk:NH-Gen}, we may speak of $\Sp(n)$-instantons on quaternionic-K\"{a}hler manifolds.  Indeed, the above discussion shows the following:

\begin{prop} Let $A$ be a connection on a Hermitian vector bundle $E \to Q^{4n}$.  The following are equivalent:
\begin{enumerate}[(i)]
\item $A$ is an $\Sp(n)$-instanton on $E$ (i.e., $F_A \in \Gamma(W \otimes \mathrm{ad}_E)$).
\item $A$ is $\Omega$-ASD for $\Omega = \Pi_{n-1}$.
\end{enumerate}
\end{prop}

\indent The study of $\Sp(n)$-instantons on QK manifolds goes back to the late 1980's with the work of \cite{capria1988yang} and \cite{nitta1989moduli}, who show that they are Yang-Mills.  See also, for example, \cite{bartocci1998hyperkahler}, \cite{devchand2020hyperkahler}, \cite{devchand2020instantons}, \cite{galicki1991duality}.

\subsection{$\Sp(n)$-Instantons on Hyperk\"{a}hler Manifolds} \label{sub:Spn-HK}

\indent \indent We now specialize to $\Sp(n)$-instantons over hyperk\"{a}hler manifolds.  In the literature, vector bundles over hyperk\"{a}hler manifolds equipped with an $\Sp(n)$-instanton are called \emph{hyperholomorphic} (recalling Remark \ref{rmk:U(n)-Holo}).  Such bundles have been investigated in numerous studies, including for example Verbitsky \cite{verbitsky1993hyperholomorphic, verbitsky2003hyperkahler}, Kaledin--Verbitsky \cite{kaledin1998non}, Feix \cite{feix2002hypercomplex}, Haydys \cite{haydys2008hyperkahler}, Hitchin \cite{hitchin2014hyperholomorphic}, Iona{\c{s}} \cite{ionacs2019twisted}, Meazzini--Onorati \cite{meazzini2023hyper}, and others.

\subsubsection{Linear Algebra} \label{subsub:LinAlgHK}

\indent \indent A \emph{hyper-Hermitian vector space} is a $4n$-dimensional $\R$-vector space $V$ equipped with a pair $(g, (I,J,K))$, where $g(\cdot, \cdot) = \langle \cdot, \cdot \rangle$ is a positive-definite inner product, and $I,J,K \in \End(V)$ are $g$-orthogonal complex structures that satisfy the quaternionic relations $IJ = K$, etc.  In other words, a hyper-Hermitian vector space is a quaternionic-Hermitian vector space together with a choice of admissible basis.  As such, there is an $\Sp(n)\Sp(1)$-irreducible splitting
\begin{equation}
\Lambda^2(V^*) = U \oplus W \oplus Y \label{eq:Real-2forms}
\end{equation}
as discussed in $\S$\ref{subsub:Lin-Alg-QuatHerm}.  We aim to refine this decomposition further into $\Sp(n)$-irreducible pieces. \\
\indent For $L \in \{I,J,K\}$, define a non-degenerate $2$-form $\omega_L \in \Lambda^2(V^*)$ by $\omega_L = g(L \cdot, \cdot)$.  Next, letting $\Lambda^{p,q}_L \subset \Lambda^{p+q}(V; \C)$ denote the space of $(p,q)$-forms with respect to $L$, and letting $\Lambda^{p,p}_{0,L} \subset \Lambda^{2p}(V^*;\C)$ denote the primitive $(p,p)$-forms, we have a $\U(2n)$-irreducible decomposition
\begin{equation*}
\Lambda^2(V^*) = \LB \Lambda^{2,0}_L \RB \oplus [\Lambda^{1,1}_{0,L}] \oplus \R \omega_L.
\end{equation*}
Concretely,
\begin{align*}
\LB \Lambda^{2,0}_L \RB & = \left\{ \beta \in \Lambda^2(V^*) \colon \beta(LX, LY) = -\beta(X,Y) \right\} & \dim \LB \Lambda^{2,0}_L \RB & = 4n^2 - 2n \\
[ \Lambda^{1,1}_L ] & = \left\{ \beta \in \Lambda^2(V^*) \colon \beta(LX, LY) = \beta(X,Y) \right\} & \dim [ \Lambda^{1,1}_L ] & = 4n^2.
\end{align*}
 Note that there are relationships between the various subspaces $\LB \Lambda^{2,0}_L \RB$ and $[\Lambda^{1,1}_L]$ for $L \in \{I,J,K\}$.  For example, it is easy to check that $\omega_K \in \LB \Lambda^{2,0}_I \RB \cap \LB \Lambda^{2,0}_J \RB$, and in fact
\begin{align*}
\LB \Lambda^{2,0}_I \RB \cap \LB \Lambda^{2,0}_J \RB & = \LB \Lambda^{2,0}_J \RB \cap \LB \Lambda^{2,0}_K \RB \cap [\Lambda^{1,1}_{K}],
\end{align*}
and similarly for cyclic permutations of $(I,J,K)$.  Moreover, one can check that
\begin{equation} \label{eq:Tri-Primitivity}
[\Lambda^{1,1}_I] \cap [\Lambda^{1,1}_J] = [\Lambda^{1,1}_I] \cap [\Lambda^{1,1}_J] \cap [\Lambda^{1,1}_K] = [\Lambda^{1,1}_{I,0}] \cap [\Lambda^{1,1}_{J,0}] \cap [\Lambda^{1,1}_{K,0}].
\end{equation}
To prove the last equality, let $\beta \in [\Lambda^{1,1}_I] \cap [\Lambda^{1,1}_J] \cap [\Lambda^{1,1}_K]$, and write $\beta = \gamma + r\omega_I$ for some $\gamma \in [\Lambda^{1,1}_{I,0}]$ and $r \in \R$.  Since $\beta \in [\Lambda^{1,1}_J]$ and $r\omega_I \in \LB \Lambda^{2,0}_J \RB$, we have $r = \langle \beta, r\omega_I \rangle = 0$, so that $\beta \in [\Lambda^{1,1}_{I,0}]$. \\
\indent Now, using the hyper-Hermitian structure, we can refine (\ref{eq:Real-2forms}) even further.

\begin{prop} \label{prop:HyperHermitianDecomp} Let $V$ be a hyper-Hermitian vector space.  Then there is an $\Sp(n)$-irreducible decomposition
$$\Lambda^2(V^*) \cong U_I \oplus U_J \oplus U_K \oplus W \oplus \R \omega_I \oplus \R \omega_J \oplus \R \omega_K,$$
where
\begin{align*}
U_I & = [\Lambda^{1,1}_{0,I}] \cap \LB \Lambda^{2,0}_J \RB \cap \LB \Lambda^{2,0}_K \RB & W & =  [\Lambda^{1,1}_{0,I}] \cap  [\Lambda^{1,1}_{0,J}] \cap  [\Lambda^{1,1}_{0,K}]. \\
U_J & = \LB \Lambda^{2,0}_I \RB \cap [\Lambda^{1,1}_{0,J}] \cap \LB \Lambda^{2,0}_K \RB \\
U_K & = \LB \Lambda^{2,0}_I \RB \cap \LB \Lambda^{2,0}_J \RB \cap [\Lambda^{1,1}_{0,K}]
\end{align*}
Moreover,
\begin{align*}
\dim(U_I) = \dim(U_J) = \dim(U_K) & = 2n^2 - n - 1, & \dim(W) & = 2n^2 + n.
\end{align*}
\end{prop}

\begin{proof} First, note that every pairwise intersection among $\{U_I, U_J, U_K, W\}$ is equal to $\{0\}$, and so $U_I + U_J + U_K + W = U_I \oplus U_J \oplus U_K \oplus W$.  Next, set $Y := \mathrm{span}_\R(\omega_I, \omega_J, \omega_K)$, and let  $Y^\perp \subset \Lambda^2(V^*)$ be its orthogonal complement.  It is clear that if $\beta \in U_I \oplus U_J \oplus W$, then $\beta \in Y^\perp$.  Conversely, every $\beta \in Y^\perp$ can be written
\begin{align*}
\beta & = \beta_I + \beta_J + \beta_K + \beta_W,
\end{align*}
where
\begin{align*}
\beta_I & := \textstyle \frac{1}{4}( \beta + I^*\beta - J^*\beta - K^*\beta)  & \beta_W & := \textstyle \frac{1}{4}( \beta + I^*\beta + J^*\beta + K^*\beta). \\
\beta_J & := \textstyle \frac{1}{4}( \beta - I^*\beta + J^*\beta - K^*\beta)  \\
\beta_K & := \textstyle \frac{1}{4}( \beta - I^*\beta - J^*\beta + K^*\beta)
\end{align*}
It is straightforward to check that $\beta_L \in U_L$ for $L \in \{I,J,K\}$ and $\beta_W \in W$.  This proves that $U_I \oplus U_J \oplus U_K \oplus W = Y^\perp$. \\
\indent We now count dimensions.  Since $W = [\Lambda^{1,1}_I] \cap [\Lambda^{1,1}_J] \cap [\Lambda^{1,1}_K] = \mathrm{u}(V,I) \cap \mathrm{u}(V,J) \cap \mathrm{u}(V,K)$, where $\mathrm{u}(V,L)$ denotes the Lie algebra of the unitary group $\mathrm{U}(V, L)$, we see that $W \cong \mathfrak{sp}(n)$, and hence $\dim(W) = \binom{2n+1}{2} = 2n^2 + n$.  Now, since $\dim(U_I) = \dim(U_J) = \dim(U_K)$ for symmetry reasons, we find that
$$\textstyle 8n^2 - 2n = \binom{4n}{2} = \dim[\Lambda^2(V^*)] = 3 \dim(U_I) + \dim(W) + \dim(Y).$$
Substituting $\dim(W) = 2n^2 + n$ and $\dim(Y) = 3$ yields $\dim(U_I) = 2n^2 - n - 1$. \\
\indent Finally, to see that $U_I, U_J, U_K$ are $\Sp(n)$-irreducible, we argue as follows.  Let $\mathcal{E} \simeq \C^{2n} = \HH^n$ be the standard complex $\Sp(n)$-representation given by left-multiplication by quaternionic matrices, and let $\mathcal{H} \simeq \C^2 = \HH$ be the complex $\Sp(1)$-representation given by right-multiplication by unit quaternions.  Then it is well-known that there is an $\Sp(n)\Sp(1)$-equivariant isomorphism $U \cong [\Lambda^2_0(\mathcal{E}) \otimes \Sym^2(\mathcal{H})]$, from which we obtain  $\Sp(n)$-equivariant isomorphisms $U_I \cong U_J \cong U_K \cong [\Lambda^2_0(\mathcal{E})]$.  Since $[\Lambda^2_0(\mathcal{E})]$ is $\Sp(n)$-irreducible, this establishes the claim.
\end{proof}

\indent The following algebraic fact is useful when studying of $\Sp(n)$-instantons on complex line bundles.  To set notation, for a $2$-form $F \in \Lambda^2(V^*; \C)$, we let $F^{p,q}_L$ denote its type $(p,q)$-component with respect to $L \in \{I,J,K\}$.

\begin{lem} \label{lem:Spn-Equiv-Cond} Let $F \in \Lambda^2(V^*; i\R)$ be an $i\R$-valued $2$-form.  The following are equivalent:
\begin{enumerate}[(i)]
\item $F^{0,2}_J = 0$ and $F^{0,2}_K = 0$.
\item $F^{0,2}_I = 0$ and $F^{0,2}_J = 0$ and $F^{0,2}_K = 0$.
\item $F$ satisfies the following four equations:
\begin{align*}
F^{0,2}_J & = 0 & F^{0,2}_K & = 0 \\
F \wedge \omega_J^{2n-1} & = 0 & F \wedge \omega_K^{2n-1} & = 0
\end{align*}
\item $F$ satisfies the following six equations:
\begin{align*}
F^{0,2}_I & = 0 & F^{0,2}_J & = 0 & F^{0,2}_K & = 0 \\
F \wedge \omega_I^{2n-1} & = 0 & F \wedge \omega_J^{2n-1} & = 0 & F \wedge \omega_K^{2n-1} & = 0.
\end{align*}
\end{enumerate}
\end{lem}

\begin{proof} It is clear that (iv) $\implies$ (iii) $\implies$ (i), and that (iv) $\implies$ (ii) $\implies$ (i).  Thus, it remains to prove (i) $\implies$ (iv).  Suppose that $F$ satisfies (i).  Since $F$ is $i\R$-valued, we have $F^{2,0}_J = -\overline{F^{0,2}_J} = 0$ and $F^{2,0}_K = -\overline{F^{0,2}_K} = 0$, so that $F \in \Lambda^{1,1}_J \cap \Lambda^{1,1}_K$.  The result now follows from (\ref{eq:Tri-Primitivity}). 
\end{proof}

\subsubsection{Hyperk\"{a}hler Manifolds}

\indent \indent Let $X^{4n}$ be a hyper\"{a}hler manifold.  Since each tangent space $T_xX$ has a natural hyper-Hermitian structure, Proposition \ref{prop:HyperHermitianDecomp} implies that the bundle of $2$-forms on $X$ admits an $\Sp(n)$-invariant decomposition
$$\Lambda^2(T^*X) \cong U_I \oplus U_J \oplus U_K \oplus W \oplus \R \omega_I \oplus \R \omega_J \oplus \R \omega_K,$$
where now $W =  [\Lambda^{1,1}_{0,I}] \cap  [\Lambda^{1,1}_{0,J}] \cap  [\Lambda^{1,1}_{0,K}]$, along with $U_I, U_J, U_K$, etc., refer to vector bundles (rather than vector spaces). \\
\indent Now, let $E \to X$ be a Hermitian vector bundle.  An \emph{$\Sp(n)$-instanton} is a connection $A$ on $E$ satisfying either of the following equivalent conditions:
$$F_A \in \Gamma(\Lambda^2_{\mathfrak{sp}(n)} \otimes \mathrm{ad}_E) \ \iff \ F_A \in \Gamma(W \otimes \mathrm{ad}_E).$$
More explicitly, the linear algebraic discussion above provides the following characterizations:

\begin{prop} \label{prop:Sp(n)-equiv-def} Let $A$ be a connection on a Hermitian vector bundle $E \to X$.  The following are equivalent:
\begin{enumerate}[(i)]
\item $A$ is an $\Sp(n)$-instanton (i.e., $F_A \in \Gamma(W \otimes \mathrm{ad}_E)$).
\item $A$ is $\Omega$-ASD for $\Omega = \frac{1}{(2n-1)!}(\omega_I^2 + \omega_J^2 + \omega_K^2)^{n-1}$.
\item $A$ is simultaneously $\Omega_I$-, $\Omega_J$-, and $\Omega_K$-ASD, where $\Omega_L = \frac{1}{(2n-2)!}\omega_L^{2n-2} \in \Omega^{4n-4}(X)$.
\item $A$ is a primitive HYM connection with respect to $I$, $J$, and $K$:
\begin{align*}
(F_A)^{0,2}_I & = 0 & (F_A)^{0,2}_J & = 0 & (F_A)^{0,2}_K & = 0 \\
F_A \wedge \omega_I^{2n-1} & = 0 & F_A \wedge \omega_J^{2n-1} & = 0 & F_A \wedge \omega_K^{2n-1} & = 0.
\end{align*}
\end{enumerate}
Moreover, if $E$ is a Hermitian line bundle (i.e., if $\mathrm{rank}_\C(E) = 1$), then the above conditions are equivalent to:
\begin{enumerate}[(i)]  \setcounter{enumi}{4}
\item $(F_A)^{0,2}_J = 0$ and $(F_A)^{0,2}_K = 0$.
\end{enumerate}
\end{prop}

\subsubsection{Remark on Low Dimensions}

\indent \indent In low dimensions, $\Sp(n)$-instantons on hyperk\"{a}hler manifolds have relationships with other natural gauge-theoretic objects.  Indeed:

\begin{example}[Dimension $4$] Let $X^4$ be a hyperk\"{a}hler $4$-manifold, let $E \to X$ be a Hermitian vector bundle, and let $A$ be a connection on $E$.  Then
$$A \text{ is an }\Sp(1)\text{-instanton} \ \iff \ A \text{ is primitive Hermitian Yang-Mills} \ \iff \ A \text{ is ASD.}$$
In this way, $\Sp(n)$-instantons are seen to be natural generalizations of the classical ASD equations to $4n$ dimensions.
\end{example}

\begin{example}[Dimension $8$] \label{ex:Dim8} Let $X^8$ be a hyperk\"{a}hler $8$-manifold, let $E \to X$ be a Hermitian vector bundle, and let $A$ be a connection on $E$.  The hyperk\"{a}hler structure on $X$ induces a torsion-free $\Spin(7)$-structure $\Phi \in \Omega^4(X)$.  Then:
$$A \text{ is an }\Sp(2)\text{-instanton} \, \implies \, A \text{ is primitive Hermitian Yang-Mills} \, \implies \, A \text{ is a }\Spin(7)\text{-instanton.}$$
If $X$ is compact and $E \to X$ admits at least one $\Sp(2)$-instanton, then both converses hold: see Corollary \ref{cor:Lewis-Dimension8}.
\end{example}

\subsubsection{Remark on Bubbling Loci} \label{subsub:Bubbling}

\indent \indent We now briefly consider the bubbling of $\Sp(n)$-instantons on hyperk\"{a}hler manifolds.  In general, the bubbling loci of sequences of $\Omega$-ASD are $\Omega$-calibrated.  The algebraic reason for this is the following fact, whose proof can be found in \cite[Prop 2.88]{fadel2019gauge}:

\begin{prop} Equip $(\R^{4n}, g_0)$ with the flat metric, where $n \geq 2$.   Let $\R^{4n} = V \oplus \R^4$ be an orthogonal decomposition, and let $\pi \colon \R^{4n} \to \R^4$ be the corresponding projection. Let $E \to \R^4$ be a principal $G$-bundle, where $G$ is a compact Lie group, and let $A$ be a non-flat connection on $E$.  Diagrammatically:
$$\begin{tikzcd}
\pi^*E \arrow[d] \arrow[r]         & E \arrow[d]  \\
\mathbb{R}^{4n} \arrow[r, "\pi"] & \mathbb{R}^4
\end{tikzcd}$$
Let $\Omega \in \Lambda^{4n-4}(\R^{4n})^*$ be a calibration.  Then:
\begin{equation*}
\pi^*A \text{ is an }\Omega\text{-ASD instanton} \iff \begin{cases}  A \text{ is an ASD instanton, and} \\ V \subset \R^{4n} \text{ is } \pm\!\Omega\text{-calibrated.} \\ \end{cases}
\end{equation*}
\end{prop}

\indent Let $A$ be a non-flat connection on $E \to \R^4$.  By the above proposition, $\pi^*A$ is an $\Sp(n)$-instanton on $\pi^*E \to \R^{4n}$ if and only if $A$ is an ASD instanton and the $(4n-4)$-dimensional subspace $V \subset \R^{4n}$ is simultaneously calibrated by $\frac{1}{(2n-2)!}\omega_I^{2n-2}, \frac{1}{(2n-2)!}\omega_J^{2n-2}$, and $\frac{1}{(2n-2)!}\omega_K^{2n-2}$.  This last condition is equivalent to requiring that $V$ be a complex subspace with respect to $I$, $J$, and $K$. \\
\indent Accordingly, on a hyperk\"{a}hler manifold $X^{4n}$, we expect that the bubbling locus of a sequence of $\Sp(n)$-instantons admits the structure of a (possibly singular) $(4n-4)$-dimensional submanifold that is $I$-, $J$-, and $K$-complex.  However, it is well-known that such submanifolds (when smooth) are totally-geodesic \cite{gray1969note}, \cite[$\S$3]{alekseevsky2000almost}.  This suggests that sequences of $\Sp(n)$-instantons generically do not bubble.

\section{Complex Lagrangians and Gauge Theory}

\indent \indent In hyperk\"{a}hler $4n$-manifolds $(X, g, (I,J,K))$, one of the most important classes of submanifolds is the \emph{complex Lagrangians}, $2n$-dimensional submanifolds that are simultaneously complex with respect to $I$ and Lagrangian with respect to both $\omega_J$ and $\omega_K$.  They have been studied extensively by algebraic geometers, often in the context of fibrations of complex symplectic manifolds.  In this section, we make two observations relating complex Lagrangians in $X$ to gauge theoretic objects on bundles over $X$. \\
\indent First, in $\S$\ref{sub:RFM-ComplexLag}, we prove Theorem \ref{thm:RFM-Result}, which says that complex Lagrangian graphs correspond to $\Sp(n)$-instantons via the real Fourier-Mukai transform.  Second, in $\S$\ref{sub:Complex-Lag-dHYM}, we observe that when $X$ is equipped with a distribution of complex Lagrangian $2n$-planes, there exists a simple first-order PDE system whose solutions are deformed HYM connections.  Subsections \ref{sub:RFM-ComplexLag} and \ref{sub:Complex-Lag-dHYM} are independent from one another, and also from the rest of the paper.

\subsection{The Real Fourier-Mukai Transform of Complex Lagrangians} \label{sub:RFM-ComplexLag}

\begin{defn} Let $(X^{4n}, g, (I,J,K))$ be a hyperk\"{a}hler $4n$-manifold with symplectic forms $\omega_I, \omega_J, \omega_K$.  Consider the following three conditions on a $2n$-dimensional submanifold $L \subset X$:
\begin{enumerate}[(i)]
\item $L$ is complex with respect to $I$
\item $L$ is Lagrangian with respect to $\omega_J$
\item $L$ is Lagrangian with respect to $\omega_K$
\end{enumerate}
If $L$ satisfies any two of these conditions, then $L$ also satisfies the third.  In this case, $L$ is called an \emph{$I$-complex Lagrangian}.
\end{defn}

\begin{rmk} We remark that complex Lagrangian submanifolds are, in fact, special Lagrangians.  Indeed, recall that $X$ has complex volume forms $\Upsilon_J := \frac{1}{n!}(\omega_K + i\omega_I)^n$ and $\Upsilon_K = \frac{1}{n!}(\omega_I + i\omega_J)^n$.  If $L \subset X$ is an $I$-complex Lagrangian, then it is easy to check that $L$ is calibrated by both $\mathrm{Re}( (-i)^{n+1} \Upsilon_J)$ and $\mathrm{Re}(\Upsilon_K)$.  By definition, this means that $L$ is \emph{special Lagrangian} for $\Upsilon_J$ with phase $i^{n+1}$, and also special Lagrangian for $\Upsilon_K$ of phase $1$.
\end{rmk}

\indent Recall from $\S$\ref{subsub:Examples-Deformed} that the the real Fourier-Mukai (RFM) transform of a special Lagrangian graph $S_h \subset B^n \times T^n$ is a deformed HYM connection.  By analogy, we will say that the RFM transform of a complex Lagrangian graph is a \emph{deformed $\Sp(n)$-instanton}.  As we now prove, this deformed $\Sp(n)$-instanton equation coincides with the usual $\Sp(n)$-instanton equation:

\begin{thm} \label{thm:RFM} Let $B \subset \R^{2n}$ be an open set, and equip $X = B \times T^{2n}$ with a product hyperk\"{a}hler structure $(g, (I,J,K))$ such that $B \subset X$ is a complex Lagrangian with respect to $I$.  Let $h \colon B \to T^{2n}$ be a smooth function, let $S_h \subset X$ be its graph, and let $\nabla^h$ be its real Fourier-Mukai transform on $\underline{\C} \to B \times (T^n)^*$.  Then:
$$S_h \text{ is complex Lagrangian with respect to }I \ \iff \ \nabla^h \text{ is an }\Sp(n)\text{-instanton.}$$
\end{thm}

\indent This result is surprising for two reasons.  First, it stands in stark contrast to the situation for the other Ricci-flat holonomy groups ($H = \SU(n), \G_2, \Spin(7)$).  In each of those cases, the deformed $H$-instanton equation is more complicated than the ordinary $H$-instanton equation.  Second, since the complex Lagrangian condition is equivalent to the vanishing of the $2$-forms $\omega_J$ and $\omega_K$, one expects that the deformed $\Sp(n)$-instanton equation would be quadratic in the curvature of $\nabla$.  Despite this, Theorem \ref{thm:RFM} shows that the equation is \emph{linear} in the curvature.

\subsubsection{Setup}

\indent \indent We now prove Theorem \ref{thm:RFM}.  To begin, let $B \subset \R^{2n}$ be an open set, and let $T^{2n}$ be a $2n$-torus.  We introduce:
\begin{align*}
\text{coordinates }(z^1, \ldots, z^n) & = (x^1 + iy^1, \ldots, x^n + iy^n) \text{ on } B^{2n} \\
\text{coordinates }(w^1, \ldots, w^n) & = (u^1 + iv^1, \ldots, u^n + iv^n) \text{ on } T^{2n}
\end{align*}
We use index ranges $1 \leq i,j \leq n$ and $1 \leq a,b \leq n$.  Let us equip the product $X^{4n} = B \times T^{2n}$ with the flat metric $g = \sum (dx^a)^2 + (dy^a)^2 + (du^a)^2 + (dv^a)^2$ and complex structures $I,J,K$ defined by the following requirements:
\begin{align*}
I \!\left(\frac{\partial}{\partial x^a} \right) & = \frac{\partial}{\partial y^a} & J \!\left(\frac{\partial}{\partial x^a} \right) & = \frac{\partial}{\partial u^a} & K \!\left(\frac{\partial}{\partial x^a} \right) & = \frac{\partial}{\partial v^a} \\
I \!\left(\frac{\partial}{\partial u^a} \right) & = \frac{\partial}{\partial v^a} & J \!\left(\frac{\partial}{\partial v^a} \right) & = \frac{\partial}{\partial y^a} & K \!\left(\frac{\partial}{\partial y^a} \right) & = \frac{\partial}{\partial u^a}
\end{align*}
One can check that $IJ = K$.  Finally, we note that the induced K\"{a}hler forms $\omega_I, \omega_J, \omega_K$, defined by
\begin{align*}
\omega_I(X,Y) & = g(IX,Y) & \omega_J(X,Y) & = g(JX,Y) & \omega_K(X,Y) & = g(KX,Y)
\end{align*}
are given in coordinates by
\begin{align*}
\omega_I & = \sum dx^a \wedge dy^a + du^a \wedge dv^a, & \omega_J & = \sum dx^a \wedge du^a + dv^a \wedge dy^a, \\
& &  \omega_K & = \sum dx^a \wedge dv^a + dy^a \wedge du^a. 
\end{align*}
In particular, $\omega_J$ and $\omega_K$ contain no terms of the form $du^a \wedge dv^b$.  As we are about to demonstrate, the absence of such terms is the reason why the ``deformed $\Sp(n)$-instanton equation" is linear in the first derivatives (rather than quadratic).

\subsubsection{The graph $S_h$}

\indent \indent Let $h \colon B \to T^{2n}$ be a smooth function, say $h = (f^1, \ldots f^n, g^1, \ldots, g^n)$.  We denote the graph of $h$ by
\begin{align*}
S_h & = \{ (z, h(z)) \colon z \in B \} = \{ (x,y, f^a(x,y), g^a(x,y)) \colon (x,y) \in B\} \subset B \times T^{2n}.
\end{align*}
For ease of notation, we define functions
\begin{align*}
A_{aj} & = \frac{\partial f^a}{\partial x^j} & B_{aj} & = \frac{\partial f^a}{\partial y^j} & C_{aj} & = \frac{\partial g^a}{\partial x^j} & G_{aj} & = \frac{\partial g^a}{\partial y^j}.
\end{align*}
Then the tangent space to $S_h$ is spanned by the $2n$ vectors $X_1, \ldots, X_n, Y_1, \ldots, Y_n$ given by
\begin{align*}
X_j & := \frac{\partial}{\partial x^j} + \sum_{a} A_{aj} \frac{\partial}{\partial u^a} + \sum_a C_{aj} \frac{\partial}{\partial v^a}, & Y_j & := \frac{\partial}{\partial y^j} + \sum_{a} B_{aj} \frac{\partial}{\partial u^a} + \sum_a G_{aj} \frac{\partial}{\partial v^a}.
\end{align*}
Notice that
\begin{align*}
dx^a(X_i) & = \delta_{ai} & dy^a(X_i) & = 0 & du^a(X_i) & = A_{ai} & dv^a(X_i) & = C_{ai} \\
dx^a(Y_i) & = 0 & dy^a(Y_i) & = \delta_{ai} & du^a(Y_i) & = B_{ai}  & dv^a(Y_i) & = G_{ai}.
\end{align*}

\begin{lem} \label{lem:omega-comp} We have
\begin{align*}
\omega_J(X_i, X_j) & = A_{ij}  - A_{ji} & \omega_K(X_i, X_j) & = C_{ij} - C_{ji} \\
\omega_J(Y_i, Y_j) & =  G_{ji} - G_{ij} & \omega_K(Y_i, Y_j) & = B_{ij} - B_{ji} \\
\omega_J(X_i, Y_j) & = B_{ij} + C_{ji} & \omega_K(X_i, Y_j) & = G_{ij} - A_{ji}.
\end{align*}
\end{lem}

\begin{proof} Omitting the summation sign, we compute
\begin{align*}
\omega_J(X_i, X_j) & = dx^a(X_i) du^a(X_j) - dx^a(X_j) du^a(X_i) + dv^a(X_i) dy^a(X_j) - dv^a(X_i) dy^a(X_i)  \\
& = \delta_{ai} A_{aj} - \delta_{aj} A_{ai} + C_{ai} \cdot 0 - C_{aj} \cdot 0 \\
& = A_{ij} - A_{ji}
\end{align*}
and
\begin{align*}
\omega_J(Y_i, Y_j) & = dx^a(Y_i) du^a(Y_j) - dx^a(Y_j) du^a(Y_i) + dv^a(Y_i) dy^a(Y_j) - dv^a(Y_j) dy^a(Y_i) \\
& = 0 \cdot B_{aj} - 0 \cdot B_{ai} + G_{ai} \delta_{aj} - G_{aj}  \delta_{ai} \\
& = G_{ji} - G_{ij}
\end{align*}
and
\begin{align*}
\omega_J(X_i, Y_j) & = dx^a(X_i) du^a(Y_j) - dx^a(Y_j) du^a(X_i) + dv^a(X_i) dy^a(Y_j) - dv^a(Y_j) dy^a(X_i) \\
& = \delta_{ai} B_{aj} - 0 \cdot A_{ai} + C_{ai} \delta_{aj} - G_{aj} \cdot 0 \\
& = B_{ij} + C_{ji}
\end{align*}
Analogous computations yield the expressions for $\omega_K|_{S_h}$.
\end{proof}

\subsubsection{The connection $\nabla^h$}

\indent \indent By the discussion in $\S$\ref{sub:RFM}, the real Fourier-Mukai transform of $S_h$ is the connection $\nabla^h$ on the trivial complex line bundle $\underline{\C} \to X^* = B \times (T^{2n})^*$ given by
$$\nabla^h := d + i \sum_{a=1}^{n} f^a du^a + g^a dv^a.$$
Since the curvature of $\nabla^h$, denoted $F_{\nabla^h} \in \Omega^2(X^*; i \R)$, is an $i\R$-valued $2$-form, we will write $F_{\nabla^h} = iF$.  A computation shows that
\begin{equation}
F = \sum_{j=1}^n \sum_{a = 1}^n A_{aj}\,dx^j \wedge du^a + B_{aj}\,dy^j \wedge du^a + C_{aj}\,dx^j \wedge dv^a + G_{aj}\,dy^j \wedge dv^a.
\label{eq:F-Model}
\end{equation}
We now decompose $F$ into $(p,q)$-types with respect to $J$:

\begin{lem} \label{lem:J-Type} Let $F \in \Omega^2(X)$ be a $2$-form of the form (\ref{eq:F-Model}).  The following are equivalent:
\begin{enumerate}[(i)]
\item $F^{2,0}_J = 0$.
\item $F^{0,2}_J = 0$.
\item $F$ is of $J$-type $(1,1)$.
\item $A_{aj} = A_{ja}$ and $G_{aj} = G_{ja}$ and $B_{aj} = -C_{ja}$.
\end{enumerate}
\end{lem}

\begin{proof} Recall that $X^{4n}$ has coordinates $(z,w) = (x+iy, u+iv)$. Define a new coordinate system on $X^{4n}$ by
$$(\zeta, \eta) := (x+iu, y + iv).$$
These coordinates have the property that $d\zeta^i \wedge d\zeta^j$ and $d\zeta^i \wedge d\eta^j$ and $d\eta^i \wedge d\eta^j$ are all of $J$-type $(2,0)$.  Now, observe that
\begin{align*}
dx^a & = \textstyle \frac{1}{2}(d\zeta^a + d\overline{\zeta}^a) & du^a & = \textstyle \frac{1}{2i}(d\zeta^a - d\overline{\zeta}^a) \\
dy^a & = \textstyle \frac{1}{2i}(d\eta^a - d\overline{\eta}^a) & dv^a & = \textstyle \frac{1}{2}(d\eta^a + d\overline{\eta}^a)
\end{align*}
Substituting these expressions into the terms of (\ref{eq:F-Model}), a straightforward computation yields:
\begin{align*}
\sum_{a,j} A_{aj}\,dx^j \wedge du^a & = \frac{1}{4i}\left[ \sum_{a < j} \left(A_{aj} - A_{ja}\right) d\zeta^j \wedge d\zeta^a - \sum_{a < j} \left(A_{aj} - A_{ja}\right) d\overline{\zeta}^j \wedge d\overline{\zeta}^a \right.  \\
& \ \ \ \ - \left.\sum_{j} 2A_{jj}\,d\zeta^j \wedge d\overline{\zeta}^j - \sum_{a < j} \left(A_{aj} + A_{ja}\right) \left(d\zeta^a \wedge d\overline{\zeta}^j + d\zeta^j \wedge d\overline{\zeta}^a\right)  \right]
\end{align*}
and
\begin{align*}
\sum_{a,j} G_{aj}\,dy^j \wedge dv^a & = \frac{1}{4i}\left[ \sum_{a < j} \left(G_{aj} - G_{ja}\right) d\eta^j \wedge d\eta^a - \sum_{a < j} \left(G_{aj} - G_{ja}\right) d\overline{\eta}^j \wedge d\overline{\eta}^a \right.  \\
& \ \ \ \ + \left.\sum_{j} 2G_{jj}\,d\eta^j \wedge d\overline{\eta}^j + \sum_{a < j} \left(G_{aj} + G_{ja}\right) \left(d\eta^a \wedge d\overline{\eta}^j + d\eta^j \wedge d\overline{\eta}^a \right) \right] 
\end{align*}
and
\begin{align*}
\sum_{a,j} B_{aj}\,dy^j \wedge du^a + C_{aj}\,dx^j \wedge dv^a & = \frac{1}{4} \left[ \sum_{a,j} (B_{aj} + C_{ja}) \left(d\zeta^a \wedge d\eta^j + d\overline{\zeta}^a \wedge d\overline{\eta}^j \right) \right. \\
& \left. + \sum_{a,j} \left(B_{aj} - C_{ja}\right)\left(-d\zeta^a \wedge d\overline{\eta}^j + d\eta^j \wedge d\overline{\zeta}^a \right) \right]
\end{align*}
The result follows.
\end{proof}

We will also need to decompose $F$ into $(p,q)$-types with respect to $K$.

\begin{lem} \label{lem:K-Type} Let $F \in \Omega^2(X)$ be a $2$-form of the form (\ref{eq:F-Model}).  The following are equivalent:
\begin{enumerate}[(i)]
\item $F^{2,0}_K = 0$.
\item $F^{0,2}_K = 0$.
\item $F$ is of $K$-type $(1,1)$.
\item $B_{aj} = B_{ja}$ and $C_{aj} = G_{ja}$ and $A_{aj} = G_{ja}$.
\end{enumerate}
\end{lem}

\begin{proof} Recall that $X^{4n}$ has coordinates $(z,w) = (x+iy, u+iv)$. Define a new coordinate system on $X^{4n}$ by
$$(\alpha, \beta) := (x+iv, y + iu).$$
These coordinates have the property that $d\alpha^i \wedge d\alpha^j$ and $d\alpha^i \wedge d\beta^j$ and $d\beta^i \wedge d\beta^j$ are all of $K$-type $(2,0)$.  Now, observe that
\begin{align*}
dx^a & = \textstyle \frac{1}{2}(d\alpha^a + d\overline{\alpha}^a) & du^a & = \textstyle \frac{1}{2i}(d\beta^a - d\overline{\beta}^a) \\
dy^a & = \textstyle \frac{1}{2}(d\beta^a + d\overline{\beta}^a) & dv^a & = \textstyle \frac{1}{2i}(d\alpha^a - d\overline{\alpha}^a).
\end{align*}
Substituting these expressions into the terms of (\ref{eq:F-Model}), a tedious but straightforward computation yields the result exactly as in the proof of Lemma \ref{lem:J-Type}.
\end{proof}

\subsubsection{Completing the proof of Theorem \ref{thm:RFM}}

\indent \indent The condition that $S_h$ be complex Lagrangian with respect to $I$ is equivalent to requiring that both $\left.\omega_J\right|_{S_h} = 0$ and $\left.\omega_K\right|_{S_h} = 0$.  This is equivalent to requiring that for all $1 \leq i,j \leq n$, we have
\begin{align*}
\omega_J(X_i, X_j) & = 0 & \omega_K(X_i, X_j) & = 0 \\
\omega_J(Y_i, Y_j) & =  0 & \omega_K(Y_i, Y_j) & = 0 \\
\omega_J(X_i, Y_j) & = 0 & \omega_K(X_i, Y_j) & = 0.
\end{align*}
By Lemma \ref{lem:omega-comp}, this is equivalent to the first-order linear PDE system
\begin{align*}
A_{ij}  - A_{ji} & = 0 & C_{ij} - C_{ji} & = 0 \\
G_{ji} - G_{ij} & = 0 &  B_{ij} - B_{ji} & = 0 \\
B_{ij} + C_{ji} & = 0 &  G_{ij} - A_{ji} & = 0.
\end{align*}
By Lemmas \ref{lem:J-Type} and \ref{lem:K-Type}, we see that this system is equivalent to requiring that $F \in [\Lambda^{1,1}_J] \cap [\Lambda^{1,1}_K]$.  Finally, by Lemma \ref{lem:Spn-Equiv-Cond}, this last condition is equivalent to $\nabla^h$ being an $\Sp(n)$-instanton.  This completes the proof of Theorem \ref{thm:RFM}.

\subsection{Complex Lagrangian Distributions and dHYM Connections} \label{sub:Complex-Lag-dHYM}

\indent \indent Recall from $\S$\ref{subsub:Examples-Deformed} that a connection $A$ on a Hermitian line bundle $L \to X$ over a Calabi-Yau $2n$-manifold $(X^{2n}, ( g,\omega_I, I, \Upsilon_I))$ is a \emph{deformed HYM connection} if it satisfies
\begin{align}
(F_A)^{0,2} & = 0 \label{eq:dHYM1} \\
 \mathrm{Im}( (\omega_I + F_A)^{n}) & = 0. \label{eq:dHYM2}
\end{align}
As the PDE system (\ref{eq:dHYM1})-(\ref{eq:dHYM2}) is nonlinear, examples of dHYM connections are non-trivial to construct.  In this subsection, we point out that when $X$ is a \emph{hyperk\"{a}hler} manifold equipped with a complex Lagrangian distribution, there exists a considerably simpler PDE whose solutions are deformed Hermitian Yang-Mills. \\

\indent So, let $(X, g, (I,J,K))$ be a hyperk\"{a}hler $4n$-manifold equipped with a distribution $\mathsf{B} \subset TX$ of $I$-complex Lagrangian $2n$-planes.  That is, we assume that $X$ carries an orthogonal splitting $TX = \mathsf{B} \oplus \mathsf{F}$, where $\mathsf{B} \subset TX$ is a rank $2n$-subbundle consisting of $I$-complex Lagrangian $2n$-planes, and $\mathsf{F} := \mathsf{B}^\perp$.  Notice that the data of $\mathsf{B} \subset TX$ preferences a particular complex structure in the hyperk\"{a}hler triple (namely $I$).  In this context, the bundle of $2$-forms on $X$ decomposes into two subbundles of ``pure forms," and one subbundle of ``mixed forms":
$$\Lambda^2(T^*X) = \Lambda^2(\mathsf{B}^*) \oplus (\mathsf{B}^* \otimes \mathsf{F}^*) \oplus \Lambda^2(\mathsf{F}^*).$$
We will say that a $2$-form $F \in \Omega^2(X)$ is \emph{mixed} if $F \in \Gamma(\mathsf{B}^* \otimes \mathsf{F}^*)$.  This condition arises naturally in connection with the RFM transform:

\begin{example} Let $X = B \times T^{2n}$, where $B \subset \R^{2n}$ is an open subset, equip $X$ with the (flat) product hyperk\"{a}hler structure, and let $\mathsf{B} = TB$ and $\mathsf{F} = T(T^{2n})$.  Given a smooth function $h \colon B \to T^{2n}$, the RFM transformed connection $\nabla^h$ on the trivial line bundle $\underline{\C} \to B \times (T^{2n})^*$ has a curvature $2$-form $F_h$ that is mixed.
\end{example}

We now observe:

\begin{prop} Let $L \to X$ be a Hermitian line bundle over a hyperk\"{a}hler manifold $X$, and let $A$ be a connection on $L$.  Suppose that $X$ is equipped with a distribution $\mathsf{B} \subset TX$ of $I$-complex Lagrangian $2n$-planes. \\
\indent If $A$ is a $\U(2n)$-instanton (with respect to $I$) whose curvature $F_A$ is mixed (w.r.t. $I$), then $A$ is deformed Hermitian-Yang-Mills (w.r.t. $I$).
\end{prop}

\begin{proof} Let $A$ be a $\U(2n)$-instanton with respect to $I$ whose curvature $2$-form $F_A$ is mixed.  Since $A$ is a $\U(2n)$-instanton, we have $(F_A)^{0,2}_I = 0$, so it remains to verify (\ref{eq:dHYM2}).  Since $F_A \in \Omega^2(X; \mathfrak{u}(1)) \cong \Omega^2(X; i\R)$, we may write $F_A = iF$ for some real $2$-form $F \in \Omega^2(X)$.  Abbreviating $\omega := \omega_I$, we expand
\begin{equation*}
(\omega + iF)^{2n} = \textstyle \binom{2n}{0}\omega^{2n} + \binom{2n}{1}\omega^{2n-1} \wedge iF + \cdots + \binom{2n}{2n-1} \omega \wedge i^{2n-1}F^{2n-1} + \binom{2n}{2n} i^{2n} F^{2n}
\end{equation*}
so that
\begin{align*}
\mathrm{Im}[(\omega + iF)^{2n} ] & = \sum_{k=1}^n \textstyle (-1)^{k+1}\binom{2n}{2k-1}\, \omega^{2n-(2k-1)} \wedge  F^{2k-1}.
\end{align*}
Thus, $\mathrm{Im}[(\omega + iF)^{2n} ]$ is a linear combination of $(p+q)$-forms $\omega^p \wedge F^q$, where $p,q \geq 1$ are odd integers with $p+q = 2n$.  Now, since $\omega \in \Lambda^2(\mathsf{B}^*) \oplus \Lambda^2(\mathsf{F}^*)$, we have
$$\omega^p \in \Lambda^{2p}(\mathsf{B}^*) \oplus \left(\Lambda^{2p-2}(\mathsf{B}^*) \otimes \Lambda^2(\mathsf{F}^*)\right) \oplus \cdots \oplus \left(\Lambda^{2}(\mathsf{B}^*) \otimes \Lambda^{2p-2}(\mathsf{F}^*)\right) \oplus \Lambda^{2p}(\mathsf{F}^*).$$
On the other hand, since $F_A$ is mixed, we have $F^q \in \Lambda^q(\mathsf{B}^*) \otimes \Lambda^q(\mathsf{F}^*)$, and hence
$$\ast(F^q) \in \Lambda^{2n-q}(\mathsf{B}^*) \otimes \Lambda^{2n-q}(\mathsf{F}^*).$$
Thus, $\omega^p$ and $\ast(F^q)$ lie in orthogonal subspaces of $\Lambda^{2p}(T^*X)$, so that $\omega^p \wedge F^q = \langle \omega^p, \ast(F^q) \rangle\,\vol_X = 0$ for all odd integers $p,q \geq 1$ with $p+q = 2n$.  This verifies (\ref{eq:dHYM2}).
\end{proof}

\begin{rmk} In the language of $H$-structures, the (hyperk\"{a}hler) $\Sp(n)$-structure on $X$ together with the complex Lagrangian distribution $\mathsf{B} \subset TX$ constitutes a $\U(n)_\Delta$-structure, where $\U(n)_\Delta$ is the image of the usual diagonal embedding $\U(n) \hookrightarrow \Sp(n) \leq \SO(4n)$. 
\end{rmk}

\section{Gauge Theory on Calabi-Yau Cones and Sasaki-Einstein Links} \label{sec:CYCones-SasEin}

\indent \indent We now turn our attention to gauge theory on Calabi-Yau cones $C^{2n+2} = \mathrm{C}(M)$ with $n \geq 2$.  In Proposition \ref{prop:Cone-HYM}, we recall that conically singular pHYM connections over $C$ are modeled by contact instantons over the Sasaki-Einstein link $M$.  Then, in $\S$\ref{sub:DimReduction-Contact} we consider a dimensional reduction of the contact instanton equation.  More precisely, we recall in Proposition \ref{prop:Contact-Descent} how pHYM connections on the K\"{a}hler-Einstein twistor space $Z^{2n} = M/S^1$ yield simple examples of contact instantons on $M$. \\
\indent The results of this section are not new (see e.g., \cite{portilla2023instantons}, \cite{papoulias2022spin} for the case of $n = 3$).  However, they provide a useful framework for understanding the analogous results in the subsequent section.

\subsection{Calabi-Yau Cones and Sasaki-Einstein Links}

\indent \indent Let $C^{2n+2}$ be a Calabi-Yau cone, $n \geq 2$.  By the discussion in $\S$\ref{sec:H-struc}, $C$ carries an $\SU(n+1)$-structure $(g_C, J_C, \omega_C, \Upsilon_C)$, where $(g_C, J_C, \omega_C)$ is a K\"{a}hler structure and $\Upsilon_C \in \Omega^{n+1,0}(C)$ is a complex volume form.  Since $C$ is a cone, we have that
$$(C^{2n+2}, g_C) = (M^{2n+1} \times \R^+, dr^2 + r^2 g_M)$$
for some Sasaki-Einstein manifold $(M, g_M)$.  Going forward, we identify $M$ with the hypersurface $M \times \{1\} \subset C$. \\
\indent On $M$, we define a $1$-form $\alpha \in \Omega^1(M)$, a vector field $\xi \in \Gamma(TM)$, an endomorphism $\mathsf{J}_M \in \End(TM)$, and an $n$-form $\Psi \in \Omega^n(M; \C)$ as follows:
\begin{align*}
\alpha & := \left.\left( r\partial_r\,\lrcorner\,\omega_C \right)\right|_M & \xi & := \alpha^\sharp & \mathsf{J}_M & := \begin{cases}
J_C & \mbox{on } \mathrm{Ker}(\alpha) \\
0 & \mbox{on } \R \xi
\end{cases} & \Psi & := \left.\left( r\partial_r\,\lrcorner\,\Upsilon_C \right)\right|_M.
\end{align*}
We refer to $\alpha$ as the \emph{contact form}, and to $\xi$ as the \emph{Reeb field}.  The data $(g_M, \alpha, \mathsf{J}_M, \Psi)$ comprises an $(\SU(n) \times \{1\})$-structure.    \\
\indent Note that the tangent bundle of $M$ splits orthogonally as
\begin{align*}
TM & = \R \xi \oplus \mathsf{D}, & \mathsf{D} & := \Ker(\alpha),
\end{align*}
the endomorphism $\left.\mathsf{J}_M\right|_{\mathsf{D}} \colon \mathsf{D} \to \mathsf{D}$ is a $g_M$-orthogonal complex structure on $\mathsf{D}$, and  the restriction $\Psi|_{\mathsf{D}}$ is an $(n,0)$-form with respect to $\mathsf{J}_M|_{\mathsf{D}}$.  Altogether, the $2n$-plane field $\mathsf{D} \subset TM$ carries the Hermitian structure $(g_M, \mathsf{J}_M, \sigma)$, where $\sigma := g_M(\mathsf{J}_M \cdot, \cdot)$ is the non-degenerate $2$-form, as well as the complex volume form $\Psi|_{\mathsf{D}}$.  In particular, the metric $\left.g_M\right|_{\mathsf{D}}$ and orientation on $\mathsf{D}$ induce a natural Hodge star operator
\begin{align*}
\ast_{\mathsf{D}} \colon \Lambda^k(\mathsf{D}^*) & \to \Lambda^{2n-k}(\mathsf{D}^*).
\end{align*}
We will say that a $k$-form $\beta \in \Lambda^k(T^*M)$ is \emph{transverse} if $\beta \in \Lambda^k(\mathsf{D}^*) \subset \Lambda^k(T^*M)$. Note that if $\beta$ is transverse, then
\begin{equation}
\ast\! \beta = (\ast_{\mathsf{D}} \beta) \wedge \alpha. \label{eq:TransverseHodge}
\end{equation}
Finally, noting that $\omega_C = r\,dr \wedge \alpha + r^2\sigma$, we observe that $d\omega_C = 0$ implies $d\alpha = 2\sigma$.

\subsection{Contact Instantons on Sasaki-Einstein Manifolds} \label{subsec:Contact-Sas-Ein}

\indent \indent Let $M^{2n+1}$ be a Sasaki-Einstein manifold with $n \geq 2$ as above.  The bundle of $2$-forms on $M$ decomposes into $(\U(n) \times \{1\})$-irreducible subbundles as follows:
$$\Lambda^2(T^*M) = \left(\alpha \otimes \mathsf{D}^*\right) \oplus \LB \Lambda^{2,0} \RB \oplus [\Lambda^{1,1}_0] \oplus \R\sigma$$
where here
\begin{align*}
\LB \Lambda^{2,0} \RB & = \{ \beta \in \Lambda^2(\mathsf{D}^*) \colon \beta(\mathsf{J}_MX, \mathsf{J}_MY) = -\beta(X,Y) \text{ for all } X,Y \in \mathsf{D} \} \\
[ \Lambda^{1,1}_0 ] & = \{ \beta \in \Lambda^2(\mathsf{D}^*) \colon \beta(\mathsf{J}_MX, \mathsf{J}_MY) = \beta(X,Y) \text{ for all } X,Y \in \mathsf{D} \text{ and } \langle \beta, \sigma \rangle = 0 \}.
\end{align*}
We now define the notion of a (anti-self-dual) contact instanton.

\begin{prop} \label{prop:Contact-Instanton-Def} Let $A$ be a connection on a principal bundle $P \to M$.  We say that $A$ is a \emph{contact instanton} if either of the following equivalent conditions hold:
\begin{enumerate}[(i)]
\item $A$ is an $(\SU(n) \times \{1\})$-instanton (i.e., $F_A \in \Gamma([\Lambda^{1,1}_0] \otimes \mathrm{ad}_P)$).
\item $A$ is $\Omega$-ASD for $\Omega = \frac{1}{(n-2)!} \alpha \wedge \sigma^{n-2}$.
\end{enumerate}
\end{prop}

\begin{proof} Note that if $A$ satisfies (i), then $F_A$ is a transverse $\mathrm{ad}_P$-valued $2$-form.  Similarly, if $A$ satisfies (ii), then
$$\ast F_A = -\frac{1}{(n-2)!} \alpha \wedge \sigma^{n-2} \wedge F_A \in \Gamma( \R\alpha \otimes \Lambda^{2n-2}(\mathsf{D}^*) \otimes \mathrm{ad}_P),$$
and hence $F_A = \ast (\ast F_A)$ is a section of $\Lambda^2(\mathsf{D}^*) \otimes \mathrm{ad}_P$.  Thus, in either case, $F_A$ is transverse.  Therefore, we have
\begin{align*}
F_A \in \Gamma( [\Lambda^{1,1}_0] \otimes \mathrm{ad}_P) & \iff \ast_{\mathsf{D}} F_A = -\frac{1}{(n-2)!} \sigma^{n-2} \wedge F_A \\
& \iff \alpha \wedge \ast_{\mathsf{D}} F_A = -\frac{1}{(n-2)!} \alpha \wedge \sigma^{n-2} \wedge F_A \\
& \iff \ast F_A = -\frac{1}{(n-2)!} \alpha \wedge \sigma^{n-2} \wedge F_A,
\end{align*}
where the last equivalence uses (\ref{eq:TransverseHodge}).
\end{proof}

\begin{rmk} Let $\Omega = \frac{1}{(n-2)!} \alpha \wedge \sigma^{n-2} \in \Omega^{2n-3}(M)$.  One can check that the $\SU(n)$-equivariant linear operator $\Lambda^2(T_x^*M) \to \Lambda^2(T_x^*M)$ given by $\beta \mapsto \ast(\beta \wedge \Omega)$ has eigenvalues $0$, $1$, $-1$, and $(n-1)$, with respective eigenspaces $(\alpha \otimes \mathsf{D}^*)$, $\LB \Lambda^{2,0} \RB$, $[\Lambda^{1,1}_0]$, and $\R\sigma$.  In particular, the $(+1)$-eigenspace provides a notion of ``self-dual contact instantons," i.e., connections $A$ whose curvature is a section of $\LB \Lambda^{2,0} \RB \otimes \mathrm{ad}_P$.  See, e.g., \cite{baraglia2016moduli} and \cite{portilla2023instantons}.
\end{rmk}

\indent Contact instantons on a Sasaki-Einstein manifold $M$ are intimately related to primitive Hermitian Yang-Mills connections on its Calabi-Yau cone $C$.  To explain this, we need some preliminary notions.  Let $P \to C$ be a principal $G$-bundle over the cone $C = \R^+ \times M$, let $\mathfrak{g}$ be the Lie algebra of $G$, let $\iota_r \colon M \to C$ denote the inclusion maps $\iota_r(x) = (r,x)$, and abbreviate $\iota := \iota_1$.  We have a diagram:
$$\begin{tikzcd}
\iota_r^*P \arrow[d] \arrow[r] & P \arrow[d] \\
M \arrow[r, "\iota_r"', hook]  & C          
\end{tikzcd}$$
\indent Let $A \in \Omega^1(P; \mathfrak{g})$ be a connection on $P$.  Then along $\iota_r^*P$, we can write $A = B_r + h_r\,dr$, where $B_r \in \Omega^1(\iota_r^*P; \mathfrak{g})$ is a connection, and $h_r \in \Omega^0(\iota_r^*P; \mathfrak{g})$ is a $G$-equivariant function $h_r \colon \iota_r^*P \to \mathfrak{g}$.  We say that $A$ is in \emph{temporal gauge} if $h_r = 0$ for all $r$.  Note that if $A$ is a connection and $g_r$ is a gauge transformation satisfying $g_r^{-1} \frac{\partial}{\partial r} g_r = -h_r$, then the connection $\widetilde{A} = A + g^{-1}dg$ is gauge-equivalent to $A$ and in temporal gauge (cf. \cite[Lemma 1.21]{freed1995classical}, \cite[$\S$3]{stein20232}).  If $A$ is in temporal gauge, then the curvature of $A$ is given by
$$F_A = F_{B_r} - \partial_r B_r \wedge dr.$$
Finally, a connection $A$ in temporal gauge is called \emph{dilation-invariant} if $\partial_r B_r = 0$.

\begin{prop} \label{prop:Cone-HYM} Let $P \to C^{2n+2}$ be a principal $G$-bundle, and let $A$ be a dilation-invariant connection on $P$ in temporal gauge.  Then $A$ is a primitive Hermitian Yang-Mills connection on $P \to C$ if and only if $\iota^*A$ is a contact instanton on $\iota^*P \to M$.
\end{prop}

\begin{proof} Let $A$ be a dilation-invariant connection on $P$ in temporal gauge, and abbreviate $c_k := \frac{1}{k!}$.  Before beginning, we note that the formula $\omega_C = r\,dr \wedge \alpha + r^2\sigma$ gives
\begin{align*}
c_{n-1}\omega_C^{n-1} & = c_{n-2}\,r^{2n - 3}\,dr \wedge \alpha \wedge \sigma^{n-2} + c_{n-1}r^{2n-2}\,\sigma^{n-1}.
\end{align*}
\indent $(\Longrightarrow)$ Suppose that $A$ is primitive Hermitian Yang-Mills.  Then
\begin{align*}
dr \wedge (\ast_M F_{B_r}) = \ast_C F_A & = -c_{n-1}\omega_C^{n-1} \wedge F_A \\
& = -c_{n-2}\,r^{2n - 3}\,dr \wedge \alpha \wedge \sigma^{n-2} \wedge F_{B_r} - c_{n-1}r^{2n-2}\,\sigma^{n-1} \wedge F_{B_r}
\end{align*}
Contracting with $\frac{\partial}{\partial r}$, and then setting $r = 1$, we obtain:
\begin{align*}
\ast_M F_{B_1} = -c_{n-2}\, \alpha \wedge \sigma^{n-2} \wedge F_{B_1}
\end{align*}
Noting that $F_{B_1} = F_{\iota^*A}$ gives the result.  \\
\indent $(\Longleftarrow)$ Suppose that $\iota^*A$ is a contact instanton.  Since $F_{\iota^*A} = F_{B_1}$, we have
\begin{align*}
\ast_M F_{B_1} = -c_{n-2}\, \alpha \wedge \sigma^{n-2} \wedge F_{B_1}
\end{align*}
Dilation invariance, followed by wedging with $dr$, gives:
\begin{align*}
dr \wedge (\ast_M F_{B_r}) = -c_{n-2}\, r^{2n-3}dr \wedge \alpha \wedge \sigma^{n-2} \wedge F_{B_r}.
\end{align*}
Viewing $F_{B_r}$ as an $\mathrm{ad}_P$-valued $2$-form on $M$, we see that $F_{B_r}$ is transverse and primitive with respect to the non-degenerate $2$-form $\sigma$, so that $F_{B_r} \wedge \sigma^{n-1} = 0$.  Consequently,
\begin{align*}
\ast_C F_A = dr \wedge (\ast_M F_{B_r})   & = \ -c_{n-2}\,r^{2n - 3}\,dr \wedge \alpha \wedge \sigma^{n-2} \wedge F_{B_r} - c_{n-1}r^{2n-2}\,\sigma^{n-1} \wedge F_{B_r} \\
& = -c_{n-1}\omega_C^{n-1} \wedge F_A,
\end{align*}
showing that $A$ is primitive Hermitian Yang-Mills.
\end{proof}

\indent Contact instantons have been studied by both physicists \cite{kallen2012twisted} and mathematicians \cite{baraglia2016moduli}, \cite{portilla2023instantons}, \cite{papoulias2022spin}, \cite{loubeau2024weitzenb}.  On Sasaki-Einstein $7$-manifolds, contact instantons are related to $\G_2$-instantons, as we now recall.

\begin{example}[Dimension 7] \label{ex:SasEin7Mfld} Let $M^7$ be a Sasaki-Einstein $7$-manifold.  It is well-known that the Sasaki-Einstein structure on $M$ induces a natural nearly-parallel $\G_2$-structure $\varphi \in \Omega^3(M)$.  Let $A$ be a connection on a Hermitian vector bundle $E \to M$.  Then:
$$A \text{ is a contact instanton } \ \implies \ A \text{ is a }\G_2\text{-instanton.}$$
See, for example, \cite{papoulias2022spin} and \cite{portilla2023instantons}.
\end{example}

\subsection{Dimensional Reduction of Contact Instantons} \label{sub:DimReduction-Contact}

\indent \indent In this subsection, we work with Sasaki-Einstein $(2n+1)$-manifolds $M$ that are compact and quasi-regular.  The following result is well-known; see \cite[Theorems 7.1.3 and 11.1.3]{boyer2008sasakian}.

\begin{thm}[\cite{boyer2008sasakian}] Let $M$ be a compact, quasi-regular Sasaki-Einstein $(2n+1)$-manifold.  Let $\mathcal{F} \subset M$ denote the quasi-regular foliation defined by the Reeb field $\xi$.  Then:
\begin{enumerate}[(a)]
\item The (orbifold) leaf space $Z := M/\mathcal{F}$ admits a K\"{a}hler-Einstein structure $(g_{\mathrm{KE}}, J_{\mathrm{KE}}, \omega_{\mathrm{KE}})$ of positive scalar curvature such that the canonical projection $p \colon M \to Z$ is an orbifold Riemannian submersion.  Moreover, $p^*(\omega_{\mathrm{KE}}) = \sigma$.
\item The projection map $p \colon M \to Z$ is a principal $S^1$-orbibundle with connection $1$-form $\alpha$.
\end{enumerate}
\end{thm}

\begin{rmk} In general, $Z$ is an orbifold.  To avoid technical complications, we will always work over the smooth part of $Z$.
\end{rmk}

\indent Equipping $Z^{2n}$ with the K\"{a}hler-Einstein structure $(g_{\mathrm{KE}}, J_{\mathrm{KE}}, \omega_{\mathrm{KE}})$, its bundle of $2$-forms on decomposes into $\U(n)$-irreducible subbundles as follows:
\begin{equation*}
\Lambda^2(T^*Z) = \LB \Lambda^{2,0} \RB \oplus [\Lambda^{1,1}_0] \oplus \R \omega_{\mathrm{KE}}.
\end{equation*}

\begin{prop} \label{prop:Contact-Descent} Let $P \to Z^{2n}$ be a principal bundle, and let $A$ be a connection on $P$.  The following are equivalent:
\begin{enumerate}[(i)]
\item $A$ is a primitive Hermitian Yang-Mills connection on $P$.
\item $p^*A$ is a contact instanton on $p^*P$.
\end{enumerate}
Moreover, if $n = 3$ (so that $\dim(Z) = 6$ and $\dim(M) = 7$), then the above are equivalent to:
\begin{enumerate}[(i)]   \setcounter{enumi}{2}
\item $p^*A$ is a $\G_2$-instanton on $p^*P$.
\end{enumerate}
\end{prop}

\begin{proof} Let $F := F_A$ denote the curvature of $A$, and note that $p^*F = F_{p^*A}$ is the curvature of $p^*A$.  For ease of notation, let $c_n = \frac{1}{(n-2)!}$, and abbreviate $\omega := \omega_{\mathrm{KE}}$. \\
\indent (i) $\iff$ (ii).  Suppose $A$ is primitive Hermitian Yang-Mills.  Since $p^*F$ is a transverse $\mathrm{ad}_P$-valued $2$-form, we have 
$$\ast(p^*F) = \alpha \wedge \ast_{\mathsf{D}}(p^*F) = \alpha \wedge p^*(\ast F) = \alpha \wedge p^*(-c_n \omega^{n-2} \wedge F) = -c_n\,\alpha \wedge \sigma^{n-2} \wedge p^*F,$$
which shows that $p^*A$ is a contact instanton.   Conversely, suppose that $p^*A$ is a contact instanton.  Then
$$ \alpha \wedge p^*(\ast F)  = \alpha \wedge \ast_{\mathsf{D}}(p^*F) = \ast(p^*F) = -c_n\,\alpha \wedge \sigma^{n-2} \wedge p^*F = \alpha \wedge p^*(-c_n\,\omega^{n-2} \wedge F),$$
so contracting with the Reeb field gives $p^*(\ast F) = p^*(-c_n \omega^{n-2} \wedge F)$.  Since $p$ is a submersion, $p^*$ is injective, so $\ast F = -c_n \omega^{n-2} \wedge F$, whence $A$ is primitive Hermitian Yang-Mills. \\
\indent (ii) $\implies$ (iii).  This follows from Example \ref{ex:SasEin7Mfld}. \\
\indent (iii) $\implies$ (ii).  This direction is well-known, but we give a proof for completeness.  Suppose $p^*A$ is a $\G_2$-instanton on $p^*P \to M^7$, and let $\phi \in \Omega^3(M)$ be the induced nearly-parallel $\G_2$-structure on $M$, i.e.:
$$\phi = \alpha \wedge \sigma - \mathrm{Re}(\Psi).$$
Since $p^*A$ is a $\G_2$-instanton, we have
$$\ast(p^*F) = -\phi \wedge p^*F = -\alpha \wedge \sigma \wedge p^*F + \mathrm{Re}(\Psi) \wedge p^*F.$$
On the other hand, since $p^*F$ is transverse, it follows that $\ast(p^*F) = \alpha \wedge \psi$ for some $\psi \in \Omega^1(M; \mathrm{ad}_P)$.  Therefore, $\mathrm{Re}(\Psi) \wedge p^*F = 0$, which yields $\ast(p^*F) = -\alpha \wedge \sigma \wedge p^*F$, so that $p^*A$ is a contact instanton.
\end{proof}

\section{Gauge Theory on Hyperk\"{a}hler Cones and $3$-Sasakian Links} \label{sec:HKCones-3Sas}

\indent \indent In this section, we study conically singular $\Sp(n+1)$-instantons over hyperk\"{a}hler cones $C^{4n+4} = \Cone(M)$ with $n \geq 1$.  In Proposition \ref{prop:TriContact-Def} we define a natural gauge-theoretic condition --- the \emph{tri-contact instanton equation} --- for connections over $3$-Sasakian manifolds $M^{4n+3}$.  Then, in Proposition \ref{prop:Cone-Sp(n)}, we observe that conically singular $\Sp(n+1)$-instantons over $C$ are modeled by tri-contact instantons over $M$. \\
\indent This raises the question of how to construct tri-contact instantons.  To this end, we consider two dimensional reductions.  First, by quotienting $M$ by a particular Reeb field $\xi_1$, we reduce the tri-contact instanton equation on $M$ to the twistor space $Z^{4n+2} = M/\langle \xi_1 \rangle$.  The result (Proposition \ref{prop:Tri-Contact-MZ}) is that a particular class of pHYM connections on $Z$ yields tri-contact instantons on $M$.  Second, by quotienting $M$ by all three Reeb fields $\xi_1, \xi_2, \xi_3$ simultaneously, we may reduce further to the quaternionic-K\"{a}hler quotient $Q^{4n} = M/\langle \xi_1, \xi_2, \xi_3 \rangle$.  The result (Proposition \ref{prop:Tri-Contact-MZQ}) is that $\Sp(n)$-instantons on $Q$ yield pHYM connections on $Z$ of the relevant type, and hence also tri-contact instantons on $M$.

\subsection{Hyperk\"{a}hler Cones and $3$-Sasakian Links}

\indent \indent Let $C^{4n+4}$ be a hyperk\"{a}hler cone, $n \geq 1$.  By the discussion in $\S$\ref{sec:H-struc}, $C$ carries an $\Sp(n+1)$-structure $(g_C, (I_C, J_C, K_C))$, where $I_C, J_C, K_C$ are $g_C$-orthogonal complex structures satisfying $I_C J_C = K_C$.  We let $\omega_1 := g_C(I_C \cdot, \cdot)$, and similarly for $\omega_2, \omega_3$.  Since $C$ is a cone, we have that
$$(C^{4n+4}, g_C) = (M^{4n+3} \times \R^+, dr^2 + r^2 g_M)$$
for some $3$-Sasakian manifold $(M, g_M)$.  Going forward, we identify $M$ with the hypersurface $M \times \{1\} \subset C$. \\
\indent On $M$, we define $1$-forms $\alpha_1, \alpha_2, \alpha_3 \in \Omega^1(M)$, vector fields $\xi_1, \xi_2, \xi_3 \in \Gamma(TM)$, and endomorphisms $\mathsf{I}_M, \mathsf{J}_M, \mathsf{K}_M \in \End(TM)$ as follows:
\begin{align*}
\alpha_1 & := \left.\left( r\partial_r\,\lrcorner\,\omega_1 \right)\right|_M & \xi_1 & := \alpha_1^\sharp & \mathsf{I}_M & = \begin{cases}
I_C & \mbox{on } \mathrm{Ker}(\alpha_1) \\
0 & \mbox{on } \R \xi_1
\end{cases}
\end{align*}
and analogously for $\alpha_2, \alpha_3$ and $\xi_2, \xi_3$ and $\mathsf{J}_M, \mathsf{K}_M$. The data $(g_M, (\alpha_1, \alpha_2, \alpha_3), (\mathsf{I}_M, \mathsf{J}_M, \mathsf{K}_M))$ comprises an $(\Sp(n) \times \{1\}^3)$-structure on $M$. \\
\indent Note that the tangent bundle splits orthogonally as
\begin{align*}
TM & = \widetilde{\mathsf{V}} \oplus \widetilde{\mathsf{H}}, & \widetilde{\mathsf{V}} & = \R \xi_1 \oplus \R \xi_2 \oplus \R \xi_3 & \widetilde{\mathsf{H}} & = \Ker(\alpha_1, \alpha_2, \alpha_3).
\end{align*}
Moreover, the restricted endomorphisms $\mathsf{I}_M, \mathsf{J}_M, \mathsf{K}_M \colon \widetilde{\mathsf{H}} \to \widetilde{\mathsf{H}}$ are $g_M$-orthogonal complex structures that satisfy the quaternionic relations, so each $4n$-dimensional subspace $\widetilde{\mathsf{H}}_x \subset T_xM$ is hyper-Hermitian (recall $\S$\ref{subsub:LinAlgHK}). \\
\indent As in the previous section, we let $\sigma_1 := g_M(\mathsf{I}_M \cdot, \cdot)$, and similarly for $\sigma_2, \sigma_3$.  With respect to the decomposition 
$$\Lambda^2(T^*M) = \Lambda^2(\widetilde{\mathsf{V}}^*) \oplus (\widetilde{\mathsf{V}}^* \otimes \widetilde{\mathsf{H}}^*) \oplus \Lambda^2(\widetilde{\mathsf{H}}^*)$$
one can show that $\sigma_2$ (respectively, $\sigma_3$) has no component in $\widetilde{\mathsf{V}}^* \otimes \widetilde{\mathsf{H}}^*$, and that its $\Lambda^2(\widetilde{\mathsf{V}}^*)$-component is $\alpha_3 \wedge \alpha_1$ (respectively, $\alpha_1 \wedge \alpha_2)$.  Letting $\kappa_2$ and $\kappa_3$ denote the respective $\Lambda^2(\widetilde{\mathsf{H}}^*)$-components of $\sigma_2$ and $\sigma_3$, we have
\begin{align*}
\sigma_2 & = \alpha_3 \wedge \alpha_1 + \kappa_2 & \sigma_3 & = \alpha_1 \wedge \alpha_2 + \kappa_3.
\end{align*}

\subsection{Tri-Contact Instantons on $3$-Sasakian Manifolds}

\indent \indent Let $M^{4n+3}$ be a $3$-Sasakian manifold with $n \geq 1$ as above.   The bundle of $2$-forms
$$\Lambda^2(T^*M) = \Lambda^2(\widetilde{\mathsf{V}}^*) \oplus (\widetilde{\mathsf{V}}^* \otimes \widetilde{\mathsf{H}}^*) \oplus \Lambda^2(\widetilde{\mathsf{H}}^*)$$
decomposes further into $(\Sp(n) \times \{1\}^3)$-invariant subbundles as follows:
\begin{align*}
\Lambda^2(\widetilde{\mathsf{V}}^*) & = \R(\alpha_2 \wedge \alpha_3) \oplus \R(\alpha_3 \wedge \alpha_1) \oplus \R(\alpha_{I} \wedge \alpha_2) \\
\widetilde{\mathsf{V}}^* \otimes \widetilde{\mathsf{H}}^* & = (\alpha_1 \otimes \widetilde{\mathsf{H}}^*) \oplus (\alpha_2 \otimes \widetilde{\mathsf{H}}^*) \oplus (\alpha_3 \otimes \widetilde{\mathsf{H}}^*) \\
\Lambda^2(\widetilde{\mathsf{H}}^*) & = \widetilde{U}_I \oplus \widetilde{U}_J \oplus \widetilde{U}_K \oplus \widetilde{W} \oplus \R\kappa_1 \oplus \R\kappa_2 \oplus \R\kappa_3.
\end{align*}
In the last line, we are viewing $\widetilde{\mathsf{H}}$ as a hyper-Hermitian vector space and employing the notation of $\S$\ref{subsub:LinAlgHK}.  To be more explicit, for $\mathsf{L} \in \{\mathsf{I}, \mathsf{J}, \mathsf{K}\}$, we let $\Lambda^{p,q}_{\mathsf{L}} \subset \Lambda^{p+q}(\widetilde{\mathsf{H}}^*; \C)$ denote the subbundle of $(p+q)$-forms on $\mathsf{H}$ that have type $(p,q)$ with respect to $\mathsf{L}_M|_{\widetilde{\mathsf{H}}}$.  Then:
\begin{align*}
\widetilde{U}_I & = [\Lambda^{1,1}_{0,\mathsf{I}}] \cap \LB \Lambda^{2,0}_{\mathsf{J}} \RB \cap \LB \Lambda^{2,0}_{\mathsf{K}} \RB & \widetilde{W} & = [\Lambda^{1,1}_{0,\mathsf{I}}] \cap [\Lambda^{1,1}_{0,\mathsf{J}}] \cap [\Lambda^{1,1}_{0,\mathsf{K}}] \\
\widetilde{U}_J & = \LB \Lambda^{2,0}_{\mathsf{I}} \RB \cap [\Lambda^{1,1}_{0,\mathsf{J}}] \cap \LB \Lambda^{2,0}_{\mathsf{K}} \RB \\
\widetilde{U}_K & = \LB \Lambda^{2,0}_{\mathsf{I}} \RB \cap \LB \Lambda^{2,0}_{\mathsf{J}} \RB \cap [\Lambda^{1,1}_{0,\mathsf{K}}].
\end{align*}
In particular, these subbundles have the following ranks:
\begin{align*}
\mathrm{rank}(\widetilde{U}_I) = \mathrm{rank}(\widetilde{U}_J) = \mathrm{rank}(\widetilde{U}_K) & = 2n^2 - n - 1 & \mathrm{rank}(\widetilde{W}) & = 2n^2 + n.
\end{align*}
We now define the notion of a \emph{tri-contact instanton}.

\begin{prop} \label{prop:TriContact-Def} Let $A$ be a connection on a principal bundle $P \to M$.  We say that $A$ is a \emph{tri-contact instanton} if any of the following equivalent conditions hold:
\begin{enumerate}[(i)]
\item $A$ is an $(\Sp(n) \times \{1\}^3)$-instanton (i.e., $F_A \in \Gamma(\widetilde{W} \otimes \mathrm{ad}_P)$).
 \item $A$ is $\Omega$-ASD for $\Omega = \frac{1}{(2n-1)!} \alpha_1 \wedge \alpha_2 \wedge \alpha_3 \wedge (\kappa_1^2 + \kappa_2^2 + \kappa_3^2)^{n-1}$.
\item $A$ is simultaneously $\Omega_1$-, $\Omega_2$- and $\Omega_3$-ASD for $\Omega_p = \frac{1}{(2n-1)!}\alpha \wedge \sigma_p^{2n-1}$.
\end{enumerate}
\end{prop}

\begin{proof}  (i) $\iff$ (ii). Notice that both conditions (i) and (ii) imply that $F_A \in \Gamma( \Lambda^2(\widetilde{\mathsf{H}}^*) \otimes \mathrm{ad}_P)$, so it suffices to restrict attention to those connections $A$ whose curvature satisfies this horizontality condition.  If $A$ is such a connection, then:
\begin{align*}
\ast F_A  = -\Omega \wedge F_A & \iff \alpha_1 \wedge \alpha_2 \wedge \alpha_3 \wedge (\ast_{\widetilde{\mathsf{H} }} F_A) =  -\frac{1}{(2n-1)!} \alpha_1 \wedge \alpha_2 \wedge \alpha_3 \wedge (\kappa_1^2 + \kappa_2^2 + \kappa_3^2)^{n-1}\wedge F_A  \\
& \iff \ast_{\widetilde{\mathsf{H}}} F_A  = -\frac{1}{(2n-1)!}(\kappa_1^2 + \kappa_2^2 + \kappa_3^2)^{n-1}\wedge F_A \\
& \iff F_A \in \Gamma( \widetilde{W} \otimes \mathrm{ad}_P).
\end{align*}
The last equivalence follows from the observation that each $\widetilde{\mathsf{H}}_x \subset T_xM$ is a quaternionic-Hermitian vector space with fundamental $4k$-forms $\widetilde{\Pi}_k = \frac{1}{(2k+1)!}(\kappa_1^2 + \kappa_2^2 + \kappa_3^2)^{k}$ and using (\ref{eq:W-char}). \\
\indent (i) $\iff$ (iii).  This follows from Proposition \ref{prop:Contact-Instanton-Def}.
\end{proof}

\indent We now observe that tri-contact instantons on a $3$-Sasakian manifold $M$ are related to $\Sp(n+1)$-instantons on its hyperk\"{a}hler cone $C$.  Letting $\iota \colon M \hookrightarrow C$ denote the inclusion map $\iota(x) = (1,x)$, we have:

\begin{prop} \label{prop:Cone-Sp(n)} Let $P \to C$ be a principal $G$-bundle, and let $A$ be a dilation-invariant connection on $P$ that is in temporal gauge.  Then $A$ is an $\Sp(n+1)$-instanton on $P \to C$ if and only if $\iota^*A$ is a tri-contact instanton on $\iota^*P \to M$.
\end{prop}

\begin{proof} This follows immediately from Proposition \ref{prop:Cone-HYM} together with Propositions \ref{prop:Sp(n)-equiv-def} and \ref{prop:TriContact-Def}.
\end{proof}

\begin{example}[Dimension 7] If $M^7$ is a $3$-Sasakian $7$-manifold, then (recalling Example \ref{ex:SasEin7Mfld}), we have the implications:
$$A \text{ is a tri-contact instanton } \ \implies \ A \text{ is a contact instanton } \ \implies \ A \text{ is a }\G_2\text{-instanton.}$$
\end{example}

\subsection{Dimensional Reductions of Tri-Contact Instantons} \label{sub:DimRed}

\indent \indent In this subsection, we reduce the tri-contact instanton equation on $M^{4n+3}$ to a PDE for connections over the twistor space $Z^{4n+2}$ (cf. Proposition \ref{prop:Tri-Contact-MZ}).  We then reduce it even further to a PDE for connections over a positive quaternionic-K\"{a}hler $4n$-orbifold $Q^{4n}$ (cf. Proposition \ref{prop:Tri-Contact-MZQ}).  Our treatment makes crucial use of the canonical $\Sp(n)\U(1)$-structure on $Z$.  Since this structure is not widely known, we spend some time reviewing the relevant geometry and representation theory.

\subsubsection{The Structure of Compact $3$-Sasakian Manifolds} \label{subsub:Structure-Theorem}

\indent \indent From now on, we work with \emph{compact} $3$-Sasakian manifolds $M^{4n+3}$.  The structure of such manifolds is governed by the following well-known result: see \cite[Theorems 13.2.5, 13.3.1, 13.3.13]{boyer2008sasakian}.

\begin{thm}[\cite{boyer2008sasakian}] \label{thm:3Sasakian-structure-1} Let $M$ be a compact $3$-Sasakian $(4n+3)$-manifold.  For $v = (v_1, v_2, v _3) \in S^2$, let $\xi_v = v_1 \xi_1 + v_2 \xi_2 + v_3 \xi_3$ denote the corresponding Reeb field.  Then:
\begin{enumerate}[(a)]
\item Each $\xi_v$ defines a locally free $S^1$-action on $M$ and quasi-regular foliation $\mathcal{F}_v \subset M$.  Further, the leaf space $Z_v := M/\mathcal{F}_v$ is naturally a positive K\"{a}hler-Einstein $(4n+2)$-orbifold.
\item The projection $p_v \colon M \to Z_v$ is a principal $S^1$-orbibundle with connection $1$-form $\alpha_v = v_1 \alpha_1 + v_2 \alpha_2 + v_3 \alpha_3$, and is an orbifold Riemannian submersion.
\item Together, $\xi_1, \xi_2, \xi_3$ define a locally free $\SU(2)$-action on $M$ and quasi-regular foliation $\mathcal{F}_A \subset M$.  Further, the leaf space $Q := M/\mathcal{F}_A$ is naturally a positive quaternionic-K\"{a}hler $4n$-orbifold.
\item The projection $h \colon M \to Q$ is a principal $\SU(2)$- or $\SO(3)$-orbibundle and orbifold Riemannian submersion.
\item For $v,v' \in S^2$, there is a diffeomorphism $Z_v \approx Z_{v'}$.  In fact, each $Z_v$ may be identified with the (orbifold) twistor space $Z$ of $Q$. \\
\end{enumerate}
\end{thm}
\indent We reiterate that there is an entire $S^2$-family of projection maps $p_v \colon M \to Z$.  However, for definiteness, we will work only with $p := p_{(1,0,0)} \colon M \to Z$.   Letting $\tau \colon Z \to Q$ denote the twistor fibration, we can summarize the various spaces and maps under consideration in the following \emph{diamond diagram}:
$$\begin{tikzcd}
M^{4n+3} \arrow[dd, "h"'] \arrow[rd, "p"'] \arrow[r, hook] & C^{4n+4} \arrow[d]           \\
                                                             & Z^{4n+2} \arrow[ld, "\tau"'] \\
Q^{4n}                                                       &                             
\end{tikzcd}$$
As in $\S$\ref{sub:DimReduction-Contact}, we will always ignore the orbifold points of $Z$ and $Q$, and work only on the smooth loci.

\subsubsection{The Twistor Space: Revisited}

\indent \indent In $\S$\ref{sub:DimReduction-Contact}, we noted that if $M^{4n+3}$ is Sasaki-Einstein, then the twistor space $Z^{4n+2}$ admits a K\"{a}hler-Einstein structure.  It turns out that if $M$ is $3$-Sasakian, then $Z$ admits even more structure.  To explain this, we need the following notion:

\begin{defn}[\cite{aslan2023calibrated}] Let $Z^{4n+2}$ be a $(4n+2)$-manifold.  An \emph{$\Sp(n)\U(1)$-structure} on $Z$ consists of data $(g_Z, J_Z, \omega_Z, \mathsf{H}, \gamma)$, where:
\begin{itemize}
\item $(g_Z, J_Z, \omega_Z)$ is a $\U(2n+1)$-structure.
\item $\mathsf{H} \subset TZ$ is a distribution of $J_Z$-invariant $4n$-planes.  Setting $\mathsf{V} := \mathsf{H}^\perp$, we have a $J_Z$-invariant, $g_Z$-orthogonal splitting
$$TZ = \mathsf{V} \oplus \mathsf{H}.$$
With respect to this decomposition, we write $\omega_Z = \omega_{\mathsf{V}} + \omega_{\mathsf{H}}$.
\item $\gamma \in \Omega^3(Z;\C)$ is a complex $3$-form such that: At each $z \in Z$, there exists an orthonormal coframe
$$(e_{10}, e_{11}, e_{12}, e_{13}, \ldots, e_{n0}, e_{n1}, e_{n2}, e_{n3}, f_2, f_3) \colon T_zZ \to \R^{4n} \times \R^2$$
for which $\omega_{\mathsf{V}}|_z = f_2 \wedge f_3$ and
\begin{align*}
 \omega_{\mathsf{H}}|_z & = \sum_{j=1}^n (e_{j0} \wedge e_{j1} + e_{j2} \wedge e_{j3}), &  \gamma|_z & = (f_2 - if_3) \wedge \sum_{j=1}^n (e_{j0} + ie_{j1}) \wedge (e_{j2} + ie_{j3}).
\end{align*}
\end{itemize}
When the data $(g_Z, J_Z, \omega_Z, \mathsf{H})$ are understood, we sometimes abuse language and refer to $\gamma \in \Omega^3(Z;\C)$ itself as the $\Sp(n)\U(1)$-structure.  
\end{defn}

\indent The following fact was first proved by Alexandrov \cite{alexandrov2006sp}.  Recently, two alternative proofs were given in \cite{aslan2023calibrated}.

\begin{prop}[\cite{alexandrov2006sp}, \cite{aslan2023calibrated}] \label{prop:Sp(n)U(1)-Twistor} If $M$ is a compact $3$-Sasakian $(4n+3)$-manifold, then the twistor space $Z^{4n+2}$ inherits a canonical $\Sp(n)\U(1)$-structure $(g_\mathrm{KE}, J_{\mathrm{KE}}, \omega_{\mathrm{KE}}, \mathsf{H}, \gamma)$.  In fact, $\gamma \in \Omega^3(Z;\C)$ is the complex $3$-form for which $p^*(\gamma) = (\alpha_2 - i\alpha_3) \wedge (\kappa_2 + i\kappa_3)$.
\end{prop}

\indent It it shown in \cite[Prop 3.2.1]{aslan2023calibrated} that $Z$ also admits a canonical $4$-form $\xi \in \Omega^4(Z)$ defined by
$$p^*(\xi) = \kappa_2^2 + \kappa_3^2.$$
Moreover, one can show that $p^*(\omega_{\mathsf{V}}) = \alpha_2 \wedge \alpha_3$ and $p^*(\omega_{\mathsf{H}}) = \kappa_1$, from which one obtains
\begin{align*}
p^*(\omega_{\mathrm{KE}}) = p^*(\omega_{\mathsf{H}} + \omega_{\mathsf{V}}) = \alpha_2 \wedge \alpha_3 + \kappa_1.
\end{align*}

\subsubsection{Some $\Sp(n)\U(1)$-Representation Theory} \label{subsub:RepTheory}

\indent \indent We digress to discuss some aspects of $\Sp(n)\U(1)$ representation theory.  First, note that every irreducible complex $\Sp(n)\U(1)$-representation is of the form $\mathsf{F}\mathsf{L}^k := \mathsf{F} \otimes_{\C} \mathsf{L}^k$ for some $k \in \Z$, where $\mathsf{F}$ is an irreducible complex $\Sp(n)$-module, $\mathsf{L} = \C$ is the standard complex $\U(1)$-representation, and we are abbreviating $\mathsf{L}^k := \mathsf{L}^{\otimes k}$.  Moreover, if $k \neq 0$, then $\mathsf{F} \mathsf{L}^k$ is of complex type, and hence $\LB \mathsf{F} \mathsf{L}^k \RB = \LB \mathsf{F}\mathsf{L}^{-k} \RB$ is an irreducible \emph{real} $\Sp(n)\U(1)$-module. \\

\indent Let $\mathsf{T} = \mathsf{H} \oplus \mathsf{V} = \HH^n \oplus \C$ denote the following $\Sp(n)\U(1)$-representation: For $(A,\lambda) \in \Sp(n) \times \U(1)$ and $(h,v) \in \HH^n \oplus \C$, we define
\begin{equation}
(A, \lambda) \cdot (h, v) := (Ah\overline{\lambda}, \lambda^{-2}v). \label{eq:Sp(n)U(1)-Def}
\end{equation}
Since $(A,\lambda) \cdot (h,v) = (-A,-\lambda) \cdot (h,v)$, this $(\Sp(n) \times \U(1))$-action descends to an $\Sp(n)\U(1)$-action.  Note that this is the $\Sp(n)\U(1)$-representation implicit in the above definition of ``$\Sp(n)\U(1)$-structure."  We refer the reader to \cite{aslan2023calibrated} for details.

\begin{rmk} Here is a geometric interpretation.  Note that $\Sp(n+1)$ acts transitively on the set of complex lines in $\C^{2n+2}$ (i.e., on $\CP^{2n+1}$) with stabilizer $\Sp(n)\U(1)$.  The isotropy representation of $\Sp(n)\U(1)$ on the tangent space $T_z\CP^{2n+1} \simeq \R^{4n+2}$ is the $\Sp(n)\U(1)$-action given in (\ref{eq:Sp(n)U(1)-Def}).
\end{rmk}

\indent We now aim to decompose $\Lambda^2(\mathsf{T}^*; \C)$ into $\Sp(n)\U(1)$-irreducible submodules.  The result (Proposition \ref{cor:Lambda2-Twistor-Decomp}) requires some preparation.  First, following \cite[$\S$4.1]{aslan2023calibrated}, we identify $\HH^n \simeq \C^{2n}$ by writing $h \in \HH^n$ as $h = h_1 + jh_2$ with $h_1,h_2 \in \C^n$.  This identification endows $\HH^n$ with a complex structure (given by right-multiplication by $i$), thereby yielding an embedding $\iota \colon \Sp(n) \to \U(2n)$.  Consequently, the representation (\ref{eq:Sp(n)U(1)-Def}) induces an embedding
\begin{align*}
\Sp(n)\U(1) & \to \U(2n) \times \U(1) \\
(A,\lambda) & \mapsto (\iota(A)\lambda^{-1}, \lambda^{-2}).
\end{align*}
In this way, we view $\Sp(n)\U(1)$ as a subgroup of $\U(2n) \times \U(1)$. \\
\indent Next, we let $\omega_{\mathsf{H}} \in \Lambda^2(\mathsf{H}^*)$ denote the real non-degenerate $2$-form on $\mathsf{H}$ fixed by the $\U(2n)$-action, and analogously for $\omega_{\mathsf{V}} \in \Lambda^2(\mathsf{V}^*)$.  Letting $\langle \cdot, \cdot \rangle$ denote the standard inner product on $\mathsf{H} \oplus \mathsf{V} = \HH^n \oplus \C$, we obtain respective complex structures $J_{\mathsf{H}}, J_{\mathsf{V}}$ on $\mathsf{H}$ and $\mathsf{V}$ by $\omega_{\mathsf{H}} = \langle J_{\mathsf{H}} \cdot, \cdot \rangle$ and $\omega_{\mathsf{V}} = \langle J_{\mathsf{V}} \cdot, \cdot \rangle$.  With respect to these complex structures, we decompose the complexifications $\mathsf{H}^* \otimes_{\R} \C$ and $\mathsf{V}^* \otimes_{\R} \C$ into $(\pm i)$-eigenspaces in the usual fashion:
\begin{align*}
\mathsf{H}^* \otimes_{\R} \C & = \mathsf{H}^{1,0} \oplus \mathsf{H}^{0,1} & \mathsf{V}^* \otimes_{\R} \C & = \mathsf{V}^{1,0} \oplus \mathsf{V}^{0,1}.
\end{align*}

\indent Let $\mathsf{E} = \C^{2n}$ denote the standard complex $\Sp(n)$-representation, and (as above) we let $\mathsf{L} = \C$ denote the standard complex $\U(1)$-representation.  As $\Sp(n)\U(1)$-representations, we have:
\begin{align*}
\mathsf{H}^{1,0} & \cong \mathsf{E}\mathsf{L} & \mathsf{V}^{1,0} & \cong \mathsf{L}^2 \\
\mathsf{H}^{0,1} & \cong \mathsf{E}\mathsf{L}^{-1} & \mathsf{V}^{0,1} & \cong \mathsf{L}^{-2}.
\end{align*}
In view of the isomorphism $\Lambda^2(\mathsf{V}_1\mathsf{V}_2) \cong \Lambda^2(\mathsf{V}_1)\Sym^2(\mathsf{V}_2) \oplus \Sym^2(\mathsf{V}_1)\Lambda^2(\mathsf{V}_2)$ and recalling that $\dim_\C(\mathsf{L}^k) = 1$, we find that $\Lambda^2(\mathsf{E} \mathsf{L}^k) \cong \Lambda^2(\mathsf{E}) \mathsf{L}^{2k}$. Using this fact, we decompose
\begin{align} \label{eq:Lambda2-H10}
\Lambda^2(\mathsf{H}^{1,0}) = \Lambda^2(\mathsf{EL}) \cong \Lambda^2(\mathsf{E}) \mathsf{L}^2 \cong (\Lambda^2_0(\mathsf{E}) \mathsf{L}^2) \oplus \mathsf{L}^2.
\end{align}
Moreover:

\begin{lem} \label{lem:Rep-Theory} There are $\Sp(n)\U(1)$-equivariant isomorphisms
\begin{align*}
\mathsf{H}^{1,0} \otimes_{\C} \mathsf{H}^{0,1} & \cong \Sym^2(\mathsf{E}) \oplus \Lambda^2_0(\mathsf{E}) \oplus \C \omega_{\mathsf{H}} & \mathsf{H}^{1,0} \otimes_{\C} \mathsf{V}^{0,1} & \cong \mathsf{E}  \mathsf{L}^{-1} \\
\mathsf{V}^{1,0} \otimes_{\C} \mathsf{V}^{0,1} & \cong \C \omega_{\mathsf{V}} & \mathsf{V}^{1,0} \otimes_{\C} \mathsf{H}^{0,1} & \cong \mathsf{E} \mathsf{L}.
\end{align*}
 \end{lem}
 \begin{proof}  The first isomorphism follows by calculating
 $$\mathsf{H}^{1,0} \otimes_{\C} \mathsf{H}^{0,1} \cong \mathsf{E} \otimes \mathsf{E} \cong \Sym^2(\mathsf{E}) \oplus \Lambda^2(\mathsf{E}) \cong \Sym^2(\mathsf{E}) \oplus \Lambda^2_0(\mathsf{E}) \oplus \C\omega_{\mathsf{H}}.$$
 The others are immediate.
 \end{proof}

\begin{prop}[\cite{aslan2023calibrated}] \label{cor:Lambda2-Twistor-Decomp} Let $\Lambda^{p,q} \subset \Lambda^{p+q}(\mathsf{T}^*; \C)$ denote the space of $(p,q)$-forms on $\mathsf{T}$ with respect to the complex structure $J := J_{\mathsf{H}} + J_{\mathsf{V}}$.
\begin{enumerate}[(a)]
\item There are $\Sp(n)\U(1)$-irreducible decompositions
\begin{align*}
\Lambda^2(\mathsf{V}^*) & \cong \R \omega_{\mathsf{V}} \\
\mathsf{H}^* \otimes_{\R} \mathsf{V}^* & \cong \LB \mathsf{E}\mathsf{L}^3 \RB \oplus \LB \mathsf{E} \mathsf{L} \RB  \\
\Lambda^2(\mathsf{H}^*) & \cong \LB \Lambda^2_0(\mathsf{E}) \mathsf{L}^2 \RB \oplus [\Lambda^2_0(\mathsf{E})] \oplus [\Sym^2(\mathsf{E})] \oplus \R \omega_{\mathsf{H}} \oplus \LB \mathsf{L}^2 \RB
\end{align*}
\item There are $\Sp(n)\U(1)$-irreducible decompositions
\begin{align*}
\LB \Lambda^{2,0} \RB & \cong \LB \Lambda^2_0(\mathsf{E})\mathsf{L}^2 \RB \oplus \LB \mathsf{L}^2 \RB \oplus \LB \mathsf{E}\mathsf{L}^3 \RB \\
[\Lambda^{1,1}] & \cong   [\Lambda^2_0(\mathsf{E})] \oplus [\Sym^2(\mathsf{E})] \oplus \LB \mathsf{EL} \RB   \oplus \R \omega_{\mathsf{H}} \oplus \R \omega_{\mathsf{V}}.
\end{align*}
\end{enumerate}
\end{prop}
 \begin{proof} (a) The first of these is immediate.  The second follows from
 $$\mathsf{H}^* \otimes_{\R} \mathsf{V}^* = \LB \mathsf{H}^{1,0} \RB \otimes_{\R} \LB \mathsf{V}^{1,0} \RB \cong \LB \mathsf{H}^{1,0} \otimes_{\C} \mathsf{V}^{1,0} \RB \oplus \LB \mathsf{H}^{1,0} \otimes_{\C} \mathsf{V}^{1,0} \RB$$
 and then applying Lemma \ref{lem:Rep-Theory}.  The third follows from
 $$\Lambda^2(\mathsf{H}^*) = \LB \Lambda^{2,0}(\mathsf{H}) \RB \oplus [\Lambda^{1,1}(\mathsf{H})] \cong \LB \Lambda^2(\mathsf{H}^{1,0}) \RB \oplus [ \mathsf{H}^{1,0} \otimes_{\C} \mathsf{H}^{0,1} ]$$
 followed by (\ref{eq:Lambda2-H10}) and Lemma \ref{lem:Rep-Theory}. \\
 \indent (b) We note that
 \begin{align*}
\Lambda^{2,0} & \cong \Lambda^{2,0}(\mathsf{H}^{1,0}) \oplus (\mathsf{H}^{1,0} \otimes_{\C} \mathsf{V}^{1,0}) \\
 \Lambda^{1,1} & \cong \left( \mathsf{H}^{1,0} \otimes_{\C} \mathsf{H}^{0,1} \right)  \oplus \left( \mathsf{H}^{1,0} \otimes_{\C} \mathsf{V}^{0,1} \right)  \oplus \left( \mathsf{V}^{1,0} \otimes_{\C} \mathsf{H}^{0,1} \right) \oplus \left( \mathsf{V}^{1,0} \otimes_{\C} \mathsf{V}^{0,1} \right)
 \end{align*}
 and then apply (\ref{eq:Lambda2-H10}) and Lemma \ref{lem:Rep-Theory} as needed.
 \end{proof}
 
 \indent To conclude this digression, we make two more remarks.  First, we observe that $\mathsf{H}$ is a quaternionic-Hermitian vector space with fundamental $4k$-forms $\frac{1}{(2k+1)!}(\omega_{\mathsf{H}}^2 + \xi)^k \in \Lambda^{4k}(\mathsf{H}^*)$.  From this point of view, the $\Sp(n)\Sp(1)$-irreducible splitting $\Lambda^2(\mathsf{H}^*) = U_Z \oplus W_Z \oplus Y_Z$ is given by
 \begin{align*}
 U_Z & = \LB \Lambda^2_0(\mathsf{E})\mathsf{L}^2 \RB \oplus [\Lambda^2_0(\mathsf{E})] & W_Z & = [\Sym^2(\mathsf{E})] & Y_Z & = \R\omega_{\mathsf{H}} \oplus \LB \mathsf{L}^2 \RB.
 \end{align*}
In particular, by the discussion in $\S$\ref{subsub:Lin-Alg-QuatHerm},
\begin{align*}
\mathfrak{sp}(n) \cong [\Sym^2(\mathsf{E})] & = \textstyle \left\{ \beta \in \Lambda^2(\mathsf{H}^*) \colon \ast_{\mathsf{H}}\! \beta = -\beta \wedge \frac{1}{(2n-1)!}\, (\omega_{\mathsf{H}}^2 + \xi)^{n-1} \right\} \\
& = \textstyle \left\{ \beta \in \Lambda^2(\mathsf{T}^*) \colon \ast\!\beta = -\beta \wedge  \frac{1}{(2n-1)!}\, \omega_{\mathsf{V}} \wedge (\omega_{\mathsf{H}}^2 + \xi)^{n-1} \right\}\!.
\end{align*}
  Second, viewing $\mathsf{T}$ as a Hermitian vector space with respect to $J := J_{\mathsf{H}} + J_{\mathsf{V}}$, we have a $\U(2n+1)$-equivariant isomorphism $\mathfrak{u}(2n+1) \cong [\Lambda^{1,1}]$.  Moreover, one can check that $a\omega_{\mathsf{H}} + b\omega_{\mathsf{V}} \in [\Lambda^{1,1}]$ is primitive (with respect to $\omega = \omega_{\mathsf{H}} + \omega_{\mathsf{V}}$) if and only if $2na + b = 0$.  Consequently,
\begin{equation} \label{eq:su-identification}
\mathfrak{su}(2n+1) \cong [\Lambda^{1,1}_0] \cong [\Lambda^2_0(\mathsf{E})] \oplus [\Sym^2(\mathsf{E})] \oplus \LB \mathsf{EL} \RB   \oplus \R (\omega_{\mathsf{H}} - 2n\omega_{\mathsf{V}}).
\end{equation}

\subsubsection{Gauge Theory on Twistor Spaces}

\indent \indent We now consider connections over the $(4n+2)$-dimensional twistor space $Z$, equipped with the $\Sp(n)\U(1)$-structure $(g_{\mathrm{KE}}, J_{\mathrm{KE}}, \omega_{\mathrm{KE}}, \mathsf{H}, \gamma)$ of Proposition \ref{prop:Sp(n)U(1)-Twistor}.  Since $TZ = \mathsf{H} \oplus \mathsf{V}$, the bundle of $2$-forms decomposes as
\begin{align*}
\Lambda^2(T^*Z) = \Lambda^2(\mathsf{V}^*) \oplus (\mathsf{V}^* \otimes \mathsf{H}^*) \oplus \Lambda^2(\mathsf{H}^*).
\end{align*}
By Proposition \ref{cor:Lambda2-Twistor-Decomp}(a), these summands decompose further into $\Sp(n)\U(1)$-irreducible subbundles as follows:
\begin{align}
\Lambda^2(\mathsf{V}^*) & = \R \omega_{\mathsf{V}} \notag \\
\mathsf{V}^* \otimes \mathsf{H}^* & = \LB \mathsf{E} \mathsf{L}^3 \RB \oplus \LB \mathsf{E}\mathsf{L} \RB  \notag \\
\Lambda^2(\mathsf{H}^*) & = \LB \Lambda^2_0(\mathsf{E}) \mathsf{L}^2 \RB \oplus [\Lambda^2_0(\mathsf{E})] \oplus [\Sym^2(\mathsf{E})] \oplus \R \omega_{\mathsf{H}} \oplus \LB \mathsf{L}^2 \RB. \label{eq:Lambda2H-Decomp}
\end{align}
In particular, the components of $\Lambda^2(\mathsf{H}^*)$ have the following (real) ranks:
\begin{align*}
\mathrm{rank}  \LB \Lambda^2_0(\mathsf{E}) \mathsf{L}^2 \RB & = 4n^2 - 2n - 2 & \mathrm{rank}[\Sym^2(\mathsf{E})]   & = 2n^2+n & \mathrm{rank} \LB \mathsf{L}^2 \RB & = 2 \\
\mathrm{rank} [\Lambda^2_0(\mathsf{E})] & = 2n^2 - n - 1 & & & \mathrm{rank} (\R\omega_{\mathsf{H}}) & = 1.
\end{align*}
Note that when $n = 1$, the subbundles $\LB \Lambda^2_0(\mathsf{E})\mathsf{L}^2 \RB$ and $[\Lambda^2_0(\mathsf{E})]$ have rank $0$ (i.e., are not present in the decomposition).

\begin{rmk} It should be possible to describe the six subbundles $\LB \mathsf{E} \mathsf{L}^3 \RB$, $\LB \mathsf{E}\mathsf{L} \RB$, $\LB \Lambda^2_0(\mathsf{E}) \mathsf{L}^2 \RB$, $[\Lambda^2_0(\mathsf{E})]$, $[\Sym^2(\mathsf{E})]$, and $\LB \mathsf{L}^2 \RB$ in a concrete, explicit fashion.  For example, one can show that $\LB \mathsf{L}^2 \RB = \left\{ \iota_X(\mathrm{Re}(\gamma)) \in \Lambda^2(T^*Z) \colon X \in \mathsf{V} \right\}\!$.  We will not pursue this here.
\end{rmk}

\indent Let $P \to Z$ be a principal bundle, and let $A$ be a connection on $P$.  The decomposition (\ref{eq:Lambda2H-Decomp}) distinguishes several natural classes of connections on $P$.  For example, by  (\ref{eq:su-identification}) and Proposition \ref{prop:Contact-Descent}, we observe that:
\begin{align*}
A \text{ is an }\SU(2n+1)\text{-instanton} & \iff F_A \in \Gamma\!\left(  \left\{ \LB \mathsf{E}  \mathsf{L} \RB \oplus [\Lambda^2_0(\mathsf{E})] \oplus [\Sym^2(\mathsf{E})]  \oplus \R (\omega_{\mathsf{H}} - 2n\omega_{\mathsf{V}}) \right\} \otimes \mathrm{ad}_P \right) \\
& \iff A \text{ is primitive Hermitian Yang-Mills} \\
& \iff p^*A \text{ is a contact instanton on }p^*E.
\end{align*}
In particular, we find $2^4 = 16$ (resp. $2^3 = 8$) classes of primitive HYM connections on $Z$ when $n \geq 2$ (resp. $n = 1$).  One such consists of $\Sp(n)$-instantons:
\begin{align*}
A \text{ is an }\Sp(n)\text{-instanton} & \iff F_A \in \Gamma( [\Sym^2(\mathsf{E})] \otimes \mathrm{ad}_P) \\
& \iff A \text{ is }\Omega\text{-ASD for } \Omega \textstyle = \frac{1}{(2n-1)!}\,\omega_{\mathsf{V}} \wedge (\omega_{\mathsf{H}}^2 + \xi)^{n-1}
\end{align*}
Note that this notion of ``$\Sp(n)$-instanton" is not the one discussed in $\S$\ref{sec:Spn-Instantons}.  Indeed, while that section pertained to $\Sp(n)$-instantons on $4n$-manifolds with a QK structure, here our $\Sp(n)$-instantons reside on $(4n+2)$-manifolds with an $\Sp(n)\U(1)$-structure.  (Recall also Remark \ref{rmk:RepDep}.)

\indent We now observe that $\Sp(n)$-instantons on $Z$ are related to tri-contact instantons on $M$:

\begin{prop} \label{prop:Tri-Contact-MZ} Let $P \to Z^{4n+2}$ be a principal bundle, and let $A$ be a connection on $P$.  The following are equivalent:
\begin{enumerate}[(i)]
\item $A$ is an $\Sp(n)$-instanton on $P \to Z$.
\item $p^*A$ is a tri-contact instanton on $p^*P \to M$.
\end{enumerate}
\end{prop}

\begin{proof} Abbreviate $F = F_A$ and $c_n = \frac{1}{(2n-1)!}$.  Note that
\begin{align*}
A \text{ is an }\Sp(n)\text{-instanton} & \iff \ast F = -c_n\,\omega_{\mathsf{V}} \wedge (\omega_{\mathsf{H}}^2 + \xi)^{n-1} \wedge F \\
p^*A \text{ is a }\text{tri-contact instanton}  & \iff \ast(p^*F) = -c_n\,\alpha_1 \wedge \alpha_2 \wedge \alpha_3 \wedge (\kappa_1^2 + \kappa_2^2 + \kappa_3^2)^{n-1} \wedge p^*F.
\end{align*}
The result now follows from noting that
$$p^*(\omega_{\mathsf{V}} \wedge (\omega_{\mathsf{H}}^2 + \xi)^{n-1}) = \alpha_2 \wedge \alpha_3 \wedge (\kappa_1^2 + \kappa_2^2 + \kappa_3^2)^{n-1}$$
and repeating the argument used in the proof of Proposition \ref{prop:Contact-Descent} (i) $\iff$ (ii).
\end{proof}

\subsubsection{Gauge Theory on the Quaternionic-K\"{a}hler Quotient}

\indent \indent We now consider connections over $Q$.  Since $Q$ is quaternionic-K\"{a}hler, its bundle of $2$-forms decomposes as
$$\Lambda^2(T^*Q) = U \oplus W \oplus Y$$
where $\mathrm{rank}(U) = 6n^2 - 3n - 3$ and $\mathrm{rank}(W) = 2n^2 + n$ and $\mathrm{rank}(Y) = 3$, recalling the notation of $\S$\ref{sub:Spn-QK}.  Let
$$\textstyle \Pi_{k} := \frac{1}{(2k+1)!}(\beta_I^2 + \beta_J^2 + \beta_K^2)^k \in \Omega^{4k}(Q)$$
denote the fundamental $4k$-forms on $Q$, where $\{\beta_I, \beta_J, \beta_K\}$ is a local hyperk\"{a}hler triple on $Q$ compatible with quaternionic-K\"{a}hler structure. \\
\indent Now, since the twistor fibration $\tau \colon Z \to Q$ is a Riemannian submersion, the horizontal part of the derivative $\tau_*|_{\mathsf{H}} \colon \mathsf{H} \to TQ$ is an isometry.  In fact, it is an quaternionic-Hermitian isomorphism in the sense that
$$\tau^*(\Pi_1) =  \textstyle \frac{1}{6}(\omega_{\mathsf{H}}^2 + \xi).$$
Similarly, the map $h \colon M \to Q$ has the property that $h_*|_{\widetilde{\mathsf{H}}} \colon \widetilde{\mathsf{H}} \to TQ$ is a quaternionic-Hermitian isomorphism, satisfying $h^*(\Pi_1) = \frac{1}{6}(\kappa_1^2 + \kappa_2^2 + \kappa_3^2)$.

\begin{prop} \label{prop:Tri-Contact-MZQ} Let $P \to Q^{4n}$ be a principal bundle, and let $A$ be a connection on $P$.  The following are equivalent:
\begin{enumerate}[(i)]
\item $A$ is an $\Sp(n)$-instanton on $P \to Q$.
\item $\tau^*A$ is an $\Sp(n)$-instanton on $\tau^*P \to Z$.
\item $h^*A$ is a tri-contact instanton on $h^*P \to M$.
\end{enumerate}
Moreover, if $n = 1$ (so that $\dim(Q) = 4$, $\dim(Z) = 6$, and $\dim(M) = 7$), then the above are equivalent to:
\begin{enumerate}[(i)]   \setcounter{enumi}{3}
\item $h^*A$ is a $\G_2$-instanton on $h^*P \to M$.
\end{enumerate}
\end{prop}

\begin{proof} (i) $\iff$ (ii).  Abbreviate $F = F_A$ and $c_n = \frac{1}{(2n-1)!}$.  Note that
\begin{align*}
A \text{ is an }\Sp(n)\text{-instanton over } Q & \iff \ast F = -\Pi_{n-1} \wedge F \\
\tau^*A \text{ is an }\Sp(n)\text{-instanton over } Z & \iff \ast (\tau^*F) = -c_n\,\omega_{\mathsf{V}} \wedge (\omega_{\mathsf{H}}^2 + \xi)^{n-1} \wedge \tau^*F
\end{align*}
The result now follows from noting that
$$\tau^*\Pi_{n-1} = c_n\,(\omega_{\mathsf{H}}^2 + \xi)^{n-1}$$
and modifying the argument used in the proof of Proposition \ref{prop:Contact-Descent} (i) $\iff$ (ii). \\
\indent (ii) $\iff$ (iii).  This follows immediately from Proposition \ref{prop:Tri-Contact-MZ}. \\
\indent (iii) $\iff$ (iv) This follows from the corresponding statement in Proposition \ref{prop:Contact-Descent}.
\end{proof}

\section{Lewis-Type Theorems} \label{sec:LewisThms}

\indent \indent Let $X^{4n}$ be a hyperk\"{a}hler manifold, let $P \to X$ be a principal bundle admitting at least one $\Sp(n)$-instanton, and let $A$ be a connection on $P$.  By definition,
\begin{equation} \label{eq:Spn-implies-pHYM}
A \text{ is an }\Sp(n)\text{-instanton } \ \implies \ A \text{ is a primitive HYM connection w.r.t. }I.
\end{equation}
In this section, we investigate the converse. \\
\indent In $\S$\ref{sec:CompactCase}, we show that if $X$ is compact, then the converse of (\ref{eq:Spn-implies-pHYM}) indeed holds (see Theorem \ref{thm:Compact-Lewis}).  This can be deduced from results of Verbitsky \cite{verbitsky1993hyperholomorphic}, but we will provide a self-contained proof with an eye towards the non-compact setting.  Then, in $\S$\ref{sec:AC-Case}, we show that if $X$ is asymptotically conical of rate $\nu \leq -\frac{2}{3}(2n+1)$, then the converse of (\ref{eq:Spn-implies-pHYM}) again holds under suitable decay assumptions on $A$ (see Theorem \ref{thm:ACLewis}).  Results of this sort are sometimes called \emph{Lewis-type theorems} in reference to Lewis' study of $\Spin(7)$-instantons on Calabi-Yau $4$-folds \cite{lewis1999spin}.

\subsection{The Compact Case} \label{sec:CompactCase}

\indent \indent The purpose of this subsection is to prove the following result, which states that if a bundle $P$ over a compact hyperk\"{a}hler manifold admits at least one $\Sp(n)$-instanton, then every primitive $I$-HYM connection on $P$ is automatically $J$- and $K$-HYM as well.

\begin{thm} \label{thm:Compact-Lewis} Let $(X^{4n}, g, (I,J,K))$ be a compact hyperk\"{a}hler $4n$-manifold.  Let $P \to X$ be a principal $G$-bundle, where $G \leq \U(r)$ is a compact Lie group, that admits an $\Sp(n)$-instanton.  Let $A$ be a connection on $P$.  Then:
$$A \text{ is an }\Sp(n)\text{-instanton } \iff A \text{ is primitive HYM with respect to }I.$$
\end{thm}
Although this can be obtained from results of Verbitsky \cite{verbitsky1993hyperholomorphic}, we will provide a direct proof, as certain aspects will be needed in the asymptotically conical setting.   \\

\indent In dimension $8$, the above theorem implies that if an $\Sp(2)$-instanton exists, then the class of $\Spin(7)$-instantons coincides with the (\emph{a priori} smaller) class of $\Sp(2)$-instantons.  More precisely:

\begin{cor} \label{cor:Lewis-Dimension8} Let $(X^8, g, (I,J,K))$ be a compact hyperk\"{a}hler $8$-manifold.  Let $P \to X$ be a principal $G$-bundle, where $G \leq \U(r)$ is a compact Lie group, that admits an $\Sp(2)$-instanton.  Let $A$ be a connection on $P$.  Then:
$$A \text{ is an }\Sp(2)\text{-instanton } \iff A \text{ is primitive HYM w.r.t. } I \iff A \text{ is a }\Spin(7)\text{-instanton.}$$
\end{cor} 

\begin{proof} The first equivalence is Theorem \ref{thm:Compact-Lewis}.  For the second equivalence, note that the forward implication $(\Longrightarrow)$ is automatic (recall Example \ref{ex:Dim8}), while the reverse implication $(\Longleftarrow)$ is a theorem of Lewis \cite{lewis1999spin}.
\end{proof}

\subsubsection{A result from Hermitian Yang-Mills theory}

\indent \indent Let $(M^{2k}, g, L, \omega)$ be a compact K\"{a}hler manifold, where $L \colon TM \to TM$ denotes the complex structure and $\omega = g(L \cdot, \cdot)$ denotes the K\"{a}hler form.  Let $P \to M$ be a principal $G$-bundle, where $G \leq \U(r)$ is a compact Lie group.  Let $A$ be a connection on $P$, and let $F := F_A \in \Omega^2(M; \mathrm{ad}_P)$.  Recalling the decomposition $\Lambda^2(T^*M; \C) = \Lambda^{2,0} \oplus \Lambda^{0,2} \oplus \Lambda^{1,1}_0 \oplus \R\omega$, we split
$$F = F^{2,0} + F^{0,2} + F^{1,1}_0 + F_\omega.$$
The following lemma is well-known (e.g., it follows from formulas in \cite[$\S$1.3]{itoh1990yang}), but we provide a proof for completeness.

\begin{lem} \label{lem:HYM-Decomp} We have
$$ -2 \Vert F^{2,0} \Vert^2 + \Vert F^{1,1}_{0} \Vert^2 + (1-k) \Vert F_{\omega} \Vert^2 =  \int_M \mathrm{tr}(F \wedge F) \wedge \frac{\omega^{k-2}}{(k-2)!}$$
Moreover, the integral on the right side depends only on the topology of the bundle $P \to M$.
\end{lem}

\begin{proof} Set $\Theta = \frac{1}{(k-2)!}\omega^{k-2}$.  We calculate
\begin{align*}
F \wedge \Theta - \ast F & = \left(F^{2,0} + F^{0,2} + F^{1,1}_0 + F_\omega\right) \wedge \Theta - \left(F^{2,0} + F^{2,0} - F^{1,1}_0 + \frac{1}{k-1}F_\omega\right) \wedge \Theta \\
& = \left(2 F^{1,1}_0 + \frac{k-2}{k-1}F_\omega\right) \wedge \Theta,
\end{align*}
so that
$$F \wedge \Theta = \ast F + \left(2 F^{1,1}_0 + \frac{k-2}{k-1}F_\omega\right) \wedge \Theta,$$
and hence
\begin{equation}
F^2 \wedge \Theta = F \wedge \ast F + 2 F \wedge F^{1,1}_0 \wedge \Theta + \frac{k-2}{k-1}F \wedge F_\omega \wedge \Theta. \label{eq:F-squared-Theta}
\end{equation}
Regarding the last two terms, noting that $F^{1,1}_0 \wedge \Theta$ and $F_\omega \wedge \Theta$ are $(k-1,k-1)$-forms, degree considerations give
\begin{align*}
F \wedge F^{1,1}_0 \wedge \Theta & = F^{1,1}_0 \wedge F^{1,1}_0 \wedge \Theta = -F^{1,1}_0 \wedge \ast F^{1,1}_0 \\
F \wedge F_\omega \wedge \Theta & = F_\omega \wedge F_\omega \wedge \Theta = (k-1) F_\omega \wedge \ast F_\omega.
\end{align*}
Using these, (\ref{eq:F-squared-Theta}) becomes
\begin{align*}
F^2 \wedge \Theta & = F \wedge \ast F - 2F^{1,1}_0 \wedge \ast F^{1,1}_0 + (k-2)\,F_\omega \wedge \ast F_\omega,
\end{align*}
and so
\begin{align*}
\mathrm{tr}(F^2) \wedge \Theta & = \mathrm{tr}(F \wedge \ast F) + \left( 2 |F^{1,1}_0|^2 + (2-k) |F_\omega|^2 \right)\mathrm{vol}_M.
\end{align*}
Finally, using $\mathrm{tr}(F \wedge \ast F) = -|F|^2\,\mathrm{vol} = \left( -2 |F^{2,0}|^2 - |F^{1,1}_0|^2 - |F_\omega|^2\right) \mathrm{vol}_M$ and integrating gives the desired formula.
\end{proof}

\begin{cor} If $A$ is a primitive Hermitian Yang-Mills connection, then $\Vert F \Vert^2 = \Vert F^{1,1}_0 \Vert^2$ is a topological quantity.
\end{cor}

\subsubsection{Proof of Theorem \ref{thm:Compact-Lewis}} \label{subsub:Proof-Cpt-Lewis}

\indent \indent As in the theorem statement, we let $(X^{4n}, g, (I,J,K))$ be a compact hyperk\"{a}hler $4n$-manifold, and let $P \to X$ be a principal $G$-bundle that admits an $\Sp(n)$-instanton.  Let $A$ be a connection on $P$, and let $F := F_A \in \Omega^2(X; \mathrm{ad}_P)$ denote its curvature $2$-form.  By Proposition \ref{prop:HyperHermitianDecomp}, we have a decomposition
\begin{equation} \label{eq:Spn-Split}
\Lambda^2(T^*X) = U_I \oplus U_J \oplus U_K \oplus W \oplus \R\omega_I \oplus \R\omega_J \oplus \R\omega_K
\end{equation}
so that we may split
$$F = F_I + F_J + F_K + F_W + \mu_I \omega_I + \mu_J \omega_J + \mu_K \omega_K$$
where $\mu_I, \mu_J, \mu_K \in \Gamma(\mathrm{ad}_P)$.  Observe that
\begin{align*}
F \text{ is an }I\text{-primitive HYM connection} & \iff F = F_I + F_W \\
F \text{ is an }\Sp(n)\text{-instanton} & \iff F =  F_W
\end{align*}
\indent Now, for $\mathbf{\lambda} = (\lambda_1, \lambda_2, \lambda_3) \in \R^3$ satisfying the conditions
\begin{align}
\lambda_1 + \lambda_2 + \lambda_3 & = 0 & \lambda_1 - 2\lambda_2 - 2\lambda_3 & > 0, \label{eq:lambda-conditions}
\end{align}
define a $(4n-4)$-form $\Phi_\lambda \in \Omega^{4n-4}(X)$ by
$$\Phi_\lambda = \textstyle \frac{1}{(2n-2)!}\left( \lambda_1 \omega_I^{2n-2} + \lambda_2 \omega_J^{2n-2} + \lambda_3 \omega_K^{2n-2} \right)\!,$$
and define constants $a_I, a_J, a_K$ and $b_I, b_J, b_K$ by
\begin{align}
\begin{bmatrix}
a_I \\ a_J \\ a_K \end{bmatrix} & =
\begin{bmatrix}
1 & -2 & - 2 \\
-2 & 1 & -2 \\
-2 & -2 & 1 \end{bmatrix} \begin{bmatrix} \lambda_1 \\ \lambda_2 \\ \lambda_3 \end{bmatrix}
&
\begin{bmatrix} b_I \\ b_J \\ b_K \end{bmatrix}  =
\begin{bmatrix}
1-2n & -2 & - 2 \\
-2 & 1-2n & -2 \\
-2 & -2 & 1-2n \end{bmatrix} \begin{bmatrix} \lambda_1 \\ \lambda_2 \\ \lambda_3 \end{bmatrix}\!.
\label{eq:AB-constants}
\end{align}
Note that $a_I > 0$ by assumption.

\begin{lem} \label{lem:Phi-lambda-topological} For each $\lambda \in \R^3$ satisfying (\ref{eq:lambda-conditions}), we have
$$\sum_{L \in \{I,J,K\}} a_L \left\Vert F_L \right\Vert^2 + b_L \left\Vert \mu_L \omega_L \right\Vert^2 = \int_X \mathrm{tr}(F \wedge F) \wedge \Phi_\lambda.$$
(Note that there is no $F_W$ term on the left side.)  Moreover, the integral on the right side depends only on the topology of the bundle $P \to X$.
\end{lem}

\begin{proof} Fix $\lambda \in \R^3$ satisfying (\ref{eq:lambda-conditions}).  For $L \in \{I,J,K\}$, we set $\Theta_L := \frac{\omega_L^{2n-2}}{(2n-2)!}$.  By the proof of Lemma \ref{lem:HYM-Decomp}, we have
\begin{align} \label{eq:Set1}
\mathrm{tr}(F \wedge F) \wedge \Theta_I & = \left( -2 | F^{2,0}_I |^2 + | F^{1,1}_{I,0} |^2 + \left(1-2n\right) | \mu_I \omega_I |^2\right) \vol_X \notag \\
\mathrm{tr}(F \wedge F) \wedge \Theta_J & = \left( -2 | F^{2,0}_J |^2 + | F^{1,1}_{J,0} |^2 + \left(1-2n\right) | \mu_J \omega_J |^2\right) \vol_X \\
\mathrm{tr}(F \wedge F) \wedge \Theta_K & = \left( -2 | F^{2,0}_K |^2 + | F^{1,1}_{K,0} |^2 + \left(1-2n \right) | \mu_K \omega_K |^2\right) \vol_X. \notag
\end{align}
We now decompose the first two terms on the right hand sides with respect the splitting (\ref{eq:Spn-Split}).  The result is
\begin{align} \label{eq:Set2A}
| F^{2,0}_I |^2 & =  | F_J |^2 + | F_K |^2 + | \mu_J \omega_J |^2 + | \mu_K \omega_K |^2  \notag \\
| F^{2,0}_J |^2 & = | F_I |^2 + | F_K |^2 + | \mu_I \omega_I |^2 + | \mu_K \omega_K |^2  \\
| F^{2,0}_K |^2 & = | F_I |^2 + | F_J |^2 + | \mu_I \omega_I |^2 + | \mu_J \omega_J |^2  \notag
\end{align}
and
\begin{align} \label{eq:Set2B}
| F^{1,1}_{0,L} |^2 & = | F_L |^2 + | F_W |^2.
\end{align}
For brevity, we introduce the following shorthand: For constants $p_I, p_J, p_K, r, q_I, q_J, q_K \in \R$, we let
\begin{align*}
[p_I, p_J, p_K;\ r;\ q_I, q_J, q_K] & := \sum_{L \in \{I,J,K\}} p_L | F_L |^2 + r | F_W |^2 +  \sum_{L \in \{I,J,K\}} q_L | \mu_L \omega_L |^2.
\end{align*}
Then using (\ref{eq:Set2A}) and (\ref{eq:Set2B}), we can write equations (\ref{eq:Set1}) as follows:
\begin{align}
\mathrm{tr}(F \wedge F) \wedge \Theta_I & = \left[ 1, -2, -2;\  1;\ 1-2n, -2, -2 \right] \vol_X \notag \\
\mathrm{tr}(F \wedge F) \wedge \Theta_J & = \left[ -2, 1, -2;\ 1;\ -2, 1-2n, -2 \right] \vol_X \\
\mathrm{tr}(F \wedge F) \wedge \Theta_K & = \left[ -2, -2, 1;\ 1;\ -2, -2, 1-2n \right] \vol_X \notag
\end{align}
Therefore, since $\lambda_1 + \lambda_2 + \lambda_3 = 0$, we have that
$$\mathrm{tr}(F \wedge F) \wedge \Phi_\lambda = [ a_I, a_J, a_K; \ 0; \ b_I, b_J, b_K ]\,\vol_X.$$
\end{proof}

\begin{cor} \label{cor:Primitive-Topological} If $A$ is an $I$-primitive HYM connection, then $\Vert F_{I} \Vert$ depends only on the topology of the bundle $P \to X$.
\end{cor}

\begin{proof} If $A$ is an $I$-primitive HYM connection, then $F_J = F_K = 0$ and $\mu_I = \mu_J = \mu_K = 0$.  Therefore, by Lemma \ref{lem:Phi-lambda-topological}, the quantity
$$a_I \left\Vert F_I \right\Vert^2 = \int_X \mathrm{tr}(F \wedge F) \wedge \Phi_\lambda$$
is topological, and $a_I$ is a positive constant.
\end{proof}

\indent We are now in a position to prove Theorem \ref{thm:Compact-Lewis}.  Suppose that $A$ is a primitive Hermitian Yang-Mills connection with respect to $I$, so that its curvature $2$-form decomposes as $F = F_I + F_W$.  By assumption, $P$ admits an $\Sp(n)$-instanton.  Let $\widetilde{A}$ denote such an $\Sp(n)$-instanton, and let $\widetilde{F}$ denote its curvature $2$-form.  Since both $A$ and $\widetilde{A}$ are $I$-primitive HYM connections, Corollary \ref{cor:Primitive-Topological} implies that
$$\Vert F_I \Vert = \Vert \widetilde{F}_I \Vert.$$
But since $\widetilde{A}$ is an $\Sp(n)$-instanton, we have $\widetilde{F} = \widetilde{F}_W$, meaning that $\widetilde{F}_I = 0$.  Therefore, $F_I = 0$, so that $F = F_W$, and hence $A$ is an $\Sp(n)$-instanton.  This completes the proof.

\subsection{The Asymptotically Conical Case} \label{sec:AC-Case}

\indent \indent This section is aimed at stating and proving Theorem \ref{thm:ACLewis}, a Lewis-type theorem for instantons on asymptotically conical (AC) hyperk\"{a}hler manifolds $X$.  In essence, the idea is to adapt Lemma \ref{lem:Phi-lambda-topological} and Corollary \ref{cor:Primitive-Topological} to the AC setting.  This is accomplished in Proposition \ref{prop:CS-Identity}, which, by applying Stokes' Theorem to truncations of $X$ at radius $r > 0$, features a Chern-Simons type functional $\mathrm{CS}_\lambda[A]$ that measures the failure of a pHYM connection $A$ from being an $\Sp(n)$-instanton.  By estimating $\mathrm{CS}_\lambda[A]$ at infinity, we show that pHYM connections with sufficiently fast decay will satisfy $\mathrm{CS}_\lambda[A] = 0$ for large values of $r$, which implies the desired Lewis theorem. \\
\indent Our approach follows that of Papoulias \cite{papoulias2022spin}, who obtains an analogous result for AC $\Spin(7)$-manifolds.  Preparation of this section was also aided by helpful discussions in \cite{conlon2011construction}, \cite{driscoll2019deformations}, and \cite{ghosh2024deformation}.

\subsubsection{Asymptotically conical Riemannian manifolds}

\begin{defn} Let $(X^N, h_X)$ be a Riemannian $N$-manifold, $N \geq 2$.  We say that $X$ is \emph{asymptotically conical (AC)} with \emph{rate $\nu < 0$} if there exists a compact set $K \subset X$, a constant $R > 0$, a compact connected Riemannian $(N-1)$-manifold $(\Sigma^{N-1}, g_\Sigma)$, and a diffeomorphism
$$\Psi \colon (R,\infty) \times \Sigma \to X - K$$
such that
\begin{equation*}
\left| \nabla^k_{g_C}(\Psi^*h_X - g_{C}) \right|_{g_C} = O(r^{\nu - k}) \ \ \ \text{as } r \to \infty
\end{equation*}
for all $k \in \Z_{\geq 0}$, where $(C, g_C) = ( (R,\infty) \times \Sigma, dr^2 + r^2g_\Sigma)$ is the metric cone of $\Sigma$.  We then call $\Sigma$ the \emph{asymptotic link} of $X$.
\end{defn}

\indent Let $h_X^{-1}$ and $g_C^{-1}$ denote the dual metrics on $1$-forms, and let $\mathrm{vol}_X$ and $\mathrm{vol}_C$ denote the corresponding volume forms.  It can be shown that the AC condition implies the estimates
\begin{align}
\left| \nabla^k_{g_C}(\Psi^*h_X^{-1} - g_{C}^{-1}) \right|_{g_C} & = O(r^{\nu - k}) \label{eq:dual-metric-bound} \\
\left| \nabla^k_{g_C}(\Psi^*\mathrm{vol}_X - \mathrm{vol}_C) \right|_{g_C} & = O(r^{\nu - k}). \notag
\end{align}
See \cite[$\S$3.1]{conlon2011construction} for a proof. \\

\indent We now set notation.  Let $r \colon C \to (R, \infty)$ denote the natural radius function, and let $\rho = r \circ \Psi^{-1} \colon X - K \to (R,\infty)$
denote the corresponding function on $X - K$.  For $a > R$, we set
\begin{align*}
\Sigma_a & = \{a\} \times \Sigma = \{p \in C \colon r(p) = a\} & S_a & = \{ x \in X \colon \rho(x) = a\}.
\end{align*}
Let $\iota_a \colon \Sigma \to \Sigma_a$ denote the identification $\iota_a(p) = (a,p)$, and let $\Psi_a := \left.\Psi\right|_{\Sigma_a} \colon \Sigma_a \to S_a$. We equip each $\Sigma_a \subset C$ with its induced metric $g_a$, and similarly equip each $S_a \subset X$ with its induced metric $h_a$.  The following diagram summarizes the manifolds under consideration, where vertical arrows denote the obvious inclusions.
$$\begin{tikzcd}
& {(C,g_C) = (R,\infty) \times \Sigma} \arrow[r, "\Psi"] & (X - K, h_X)         \\
(\Sigma, g_\Sigma) \arrow[r, "\iota_r"] &  (\Sigma_r, g_r) = \{r\} \times \Sigma \arrow[u] \arrow[r, "\Psi_r"']        & (S_r, h_r) \arrow[u]
\end{tikzcd}$$
\indent To streamline notation, we will often suppress the identification maps $\iota_r$ (as well as their pushforwards and pullbacks) from the notation.  Let $\{e_1, \ldots, e_{N-1}\}$ be a $g_\Sigma$-orthonormal basis of $T_p\Sigma$, and let $\{e^1, \ldots, e^{N-1}\}$ be the dual $g_\Sigma^{-1}$-orthonormal basis of $T_p^*\Sigma$.  Note that $\{  \frac{1}{r} e_j\}$ is a $g_r$-orthonormal basis of $T_p\Sigma_r$, and $\{re^j\}$ is a $g_r^{-1}$-orthonormal basis of $T_p^*\Sigma_r$.  In particular,
$$\mathrm{vol}_{\Sigma_r} = r^{N-1}\,e^1 \wedge \cdots \wedge e^{N-1} = r^{N-1}\,\mathrm{vol}_\Sigma.$$
In this way, we are able to relate the metric $g_r$ to $g_\Sigma$.  However, we also will need to relate $\Psi_r^*h_r$ to $g_\Sigma$.   This is accomplished by the following lemma.

\begin{lem} \label{lem:Psi-r-Bounds} We have
\begin{align}
\mathrm{vol}_{\Psi_r^*h_r} & = O\!\left(r^{N-1}\right) \mathrm{vol}_\Sigma \label{eq:Vol-Bound} \\
\Psi_r^*h_r & = O\!\left(r^{2}\right) g_\Sigma \label{eq:Metric-Bound} \\
\Psi_r^*h_r^{-1} & = O\!\left(r^{-2}\right) g_\Sigma^{-1} \label{eq:Dual-Metric-Bound} 
\end{align}
\end{lem}

\begin{proof} Let us fix $p \in \Sigma$ and suppress the identification map $\iota_r \colon \Sigma \to \Sigma_r$ throughout.  Since $\Psi_r^*h_r$ is a symmetric bilinear form, the spectral theorem yields a $g_\Sigma$-orthonormal basis $\{e_1, \ldots, e_{N-1}\}$ of $T_p\Sigma$ with respect to which $\Psi_r^*h_r$ is diagonal.  That is, there exist $\tau_j = \tau_j(r)$ for which
\begin{align*}
\Psi_r^*h_r = r^2 \sum \tau_j^2 (e^j)^2
\end{align*}
where in this sum (and all those throughout the proof) the index range is $j = 1, \ldots, N-1$.  In particular,
\begin{align*}
\Psi_r^*h_r^{-1} & = \frac{1}{r^2} \sum \frac{1}{\tau_j^2} e_j^2 &  \mathrm{vol}_{\Psi_r^*h_r} & = r^{N-1} \tau_1 \cdots \tau_{N-1} \,e^1 \wedge \cdots \wedge e^{N-1}.
\end{align*}
\indent Now, since $h_X$ is asymptotically conical, we have $\left| \Psi^*h_X - g_{C} \right|_{g_C} = O(r^{\nu})$. Thus, for each $j = 1, \ldots, N-1$,
\begin{align*}
|\tau_j^2 - 1| \leq \left(\sum \left|\tau_j^2 - 1\right|^2\right)^{1/2} & = \left(\sum \left| \tau_j^2 (\Psi^*h_X)\!\left( \frac{1}{\tau_j r}e_j, \frac{1}{\tau_j r} e_j \right) - g_{C}\!\left(\frac{1}{r}e_j, \frac{1}{r}e_j\right) \right|^2\right)^{1/2} = O(r^{\nu})
\end{align*}
and so
$$\tau_j^2 = 1 + O(r^{\nu}) = O(1).$$
Therefore,
$$r^{N-1} \tau_1 \cdots \tau_{N-1} = O( r^{N-1} )$$
proves (\ref{eq:Vol-Bound}) and
$$\Psi_r^*h_r = r^2 \sum \tau_j^2 (e^j)^2 = O(r^2) \sum (e^j)^2 = O(r^{2}) g_\Sigma$$
proves (\ref{eq:Metric-Bound}).  Next, since $h_X$ is asymptotically conical, we have $\left| \Psi^*h_X^{-1} - g_{C}^{-1} \right|_{g_C} = O(r^{\nu})$ from (\ref{eq:dual-metric-bound}). Thus, for each $j = 1, \ldots, N-1$,
\begin{align*}
\left|\frac{1}{\tau_j^2} - 1\right| \leq \left(\sum \left|\frac{1}{\tau_j^2} - 1\right|^2\right)^{1/2} & = \left(\sum \left| \frac{1}{\tau_j^2} (\Psi^*h_X^{-1})\!\left( \tau_j r e^j, \tau_jr e^j \right) - g_{C}^{-1}\!\left(re^j, re^j \right) \right|^2\right)^{1/2}  = O(r^{\nu})
\end{align*}
and so
$$\frac{1}{\tau_j^2}  = 1 + O(r^{\nu}) = O(1).$$
Therefore,
$$\Psi_r^*h_r^{-1} = \frac{1}{r^2} \sum \frac{1}{\tau_j^2} e_j^2 =  O(r^{-2}) \sum e_j^2 = O(r^{-2}) g_\Sigma^{-1},$$
which proves (\ref{eq:Dual-Metric-Bound}).
\end{proof}

\indent Before continuing, we establish a few more estimates.

\begin{lem} \label{lem:PullbackBounds} Let $E \to X$ be a Hermitian vector bundle.
\begin{enumerate}[(a)]
\item Let $\gamma \in \Lambda^k(T^*X) \otimes E$.  Then $\left| \Psi_r^* \gamma \right|_{g_\Sigma} \leq r^k \left| \Psi^*\gamma \right|_{g_C}$, suppressing pullbacks via the inclusions $S_r \hookrightarrow X - K$ from the notation.
\item Let $\gamma \in \Omega^k(X)$.  Then $\left| \Psi^*\gamma \right|_{g_C} =O(1) \left|\gamma\right|_{h_X}$.
\end{enumerate}
\end{lem}

\begin{proof} (a) For $\theta \in \Lambda^k(T^*C) \otimes \Psi^*E$, we may write $\theta = \theta_{H} + dr \wedge \theta_{V}$ where $\partial_r\,\lrcorner\,\theta_H = 0$.  Then
$$\left| \iota_r^* \theta\right|_{g_\Sigma} = r^k \left| \theta_H \right|_{g_C} \leq r^k \left| \theta \right|_{g_C}.$$
Taking $\theta = \Psi^*\gamma$ gives the result. \\

\indent (b) Let $\{e_1, \ldots, e_N\}$ be a $g_C$-orthonormal basis of $T_{(r,p)}C$ with respect to which $\Psi^*h_X$ is diagonal.  This means that there exist $\mu_j = \mu_j(r)$ for which
\begin{align*}
g_C & = (e^1)^2 + \cdots + (e^N)^2, & \Psi^*h_X & = \mu_1^2 (e^1)^2 + \cdots + \mu_N^2(e^N)^2.
\end{align*}
Repeating the argument in the proof of Lemma \ref{lem:Psi-r-Bounds} shows that each $\mu_j = O(1)$.  Now, in terms of the basis $\{e_j\}$, we may write $\Psi^*\gamma = \sum_{I} A_I e^I$, say, where we sum over multi-indices $I = (i_1, \ldots, i_k)$ and write $e^I = e^{i_1} \wedge \cdots \wedge e^{i_k}$.  Therefore,
\begin{align*}
\left|\Psi^*\gamma\right|_{g_C}^2 = \sum_I A_I^2 = \sum_I \frac{1}{\mu_I^2}A_I^2 \mu_I^2 = O(1) \sum_I \frac{1}{\mu_I^2}A_I^2 = O(1) \left| \Psi^*\gamma\right|^2_{\Psi^*h_X} = O(1) \left|\gamma\right|^2_{h_X}.
\end{align*}
\end{proof}

\subsubsection{Asymptotically conical hyperk\"{a}hler manifolds}

\begin{defn} Let $(X^{4n+4}, h_X, (I,J,K), (\omega_I, \omega_J, \omega_K))$ be a hyperk\"{a}hler $(4n+4)$-manifold, $n \geq 1$.  We say that $X$ is \emph{asymptotically conical (AC)} with \emph{rate $\nu < 0$} if there exists a compact set $K \subset X$, a constant $R > 0$, a compact connected $3$-Sasakian $(4n+3)$-manifold $(\Sigma^{4n+3}, g_\Sigma)$, and a diffeomorphism
$$\Psi \colon (R,\infty) \times \Sigma \to X - K$$
such that for each $L \in \{I,J,K\}$ and $k \in \Z_{\geq 0}$,
\begin{align*}
\left| \nabla^k_{g_C}(\Psi^*\omega_L - \beta_L) \right|_{g_C} & = O(r^{\nu - k}) \ \ \ \text{as } r \to \infty \\
\left| \nabla^k_{g_C}(\Psi^*L - L_{C}) \right|_{g_C} & = O(r^{\nu - k}) \ \ \ \text{as } r \to \infty
\end{align*}
where $C = \mathrm{C}(\Sigma) = (R, \infty) \times \Sigma$ is the metric cone over $\Sigma$, and $(g_C$, $(I_C, J_C, K_C)$, $(\beta_I, \beta_J, \beta_K))$ is the natural conical hyperk\"{a}hler structure on $C$.  We call $\Sigma$ the \emph{asymptotic link} of $X$.
\end{defn}

\indent It is well-known that AC hyperk\"{a}hler $(4n+4)$-manifolds of rate $\nu$ are AC Riemannian $(4n+4)$-manifolds of rate $\nu$, meaning that we have the estimate
\begin{equation*}
\left| \nabla^k_{g_C}(\Psi^*h_X - g_{C}) \right|_{g_C} = O(r^{\nu - k}) \ \ \ \text{as } r \to \infty
\end{equation*}
See \cite[$\S$3.3]{conlon2011construction} for a proof.  From now on, we assume that $X$ is an AC hyperk\"{a}hler manifold of rate $\nu$. \\

\indent As in $\S$\ref{subsub:Proof-Cpt-Lewis}, for $\lambda = (\lambda_1, \lambda_2, \lambda_3) \in \R^3$ satisfying both $\lambda_1 + \lambda_2 + \lambda_3 = 0$ and $\lambda_1 - 2\lambda_2 - 2\lambda_3 > 0$, we define a $4n$-form $\Phi_\lambda \in \Omega^{4n}(X)$ by
$$\Phi_\lambda = \textstyle \frac{1}{(2n)!}(\lambda_1 \omega_I^{2n} + \lambda_2 \omega_J^{2n} + \lambda_3 \omega_K^{2n}).$$
The following algebraic fact will be used repeatedly in the sequel.  To state it, recall the decomposition
$$\Lambda^2(T^*X) = U_I \oplus U_J \oplus U_K \oplus W \oplus \R\omega_I \oplus \R\omega_J \oplus \R\omega_J$$
of Proposition \ref{prop:HyperHermitianDecomp}.  For a $2$-form $F \in \Lambda^2(T^*X)$, we let $F_I$ and $F_W$ denote its components in $U_I$ and $W$, respectively.

\begin{lem} \label{lem:Hodge-Star-F} Let $F \in \Lambda^2(T^*X)$.  Then:
\begin{enumerate}[(a)]
\item $F_W \wedge \Phi_\lambda = 0$.
\item $F_I \wedge \Phi_\lambda = (-\lambda_1 + \lambda_2 + \lambda_3) \ast\!F_I$.
\end{enumerate}
\end{lem}

\begin{proof} (a) Note that $F_W \wedge \frac{1}{(2n)!}\omega_L^{2n} = -\ast\! F_W$ for each $L \in \{I,J,K\}$.  Therefore,
\begin{align*}
F_W \wedge \Phi_\lambda & = -(\lambda_1 + \lambda_2 + \lambda_3) \ast\!F_W = 0.
\end{align*}
\indent (b) Note that $F_I \wedge \omega_I^{2n} = -\ast\! F_I$, whereas $F_I \wedge \omega_J^{2n} = F_K \wedge \omega_J^{2n} = \ast F_I$.  Therefore,
\begin{align*}
F_I \wedge \Phi_\lambda & = \textstyle \frac{1}{(2n)!}\left( \lambda_1 \,F_I \wedge \omega_I^{2n} + \lambda_2 \,F_I \wedge \omega_J^{2n} + \lambda_3 \,F_I \wedge \omega_K^{2n} \right)  \\
& = (-\lambda_1 + \lambda_2 + \lambda_3) \ast\!F_I.
\end{align*}
\end{proof}

\subsubsection{The Chern-Simons functional}

\indent \indent Let $X^{4n+4}$ be an AC hyperk\"{a}hler manifold of rate $\nu < 0$.   To study gauge-theoretic objects over $X$, we consider a class of bundles that respects the asymptotic geometry of $X$.

\begin{defn} A principal bundle $P \to X^{4n+4}$ is \emph{asymptotically framed} if there exists a principal bundle $Q \to \Sigma$ (called an \emph{asymptotic framing}) such that $\Psi^*P \cong \pi^*Q$, where $\pi \colon C = (R,\infty) \times \Sigma \to \Sigma$ is the projection onto the second factor.
$$\begin{tikzcd}
Q \arrow[d] & \pi^*Q \cong \Psi^*P \arrow[d] \arrow[l] \arrow[r] & P \arrow[d] \\
\Sigma      & C \arrow[r, "\Psi"'] \arrow[l, "\pi"]          & X - K      
\end{tikzcd}$$
Fix a connection $A_\infty$ on $Q \to \Sigma$.  A connection $A$ on $P \to X$ is called \emph{asymptotically conical (AC)} of \emph{rate $\alpha$} \emph{with respect to $A_\infty$} provided that
$$\left| \nabla^k_{g_C}(\Psi^*A - \pi^*A_\infty) \right|_{g_C} = O(r^{\alpha-k}) \ \ \text{ as } r \to \infty$$
for all $k \geq 0$.  We let
\begin{align*}
\mathcal{A}_\alpha(P, A_\infty) & = \left\{\text{AC connections on }P \text{ of rate }\alpha \text{ with respect to }A_\infty \right\} \\
\mathcal{M}_\alpha(P, A_\infty) & = \left\{A \in \mathcal{A}_\alpha(P, A_\infty) \colon A \text{ is an }\Sp(n+1)\text{-instanton} \right\}\!. \\
\end{align*}
\end{defn} 

\indent We now seek AC analogues of Lemma \ref{lem:Phi-lambda-topological} and Corollary \ref{cor:Primitive-Topological}, which as stated apply only to \emph{compact} hyperk\"{a}hler manifolds.  To this end, by using the truncations $X_r = \{x \in X \colon \rho(x) \leq r\}$ and their boundaries $S_r = \{x \in X \colon \rho(x) = r\}$, we introduce the following functional.

\begin{defn} Let $P \to X^{4n+4}$ be an asymptotically framed principal bundle, and let $A_0$ be an $\Sp(n+1)$-instanton on $P$.  For a connection $A$ on $P$, its \emph{Chern-Simons functional} is
\begin{align*}
\mathrm{CS}_\lambda[A] \colon [R, \infty) & \to [0, \infty) \\
\mathrm{CS}_\lambda[A](r) & := \int_{S_r} \mathrm{tr}\!\left( F_A \wedge (A - A_0) - \frac{1}{3}(A - A_0)^3 \right) \wedge \Phi_\lambda.
\end{align*}
\end{defn}

\indent The key properties of the CS functional are summarized in the following proposition.  To state it, recall once more that if $A$ is a connection on $P \to X$, then its curvature $2$-form $F := F_A$ decomposes as
$$F = F_I + F_J + F_K + F_W + \mu_I \omega_I + \mu_J \omega_J + \mu_K \omega_K.$$
In particular, $A$ is $I$-primitive HYM if and only if $F = F_I + F_W$, and $A$ is an $\Sp(n+1)$-instanton if and only if $F_I = 0$.

\begin{prop} \label{prop:CS-Identity} Let $A$ be a connection on $P$.
\begin{enumerate}[(a)]
\item We have
$$\mathrm{CS}_\lambda[A](r) = \sum_{L \in \{I,J,K\}}  \int_{X_r} \left(a_L | F_L |^2 +   b_L | \mu_L \omega_L |^2\right) \vol_{X_r},$$
where the constants $a_L$ and $b_L$ are as in (\ref{eq:AB-constants}). In particular, $\mathrm{CS}_\lambda[A]$ is independent of the choice of $\Sp(n+1)$-instanton $A_0$.  
\item If $A$ is $I$-primitive HYM, then
$$\mathrm{CS}_\lambda[A](r) = a_I \int_{X_r} |F_I|^2\,\mathrm{vol}_{X_r}.$$
Thus, $A$ is an $\Sp(n+1)$-instanton if and only if $\mathrm{CS}_\lambda[A](r) = 0$ for all sufficiently large $r \geq R$.
\end{enumerate}
\end{prop}

\begin{proof} (a) Let $A$ be a connection on $P$, let $A_0$ be an $\Sp(n+1)$-instanton on $P$, and let $F$ and $F_0$ be their respective curvature forms.  It is well-known that
$$d \left[ \mathrm{tr}\!\left( (F + F_0) \wedge (A - A_0)  - \frac{1}{3} (A-A_0)^3 \right) \right] = \mathrm{tr}(F^2) - \mathrm{tr}(F_0^2).$$
Wedging with $\Phi_\lambda$, integrating over $X_r$, and then applying Stokes' Theorem gives:
\begin{align*}
\int_{S_r}  \mathrm{tr}\!\left( (F + F_0) \wedge (A - A_0)  - \frac{1}{3} (A - A_0)^3  \right) \wedge \Phi_\lambda & = \int_{X_r} \mathrm{tr}(F^2) \wedge \Phi_\lambda - \int_{X_r} \mathrm{tr}(F_0^2) \wedge \Phi_\lambda \\
& = \sum_{L \in \{I,J,K\}} \int_{X_r} \left(a_L | F_L |^2 +   b_L | \mu_L \omega_L |^2\right) \vol_{X_r},
\end{align*}
where in the last step we applied Lemma \ref{lem:Phi-lambda-topological} to both integrals.  To simplify the left side, note that since $A_0$ is an $\Sp(n+1)$-instanton, we have $F_0 \in \Gamma( W \otimes \mathrm{ad}_P)$.  Therefore, Lemma \ref{lem:Hodge-Star-F}(a) gives $F_0 \wedge \Phi_\lambda = 0$, and so
$$\mathrm{tr}(F_0 \wedge (A - A_0)) \wedge \Phi_\lambda = \mathrm{tr}(F_0 \wedge (A - A_0) \wedge \Phi_\lambda) = 0.$$
Therefore,
$$\int_{S_r}  \mathrm{tr}\!\left( F \wedge (A - A_0)  - \frac{1}{3} (A - A_0)^3 \right) \wedge \Phi_\lambda = \sum_{L \in \{I,J,K\}} \int_{X_r} \left(a_L | F_L |^2 +   b_L | \mu_L \omega_L |^2\right) \vol_{X_r}.$$
\indent (b) This follows immediately from part (a).
\end{proof}

\subsubsection{The asymptotically conical Lewis theorem} \label{subsub:ACLewis}

\indent \indent We are finally in a position to state and prove the AC analogue of Theorem \ref{thm:Compact-Lewis}.  

\begin{thm} \label{thm:ACLewis} Let $P \to X^{4n+4}$ be an asymptotically framed principal $G$-bundle over an asymptotically conical hyperk\"{a}hler manifold $X$ of rate $\nu < 0$, where $G$ is a compact Lie group.  Let $\Sigma^{4n+3}$ be the $3$-Sasakian asymptotic link of $X$, let $Q \to \Sigma$ be an asymptotic framing, and fix a tri-contact instanton $A_\infty$ on $Q \to \Sigma$. \\
\indent Suppose that $\nu \leq - 2 - \frac{4}{3}n$, and that $\mathcal{M}_\alpha(P, A_\infty)$ is non-empty for some $\alpha < \nu + 1$.  If $A \in \mathcal{A}_\alpha(P, A_\infty)$ is an asymptotically conical $I$-primitive HYM connection satisfying
\begin{align}
\left| \partial_r \,\lrcorner\,\Psi^*F_I \right|_{g_C} & = O(r^{\beta}) \ \ \text{ where } \beta < -4n - \nu - 4, \label{eq:beta-bound}
\end{align}
where $F = F_A$ is the curvature of $A$, then $A$ is an $\Sp(n+1)$-instanton.
\end{thm}

\begin{proof} Let $A \in \mathcal{A}_\alpha(P, A_\infty)$ be an asymptotically conical $I$-primitive HYM connection satisfying (\ref{eq:beta-bound}).  Since  $\mathcal{M}_\alpha(P, A_\infty)$ is non-empty, there exists an $\Sp(n+1)$-instanton $A_0$ satisfying
$$|\Psi^*A_0 - \pi^*A_\infty|_{g_C} = O(r^{\alpha}).$$
Consequently,
$$\left|\Psi^*(A-A_0)\right|_{g_C} \leq \left|\Psi^*A - \pi^*A_\infty\right|_{g_C} + \left|\Psi^*A_0 - \pi^*A_\infty\right|_{g_C} = O(r^{\alpha}).$$
In particular, suppressing pullbacks via the inclusions $S_r \hookrightarrow X - K$ from the notation, Lemma \ref{lem:PullbackBounds}(a) gives
\begin{align}
\left|\Psi_r^*(A-A_0)\right|_{g_\Sigma} & \leq r\left|\Psi^*(A-A_0)\right|_{g_C} = O(r^{\alpha+1})  \label{eq:Alpha-Estimate} \\
\left|  \partial_r\,\lrcorner\,\Psi_r^*F_I \right|_{g_\Sigma} & \leq r \left|  \partial_r\,\lrcorner\,\Psi^*F_I \right|_{g_C} = O(r^{\beta+1}).\label{eq:Beta-Estimate}
\end{align}
\indent Now, following the technique of \cite{papoulias2022spin}, we split
\begin{align*}
 \mathrm{CS}_\lambda[A](r) & = \frac{1}{2} \int_{S_r} \mathrm{tr}\!\left( F \wedge (A - A_0) - \frac{1}{3}(A - A_0)^3 \right) \wedge \Phi_\lambda \leq \left| I_1(r) \right|  +  \left| I_2(r) \right|\!,
\end{align*}
where
\begin{align*}
I_1(r) & = \frac{1}{2} \int_{S_r} \mathrm{tr}\!\left(F \wedge (A - A_0) \right) \wedge \Phi_\lambda, & I_2(r) & = \frac{1}{6}\int_{S_r} \mathrm{tr}\!\left( (A - A_0)^3 \right) \wedge \Phi_\lambda.
\end{align*}
We aim to establish the following two estimates:
\begin{align}
\left| I_1(r) \right| & = O\!\left(r^{4n+3+\alpha+\beta}\right) \label{eq:I1-bound}  \\
\left| I_2(r) \right| & = O\!\left(r^{4n+3+3\alpha}\right)\!. \label{eq:I2-bound}
\end{align}
Note that these two bounds imply the result.  To see this, observe that our assumptions imply that $4n + 3 + \alpha + \beta < 0$ and $4n + 3 + 3\alpha < 0$.  Therefore, the Monotone Convergence Theorem, followed by Proposition \ref{prop:CS-Identity}(b), followed by the above estimates gives
$$a_I \int_X |F_I|^2\,\mathrm{vol} = \lim_{r \to \infty} a_I \int_{X_r} |F_I|^2\,\mathrm{vol}_{X_r} = \lim_{r \to \infty} \mathrm{CS}_\lambda[A](r) \leq \lim_{r \to \infty} ( |I_1(r)| + |I_2(r)| ) = 0,$$
whence $F_I = 0$, and $A$ is an $\Sp(n+1)$-instanton. \\
\indent We begin by proving (\ref{eq:I1-bound}).  Set $\tau := -\lambda_1 + \lambda_2 + \lambda_3$.  Since $A$ is $I$-primitive HYM, Lemma \ref{lem:Hodge-Star-F} implies that $F \wedge \Phi_\lambda = \tau \ast_X\! F_I$, and hence
$$\left| I_1(r) \right| = \left|\frac{1}{2} \int_{S_r} \mathrm{tr}\!\left(F \wedge (A - A_0) \right) \wedge \Phi_\lambda \right| = \left| \frac{\tau}{2} \int_{S_r} \mathrm{tr}( (A-A_0) \wedge \ast_X F_I)\right|.$$
Now, setting $\widehat{\partial}_\rho := \frac{1}{\left|\partial_\rho\right|}\partial_\rho$, we have
\begin{align*}
\mathrm{tr}((A - A_0) \wedge \ast_X F_I) & = \mathrm{tr}( (A-A_0) \wedge \ast_{S_r} (\widehat{\partial}_\rho\,\lrcorner\,F_I) )  = -\frac{1}{|\partial_\rho|} \left\langle A - A_0,\, \partial_\rho \,\lrcorner\, F_I \right\rangle_{h_r} \mathrm{vol}_{S_r}.
\end{align*}
Moreover, since $X$ is asymptotically conical of rate $\nu$, we have
$$\left| h_X(\partial_\rho, \partial_\rho) - 1 \right| = \left| (\Psi^*h_X)(\partial_r, \partial_r) - g_C(\partial_r, \partial_r) \right| = O(r^{\nu}),$$
which implies that $\frac{1}{|\partial_\rho|} = O(1)$.  Therefore,
\begin{align*}
\left| I_1(r) \right| & = \left| \frac{\tau}{2} \int_{S_r} \frac{1}{|\partial_\rho|} \left\langle A - A_0,\, \partial_\rho \,\lrcorner\, F_I \right\rangle_{h_r} \mathrm{vol}_{S_r} \right| \\
& = O(1) \left| \int_{\Sigma} \Psi_r^* \!\left( \left\langle A - A_0,\, \partial_\rho \,\lrcorner\, F_I \right\rangle_{h_r} \mathrm{vol}_{S_r} \right) \right| \\
& = O(1) \left| \int_{\Sigma} \left\langle \Psi_r^*(A - A_0),\,  \Psi_r^*(\partial_\rho \,\lrcorner\, F_I) \right\rangle_{\Psi_r^*h_r} \mathrm{vol}_{\Psi_r^*h_r} \right| \\
& = O(r^{4n+1}) \left|  \int_{\Sigma}  \left\langle \Psi_r^*(A - A_0),\,  \partial_r\,\lrcorner\, \Psi_r^*F_I \right\rangle_{g_\Sigma} \mathrm{vol}_{\Sigma}  \right|
\end{align*}
where in the last step we used (\ref{eq:Vol-Bound}) and (\ref{eq:Dual-Metric-Bound}).  Finally, from the bounds (\ref{eq:Alpha-Estimate}) and (\ref{eq:Beta-Estimate}), we obtain
$$\left| \left\langle \Psi_r^*(A - A_0),\,  \partial_r\,\lrcorner\, \Psi_r^*F_I \right\rangle_{g_\Sigma} \right|  \leq \left| \Psi_r^*(A - A_0) \right|_{g_\Sigma} \left|  \partial_r\,\lrcorner\,\Psi_r^*F_I \right|_{g_\Sigma} = O(r^{\alpha + \beta +2}),$$
which gives (\ref{eq:I1-bound}). \\
\indent It remains to prove  (\ref{eq:I2-bound}).  For this, we estimate
\begin{align*}
\left|I_2(r)\right| = \left|\frac{1}{6}\int_{S_r} \mathrm{tr}\!\left[ (A - A_0)^3  \wedge \Phi_\lambda \right]\right| & = \left| \frac{1}{6}\int_{\Sigma} \mathrm{tr}\!\left[ (\Psi_r^*(A - A_0))^3  \wedge \Psi_r^*\Phi_\lambda \right] \right| \\
& \leq \frac{1}{6} \int_\Sigma \left| \Psi_r^*(A - A_0) \right|_{g_\Sigma}^3 \left| \Psi_r^*\Phi_\lambda \right|_{g_\Sigma}\,\vol_\Sigma.
\end{align*}
Now, the bound (\ref{eq:Alpha-Estimate})  gives $\left| \Psi_r^*(A - A_0) \right|_{g_\Sigma}^3 = O( r^{3\alpha + 3})$.  Moreover, Lemma \ref{lem:PullbackBounds}(a) and (b) give
$$\left| \Psi_r^*\Phi_\lambda \right|_{g_\Sigma} \leq r^{4n} \left| \Psi^*\Phi_\lambda \right|_{g_C} = O(r^{4n}) \left| \Phi_\lambda\right|_{h_X}.$$
Finally, since $\nabla \Phi_\lambda = 0$, we have $\left|\Phi_\lambda\right|_{h_X} = \text{constant}$.  Thus, we conclude that $\left|I_2(r)\right| = O(r^{4n+3\alpha+3})$ as claimed.
\end{proof}

\bibliographystyle{plain}
\bibliography{Instanton-Ref}
\Addresses

\end{document}